\documentclass[a4paper,11pt]{article} 				

\usepackage{amsmath,amsfonts,amsthm,amssymb} 			
\usepackage{geometry}							
\usepackage{fullpage}
\usepackage{comment}
\usepackage{color}

\theoremstyle{plain}										
\newtheorem{thm}{Theorem}[section]
\newtheorem{lem}[thm]{Lemma}
\newtheorem{prop}[thm]{Proposition}

\theoremstyle{definition}
\newtheorem{defn}{Definition}[section]

\newtheorem*{app}{Application}
\newtheorem{assumption}{Assumption}[section]

\theoremstyle{remark}
\newtheorem{rem}{Remark}[section]

\usepackage{graphicx}									
\usepackage{caption}
\usepackage{subcaption}
\usepackage{array}
													
\numberwithin{equation}{section} 							
\numberwithin{figure}{section} 								
\numberwithin{table}{section} 								

\usepackage{todonotes}
\newcommand\W[1]{{\color{black}{#1}}} 

\title{\normalfont \Large Two-Stroke Relaxation Oscillators}
\author{\large Samuel Jelbart \& Martin Wechselberger\\ School of Mathematics \& Statistics, University of Sydney,\\ Camperdown, NSW 2006, Australia}

\begin{document}
\maketitle


\begin{abstract}
{\em Two-stroke relaxation oscillations} consist of two distinct phases per cycle -- one slow and one fast -- which distinguishes them from the well-known van der Pol-type `four-stroke' relaxation oscillations. \W{This type of oscillation} can be found in singular perturbation problems in {\em non-standard form}, where the slow-fast timescale splitting is not necessarily reflected in a slow-fast variable splitting. \textcolor{black}{The existing literature on such non-standard problems has developed primarily through applications -- we compliment this by providing a general} framework for the application of geometric singular perturbation theory \W{in this non-standard setting} \textcolor{black}{and illustrate its applicability by proving existence and uniqueness results on a general class of two-stroke relaxation oscillators. We apply this non-standard geometric singular perturbation toolbox to a collection of examples arising in the dynamics of nonlinear transistors and models for mechanical oscillators with friction.}
\end{abstract}


\section{Introduction}\label{IntroSec}

The term {\em relaxation oscillation} was coined by Balthasar van der Pol in the 1920s \cite{vdP1920,vdP1926} to distinguish nonlinear  from harmonic oscillations observed in electronic circuits. He also derived a \textcolor{black}{prototypical} mathematical model, the now well-known {\em van der Pol (vdP) oscillator model},
\begin{equation}\label{VdPType}
\ddot{x}+r_{vdp}(x)\dot{x}+x=0
\end{equation}
with nonlinear differential `resistance' 
\begin{equation}
r_{vdp}(x)= \mu(x^2-1)\,,
\end{equation}
where $\mu\ge 0$ denotes the main system parameter that measures the ratio of the two characteristic timescales of the electronic circuit model under study. The overdot denotes differentiation with respect to time $\tau$. This oscillator model can be recast as a dynamical system (in Li\'{e}nard form),
\begin{align}\label{VdPDyn}
\begin{array}{lcl}
\dot{\bar y}=-x, \\
\dot{x}=\bar y-R_{vdp}(x),
\end{array}
\end{align}
with 
\begin{equation}\label{VdPChar}
R_{vdp}(x)=\mu\bigg(\frac{x^3}{3}-x\bigg)\,,
\end{equation}
i.e.~$R'_{vdp}(x)=r_{vdp}(x)$. The function $R_{vdp}(x)$ denotes the (non-linear) {\em characteristic} of the oscillator model which takes the form of a cubic. 
\begin{rem}\label{EMCharRem} 
The \textcolor{black}{(dimensionless)} {\em current-voltage (I-V) characteristic} \eqref{VdPChar} derives from the presence of a tunnel diode (an active, non-linear element). The requirement that the characteristic has negative slope $R'_{vdp}(x)=r_{vdp}(x)<0$ in parts of the phase space, here for $|x|<1$, gives rise to an effective negative resistance $r_{vdp}(x)$ that allows energy to be pumped back (relaxed) into the system which is necessary for oscillatory behaviour.
\end{rem}

We are interested in the relaxation case $\mu\gg1$ of the vdP oscillator model \eqref{VdPDyn} and define a new variable $y= \bar y/\mu$, change to a new (fast) timescale $dt=\mu d\tau$, and set $\epsilon:=1/\mu^2\ll 1$ to obtain 
\begin{align}\label{VdPDyn2}
\begin{array}{lcl}
x'=y+x-\frac{x^3}{3}, \\
y'=-\epsilon x,
\end{array}
\end{align}

where the dash notation denotes differentiation with respect to the new (fast) time $t$.
\begin{figure}[t]
\captionsetup{format=plain}
\centering
\begin{subfigure}{.5\textwidth}
  \centering
  \includegraphics[trim={2.5cm 8.5cm 0 8.5cm},scale=0.5]{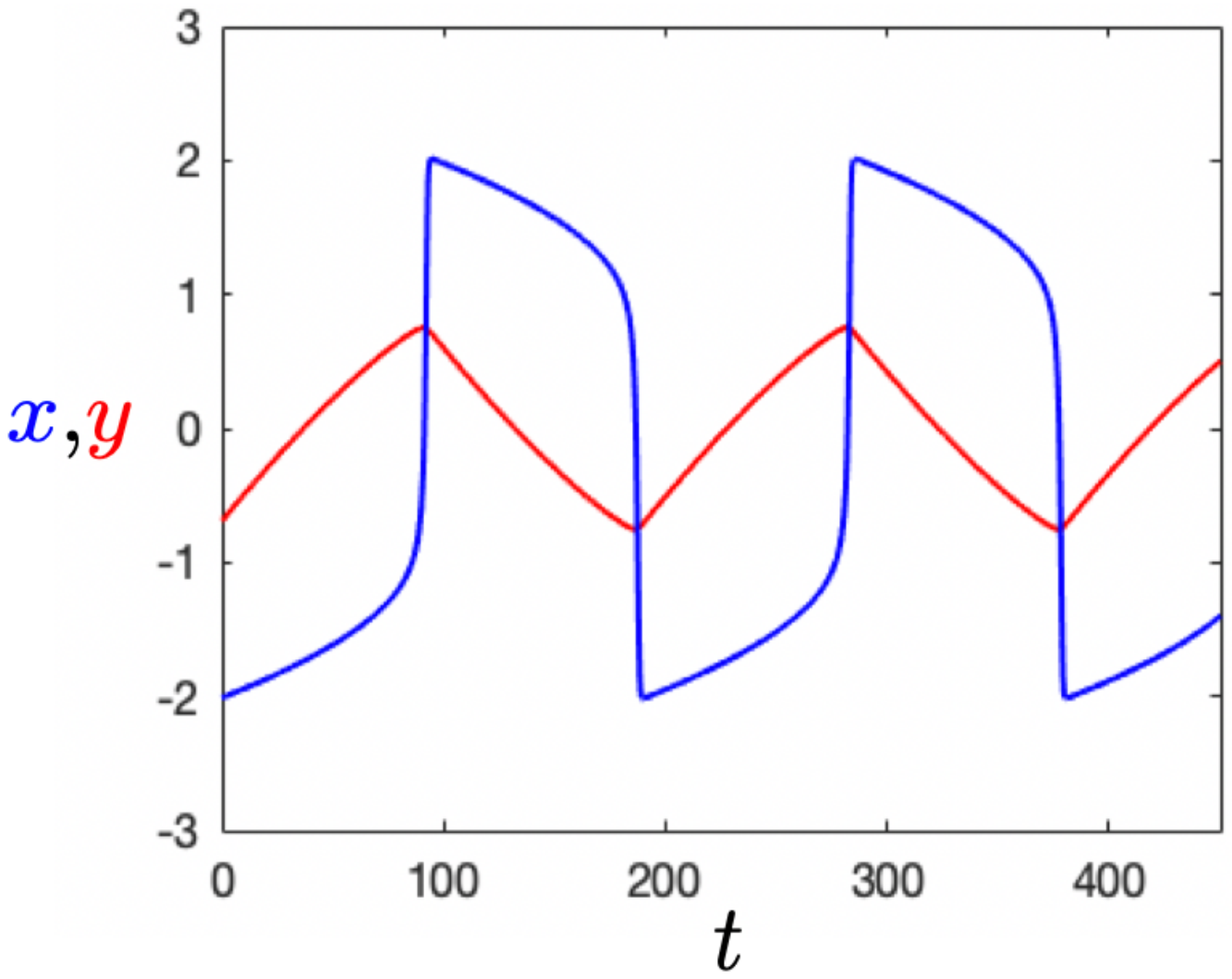}
  \caption{Time traces.}
  \label{vdPTraceFig}
\end{subfigure}%
\begin{subfigure}{.5\textwidth}
  \centering
  \includegraphics[trim={8cm 4.2cm 0 4.2cm},scale=0.5]{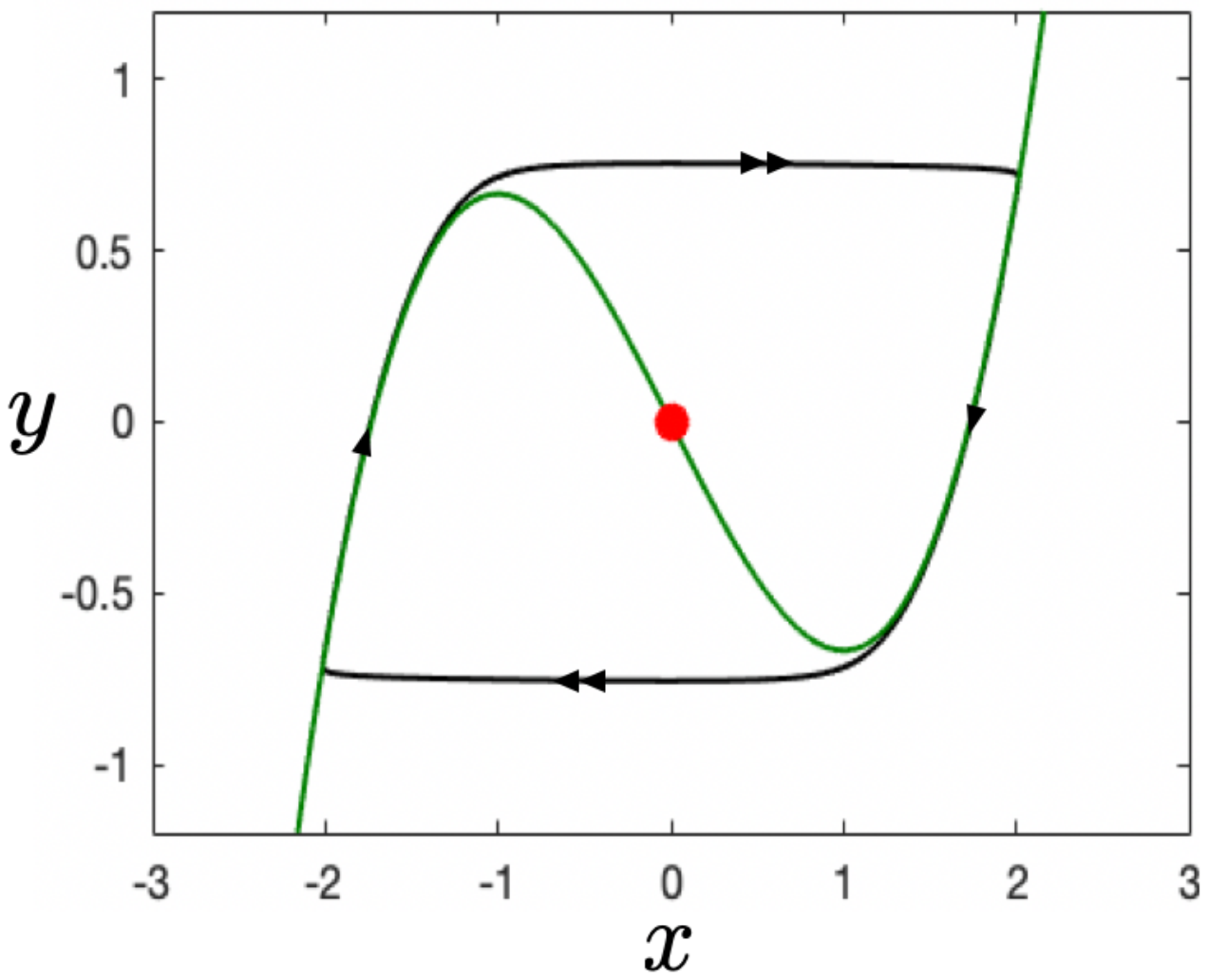}
  \caption{Phase space and limit cycle.}
  \label{vdPPhaseFig}
\end{subfigure}
\caption{Four-stroke relaxation oscillation in the vdP oscillator model \eqref{VdPDyn2} with $\epsilon=10^{-2}$. \textcolor{black}{In \ref{vdPPhaseFig} the characteristic \W{\eqref{VdPChar} which represents the $x$-nullcline of system \eqref{VdPDyn} is shown in green, and the unique equilibrium of system \eqref{VdPDyn} is shown} in red. The same colour convention applies without reference in subsequent figures.}}
\label{vdPFig}
\end{figure}
Figure \ref{vdPFig} shows the typical vdP relaxation oscillator time traces of system \eqref{VdPDyn2} as well as the associated limit cycle and \W{cubic}-shaped characteristic in phase space. In particular, we notice that the time trace $x(t)$ is characterised by alternations between two slow and two fast motions over the course of a single relaxation cycle (`slow-fast-slow-fast') while the time trace $y(t)$ is uniformly slow. This is the hallmark of vdP-type relaxation oscillations, and the phase space representation shown in Figure~\ref{vdPPhaseFig} reveals the four distinct phases of the limit cycle.

\begin{rem}
	\W{All oscillators considered in this work will be recast as dynamical systems in such a way that the corresponding `characteristic' can be identified with a nonlinear nullcline. The specific form of the characteristic depends on the manner in which the problem is recast as a dynamical system. Typically, this will be clear from the context.}
\end{rem}

From a dynamical systems point of view, system \eqref{VdPDyn2} represents a singular perturbation problem in standard form
\begin{align}\label{SF}
\begin{array}{lcl}
x'=f(x,y,\epsilon), \\
y'=\epsilon g(x,y,\epsilon),
\end{array}
\end{align}
where the variable $y$ is considered slow relative to the fast variable $x$, which is a consequence of the order of magnitude difference in the right hand sides of \eqref{SF} caused by the singular perturbation parameter $\epsilon\ll 1$. The mathematical theory for such singular perturbation problems \W{in standard form} is well established and a wealth of results for such standard (vdP-type) relaxation oscillators have been derived with a variety of mathematical tools; we refer the reader to the book by Kuehn \cite{Kuehn2015} and the many references therein. 

\begin{rem}
	\label{rem:nonstandard}
\textcolor{black}{We want to emphasise that the mathematical theory developed for the study of singular perturbation problems \textit{does not depend} on the special (standard) form \eqref{SF}.  Significant advances have been made in the study of non-standard relaxation oscillators models; see e.g. \cite{Huber2005,Gucwa2009,KuehnSzmolyan2015,Kosiuk2016, Wechselberger2019}.}
\end{rem}

In the 1960s Le Corbeiller \cite{LeCorbeiller1960}, motivated by the study of electronic oscillators, termed the vdP-type oscillations {\em `four-stroke'} by reference to the four distinct phases undergone in the relaxation cycle, in order to distinguish these oscillations from {\em `two-stroke'} oscillations, which consist of only two distinct phases per cycle. Such two-stroke oscillations arise not only in the context of electronic oscillators, but also in mechanical oscillators with friction \cite{Berger2002,Bossolini2017,Pomeau2011,Popp1990,Thomsen2003,Won2016}, models of the trade cycle in economics \cite{Puu2006}, aircraft-ground dynamics \cite{Kristiansen2017,Rankin2011}, \textcolor{black}{climate models \cite{Omta2016},} discontinuous plastic deformation in metals \cite{Brons2005} and cell-signalling models \cite{Goldbeter1997}. We refer to Section~\ref{ExmpSec} were we review some of these two-stroke oscillator models.

\

Firstly, let us introduce a representative two-stroke oscillator model
\begin{equation}\label{SS1EQ}
x''+\bigg(\frac{\epsilon}{1-x'}-x'\bigg)+x=0, \qquad 0<\epsilon\ll1,
\end{equation}
which we recast as a dynamical system
\begin{align}\label{SS1}
\begin{array}{lcl}
x'=1-y, \\
y'=x-1+y+\frac{\epsilon}{y}\,.
\end{array}
\end{align}

\begin{figure}[t]
\captionsetup{format=plain}
\centering
\begin{subfigure}{.5\textwidth}
\centering
\includegraphics[trim={2.5cm 8.5cm 0 8.5cm},scale=0.5]{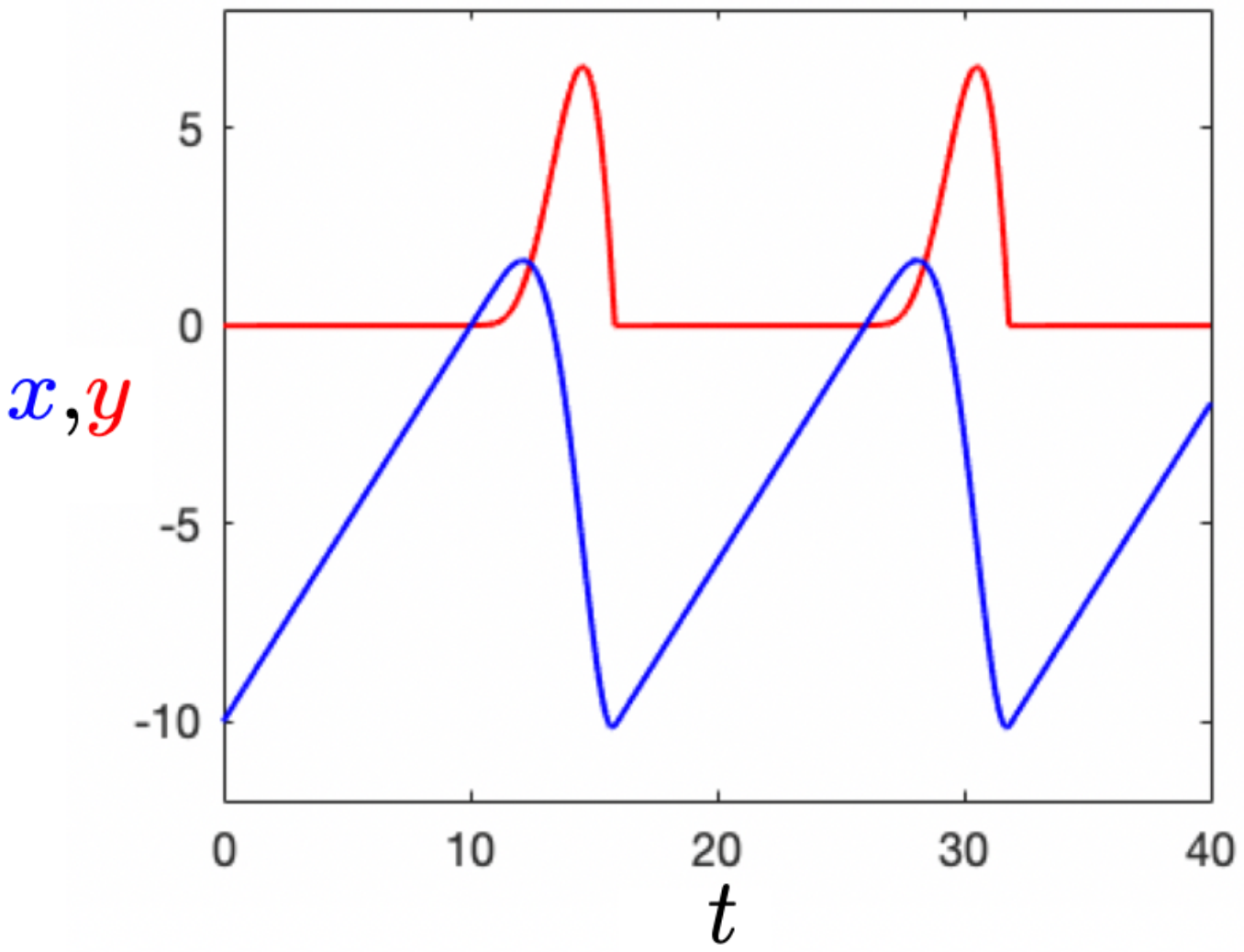}
\caption{Time trace.}\label{RatTraceFig}
\end{subfigure}%
\begin{subfigure}{.5\textwidth}
\centering
\includegraphics[trim={2.5cm 8.5cm 0 8.5cm},scale=0.5]{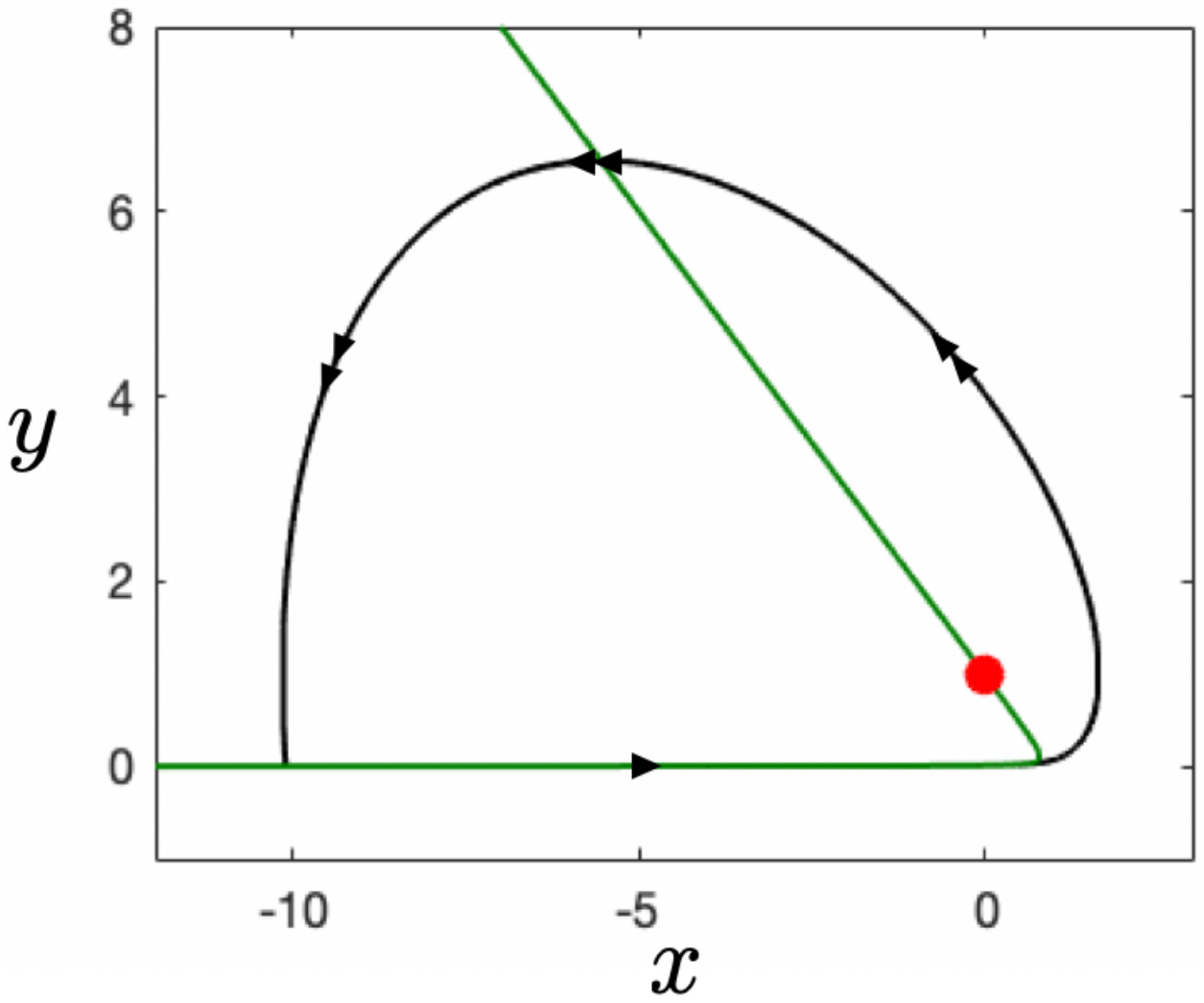}
\caption{Phase space.}\label{figure1_2b}
\end{subfigure}
\caption{Two-stroke oscillation in \eqref{SS1} with $\epsilon=10^{-2}$.}\label{RatPhaseFig}
\end{figure}

\noindent Figure \ref{RatTraceFig} shows the observed two-stroke oscillations in the time trace of $y(t)$, i.e. the limit cycle consist of two distinct phases: a `static' and a `dynamic' phase. The distinct phases can be recognised as segments in phase space relative to the characteristic
\begin{equation}\label{Rts-char}
R_{ts}(y):=1-y-\frac{\epsilon}{y}\,,
\end{equation}
 with the static phase comprising the part of the limit cycle which follows closely the part of the characteristic which asymptotes along $y=0$, and the dynamic phase comprising the part of the limit cycle which is `off' the characteristic (or crosses it); see Figure \ref{RatPhaseFig}. It is the $\epsilon$-dependent rational term $\epsilon/y$ that causes this distinct two-stroke behaviour. This term does not represent a uniformly small perturbation throughout phase space, 
 which distinguishes this model from standard perturbation problems such as the vdP-type four-stroke oscillator models.\\

While Figure~\ref{RatPhaseFig} mimics a slow-fast problem in phase space by following the characteristic for a distinct part of the two-stroke oscillations, the corresponding time trace shown in Figure~\ref{RatTraceFig} does not resemble an appreciable separation of timescales, i.e.~the static and dynamic phase evolve on comparable timescales. 
The underlying `relaxation' structure in this two-stroke model \eqref{SS1} is revealed via a \textcolor{black}{state-}dependent time transformation, a {\em desingularisation}, 
\begin{equation}\label{desing1}
dt = yd\bar t
\end{equation}
which gives
\begin{align}\label{SS2}
\begin{array}{lcl}
x'=(1-y)y, \\
y'=(x-1+y)y+\epsilon,
\end{array}
\end{align}
where with a slight abuse of notation the dash refers now to differentiation with respect to the new time $\bar t$.  System \eqref{SS2} is equivalent to system \eqref{SS1} for $y>0$ (and up to a change of orientation for $y<0$).\footnote{Equivalence is a basic topological concept in dynamical systems (see, e.g.,  \W{\cite{Guckenheimer1983,Kuznetsov2013}}) that is very useful to resolve dynamics near certain types of singularities such as poles of rational functions.}
Importantly, we have obtained a polynomial vector field in \eqref{SS2} including a uniformly small perturbation term, and this system produces now {\em relaxation-type} two-stroke oscillations as shown in Figure \ref{RatTraceFig2}, i.e.~the static and dynamic phase of the two-stroke oscillator can now be clearly identified as slow and fast segments in the corresponding time traces. Hence, the desingularisation \eqref{desing1} has allowed us to extract a singular perturbation problem in the form of a slow-fast system which preserves the dynamical properties of the original two-stroke oscillator \eqref{SS1EQ}. 
\begin{figure}[t]
\hspace{0em}\centerline{\includegraphics[trim={0 8.5cm 0 8.5cm},scale=.5]{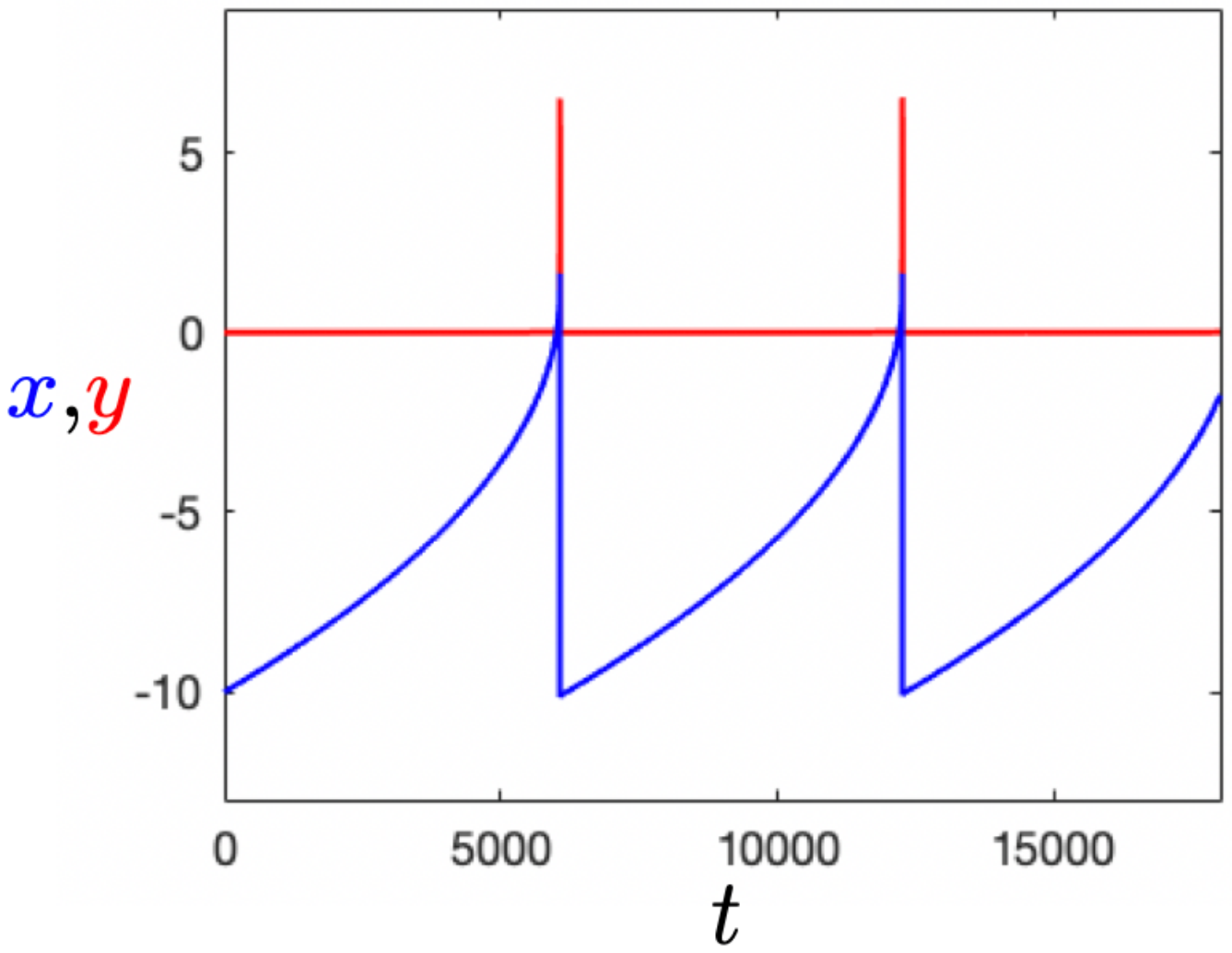}}
\caption{Two-stroke relaxation oscillation in \eqref{SS2}; compare with Figure \ref{RatTraceFig}.} 
\label{RatTraceFig2}
\end{figure}

We emphasise that the singular perturbation problem \eqref{SS2} is {\em not} given in the standard form \eqref{SF}.   While we can clearly distinguish slow and fast motions, there is no distinction between slow and fast variables in this model, i.e.~both time traces shown in Figure~\ref{RatTraceFig2} consist of slow and fast segments. System \eqref{SS2} is part of a more general class of perturbation problems,
\begin{equation}\label{generalSP}
\begin{pmatrix}
{x'} \\
{y'}
\end{pmatrix}
=
\begin{pmatrix}
N_1(x,y) \\
N_2(x,y)
\end{pmatrix}
f(x,y) +\epsilon 
\begin{pmatrix}
G_1(x,y,\epsilon) \\
G_2(x,y,\epsilon)
\end{pmatrix}
\,,
\end{equation}
with the specific choice of
\begin{equation}\label{NFG1}
\begin{pmatrix}
N_1(x,y) \\
N_2(x,y)
\end{pmatrix}
=
\begin{pmatrix}
1-y \\
x-1+y
\end{pmatrix}\,,
\quad f(x,y)=y\,,
\quad
\begin{pmatrix}
G_1(x,y,\epsilon) \\
G_2(x,y,\epsilon)
\end{pmatrix}
=\begin{pmatrix}
0 \\
1
\end{pmatrix},
\end{equation}
for system \eqref{SS2}. \W{Systems of the form \eqref{generalSP}  are identified as singular perturbation problems whenever the set
	\begin{equation}
	\label{S_0}
	S_0=\{z\in\mathbb R^2:N(z)f(z)=0\}
	\end{equation}
contains a smooth one-dimensional submanifold of $\mathbb R^2$, which reflects the geometric definition of a singular perturbation problem.\footnote{\W{A general treatment of singular perturbation theory beyond the standard form in arbitrary dimensions can be found in \cite{Wechselberger2019,Goeke2014}.}}} The main aim of this article is to consider this more general class of singular perturbation problems \eqref{generalSP} and provide a \textcolor{black}{framework for the application of geometric singular perturbation theory (GSPT)} for which a (global) separation of slow and fast variables is not required. This will allow us to prove existence, uniqueness and stability results for a class of two-stroke relaxation oscillations, in a manner conceptually analogous to the derivation of similar results in vdP-type oscillators; see, e.g., \cite{Krupa2001b}. 

\W{In general,  a standard form \eqref{SF} of a singular perturbation problem where the variables reflect the scale separation can only be achieved locally but not globally. Thus from an application point-of-view it is desirable to provide tools to analyse  singularly perturbation problems of the more general form \eqref{generalSP} in a coordinate-independent manner.} As pointed out by Fenichel in his seminal work on GSPT \cite{Fenichel1979}, a global standard {\em `form is not natural, however, because it depends on the choice of special coordinates'}, and he clearly explains how to deal with a more general form. The general GSPT framework we present here goes beyond Fenichel's work and deals also with {\em loss of normal hyperbolicity}, a necessary ingredient for two-stroke (or any relaxation-type) oscillations.

\begin{rem}
In the context of quasi steady-state reduction (QSSR) techniques commonly used in chemical reaction systems, 
Goeke \& Walcher \cite{Goeke2014} provide a general framework that applies to systems \eqref{generalSP} in the normally hyperbolic case as outlined by Fenichel. 
Kaleda \cite{Kaleda2011} shows the first results on two-stroke oscillations with a focus on the bifurcation of a slow-fast separatrix loop. 
 \textcolor{black}{\W{We would like to point out that our results on relaxation oscillations (Section 4) can be derived from a more general result found in \cite{Maesschalck2011}, Theorem 2.} We present an alternative argument in this work in order to illustrate the scope, simplicity and applicability of the methods developed herein.}
\end{rem}

The manuscript is structured as follows. In Section \ref{ExmpSec} we motivate our study by introducing examples of two-stroke oscillations arising in models for transistor oscillations and mechanical oscillations with friction, and show how they can be appropriately phrased as singular perturbation problems in the general form \eqref{generalSP}.
In Section \ref{FrameworkSec} we develop a framework for the application of GSPT to this more general class of singular perturbation problems \eqref{generalSP}\textcolor{black}{, with an emphasis on the presentation of \W{global, relaxation oscillatory}  results from classic GSPT in a coordinate-independent setting}. 
In Section \ref{ROSec} we define minimal assumptions on singular perturbation problems in the general form \eqref{generalSP} that enable two-stroke relaxation oscillations, and prove existence, uniqueness and stability results. Section \ref{DiscussionSec} collects and contrasts dynamic features associated with two- and four-stroke relaxation oscillations, and we discuss possible transitions between two and four-stroke relaxation oscillations. We also discuss the onset of two-stroke oscillations in a mechanical oscillator. Finally, in Section \ref{ConcSec} we conclude and outline future work.

\section{Two-stroke oscillators in applications}\label{ExmpSec}

We motivate our work by providing examples of two-stroke oscillators which we draw from the study of nonlinear transistor oscillators and mechanical oscillators with friction.

\subsection{An electronic  two-stroke oscillator model}

In \cite{Hester1968}, the author uses the {\em Ebers-Moll} large-signal approximation to show that a class of nonlinear transistor oscillators including tuned-collector, tuned-base, and Hartley transistor oscillators, can be described by a Lord-Rayleigh type equation
\begin{equation}\label{LRType}
\ddot{x}+r_{em}(\dot{x})+x=0,
\end{equation}
where $x$ denotes (dimensionless) `current'\textcolor{black}{,
\[
r_{em}(\dot x) = -\mu \left(e^{a \dot x} - \kappa e^{(a+b) \dot x} \right) ,
\]
$\mu,\kappa,a,b>0$ are positive constants, and we assume 
\begin{equation}\label{ass-kappa}
\kappa<\frac{a}{a+b}<1
\end{equation}
is sufficiently small.} Equation \eqref{LRType} can be recast as a dynamical system 
\begin{align}\label{LRDyn}
\begin{array}{lcl}
\dot{x}=-\bar{y}, \\
\dot{\bar y}=x-\bar{R}_{em}(\bar{y}),
\end{array}
\end{align}
with nonlinear characteristic
\begin{equation}\label{EMChar}
\bar{R}_{em}(\bar{y})=-r_{em}(-\bar{y})=\mu e^{-a\bar{y}}(1-\kappa e^{-b\bar{y}}),
\end{equation}
This characteristic $\bar{R}_{em}(\bar y)$ has a  (unique) turning point at
\begin{equation}\label{EMFold}
(x_{\ast},{y}_\ast)=\left(\mu \bigg(\frac{a}{(a+b)\kappa}\bigg)^{a/b}\bigg(\frac{b}{a+b}\bigg),-\frac{1}{b}\ln\bigg(\frac{a}{(a+b)\kappa}\bigg) \right)\,,
\end{equation}
where ${y}_\ast<0$ due to \eqref{ass-kappa}. The characteristic approaches zero as $\bar y\to\infty$ and grows exponentially towards $-\infty$ for $\bar y<{y}_\ast$. For later convenience, we make a coordinate change $\bar y=y+y_\ast$ which shifts the position of the turning point of the characteristic to the $x$-axis, i.e.~we obtain 
\begin{align}\label{EM2}
\begin{array}{lcl}
\dot{x}=-y- {y}_\ast, \\
\dot{y}=x-R_{em}(y).
\end{array}
\end{align}
with characteristic
\begin{equation}\label{EMChar2}
R_{em}(y)=x_*\bigg( \frac{a+b}{b}\,e^{-ay} -\frac{a}{b}\,e^{-(a+b)y}\bigg)\,.
\end{equation}
Figure \ref{EMTraceFig} shows the observed two-stroke oscillations in system \eqref{EM2}, i.e.~the time trace $y(t)$ consists of a static and a dynamic phase. The corresponding phase portrait, Figure~\ref{ExpTTraceFig1}, shows that the limit cycle follows closely the characteristic in the static phase while it is off the characteristic in the dynamic phase, i.e.~it shows the same qualitative features of the two-stroke oscillator shown in Figure~\ref{RatPhaseFig}. The main (mathematical) difference between the part of the two characteristics \eqref{EMChar2} and \eqref{Rts-char} that determine the static phase of the limit cycle, is exponential growth versus unlimited growth due to a pole of a rational function. 
\begin{figure}[t]
\captionsetup{format=plain}
\centering
\begin{subfigure}{.5\textwidth}
  \centering
  \includegraphics[trim={2.5cm 8.5cm 0 8.5cm},scale=0.5]{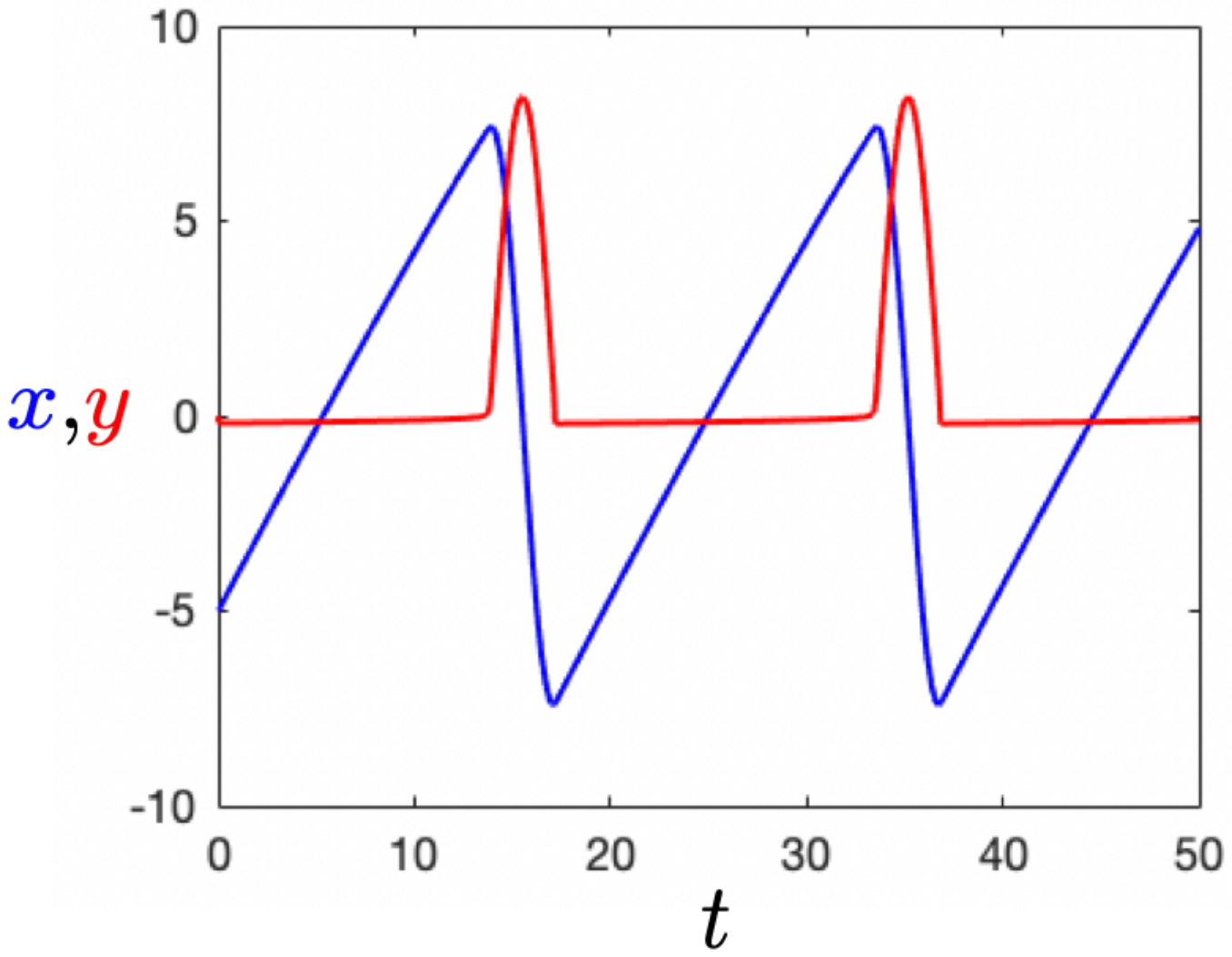}
  \caption{Time trace.}\label{EMTraceFig}
  \label{fig:sub1}
\end{subfigure}%
\begin{subfigure}{.5\textwidth}
  \centering
  \includegraphics[trim={2.5cm 8.5cm 0 8.5cm},scale=0.5]{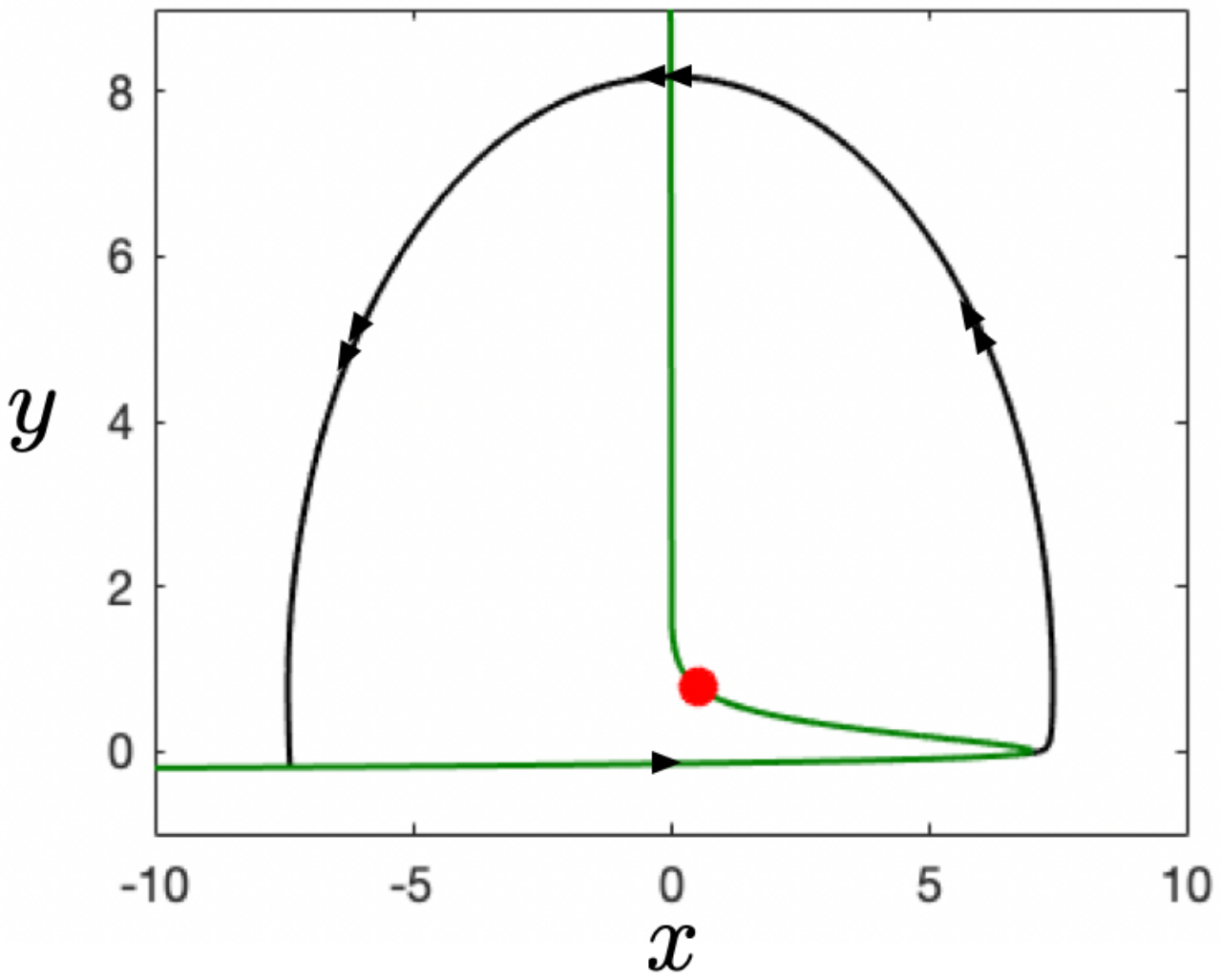}
  \caption{Phase space.}\label{ExpTTraceFig1}
  \label{fig:sub5}
\end{subfigure}
\caption{Two-stroke oscillation in an Ebers-Moll model for a nonlinear transistor \eqref{EM2}, with parameter values $(\mu,\kappa,a,b)=(1,10^{-2},4,6)$.}\label{EMFig}
\label{fig:test}
\end{figure}
\begin{figure}[t]
\hspace{0em}\centerline{\includegraphics[trim={0 8cm 0 8cm},scale=.4]{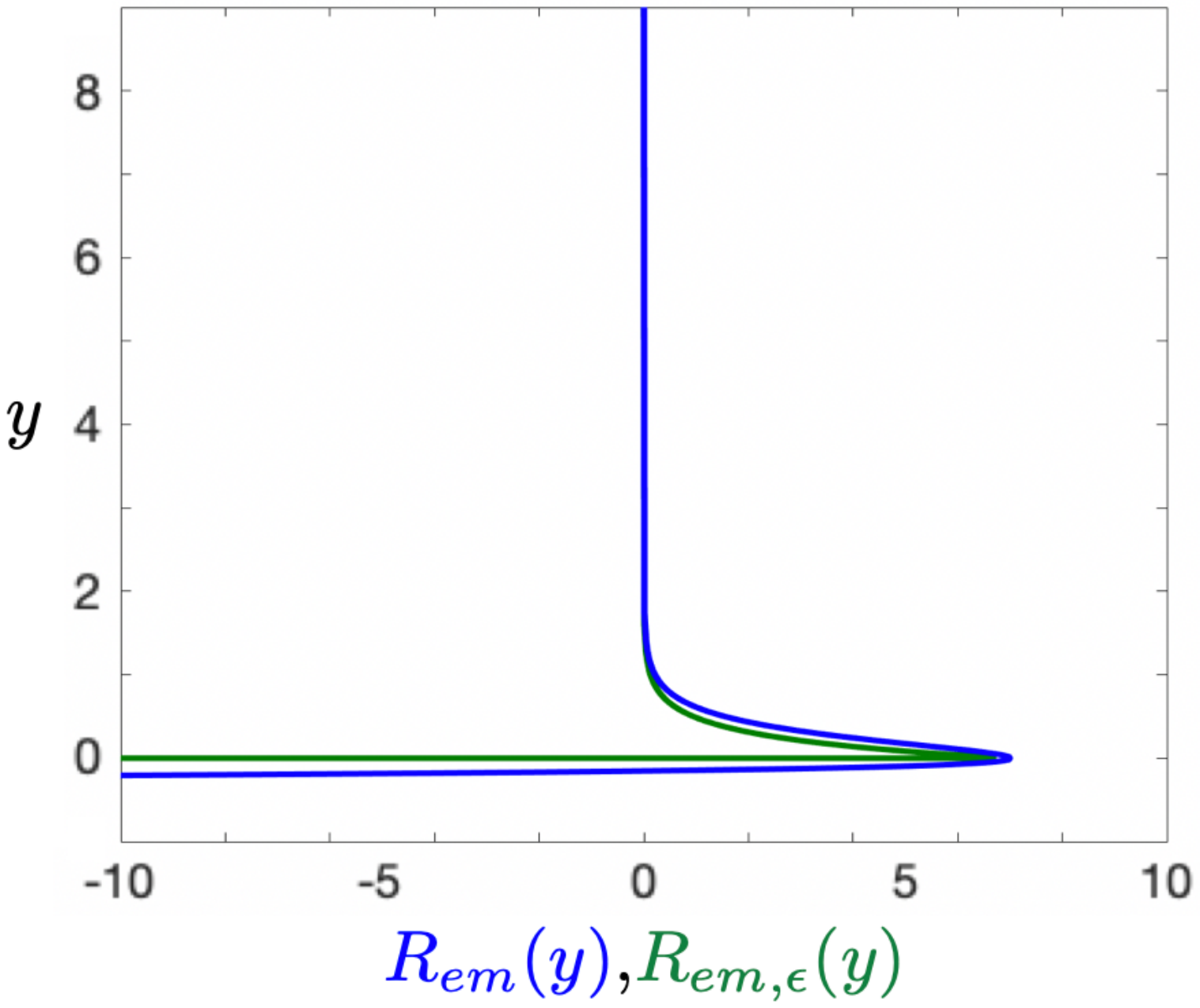}}
\caption{Comparison of the approximation \eqref{ExpCharApp} (green), and the actual characteristic \eqref{EMChar2} (blue). For the characteristic \eqref{EMChar2}, the same parameter values as Figure \ref{EMFig} were used; note the turning point is located at $(x_*,0)$. For the characteristic \eqref{ExpCharApp} we used $\epsilon=10^{-3}$.}\label{EMCompFig}
\end{figure}
This observation motivates us to model and replace the exponential growth of the characteristic \eqref{EMChar2} for $y<0$ by a rational term, and we define the following approximation 
\begin{equation}\label{ExpCharApp}
R_{em,\epsilon}(y)=x_* e^{-ay}-\frac{\epsilon}{y},
\end{equation}
\noindent
where $0<\epsilon\ll1$ is a sufficiently small parameter. The rational term $-\epsilon/y$ approximates the lower branch near $y=0$ while the exponential profile of the vertical asymptote is preserved by the $x_* e^{-ay}$ term;  Figure \ref{EMCompFig} compares the two characteristics \eqref{EMChar2} and \eqref{ExpCharApp}. \\

\begin{rem}
The characteristic \eqref{ExpCharApp} has two turning points for $\epsilon>0$: one at $(x,y)\sim(x_*,0)$, and the other at $(x,y)\sim(0,\hat{y})$, with $\hat{y}>0$. Both characteristics \eqref{EMChar2} and \eqref{ExpCharApp} have a vertical asymptote along $x=0$, though the asymptote is approached from the right in \eqref{EMChar2} and from the left in \eqref{ExpCharApp}. These minor differences in the characteristics have no significant effect on the dynamics. 
For similar reasons, we also refrain from shifting the characteristic \eqref{ExpCharApp} slightly to the right to remove (unphysical, yet very small) negative values for $y>0$.

\end{rem}
%
\begin{figure}[t]
\captionsetup{format=plain}
\centering
\begin{subfigure}{.5\textwidth}
  \centering
  \includegraphics[trim={2.5cm 8.5cm 0 8.5cm},scale=0.5]{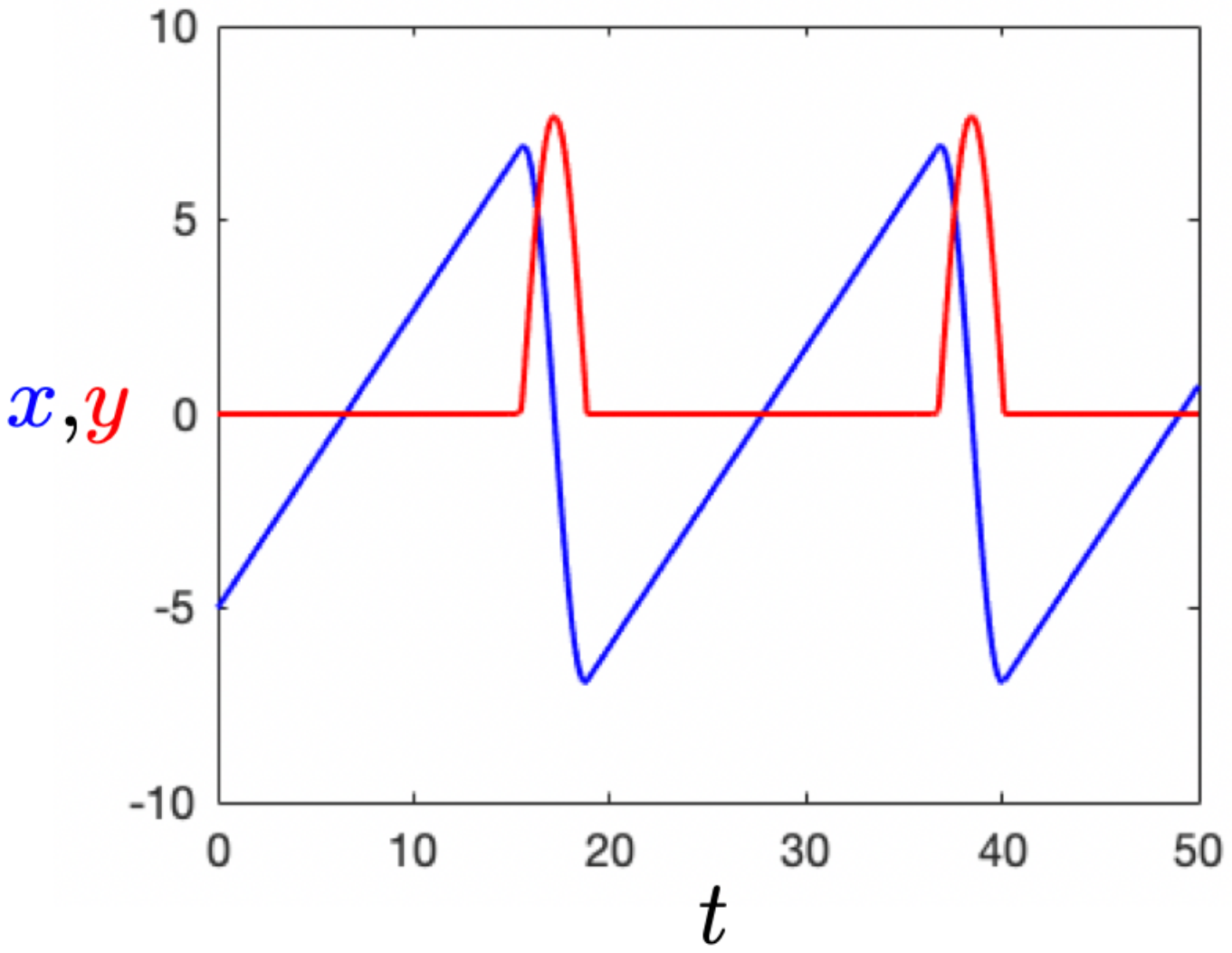}
  \caption{Time trace.}\label{ExpTraceFig}
  \label{fig:sub1}
\end{subfigure}%
\begin{subfigure}{.5\textwidth}
  \centering
  \includegraphics[trim={2.5cm 8.5cm 0 8.5cm},scale=0.5]{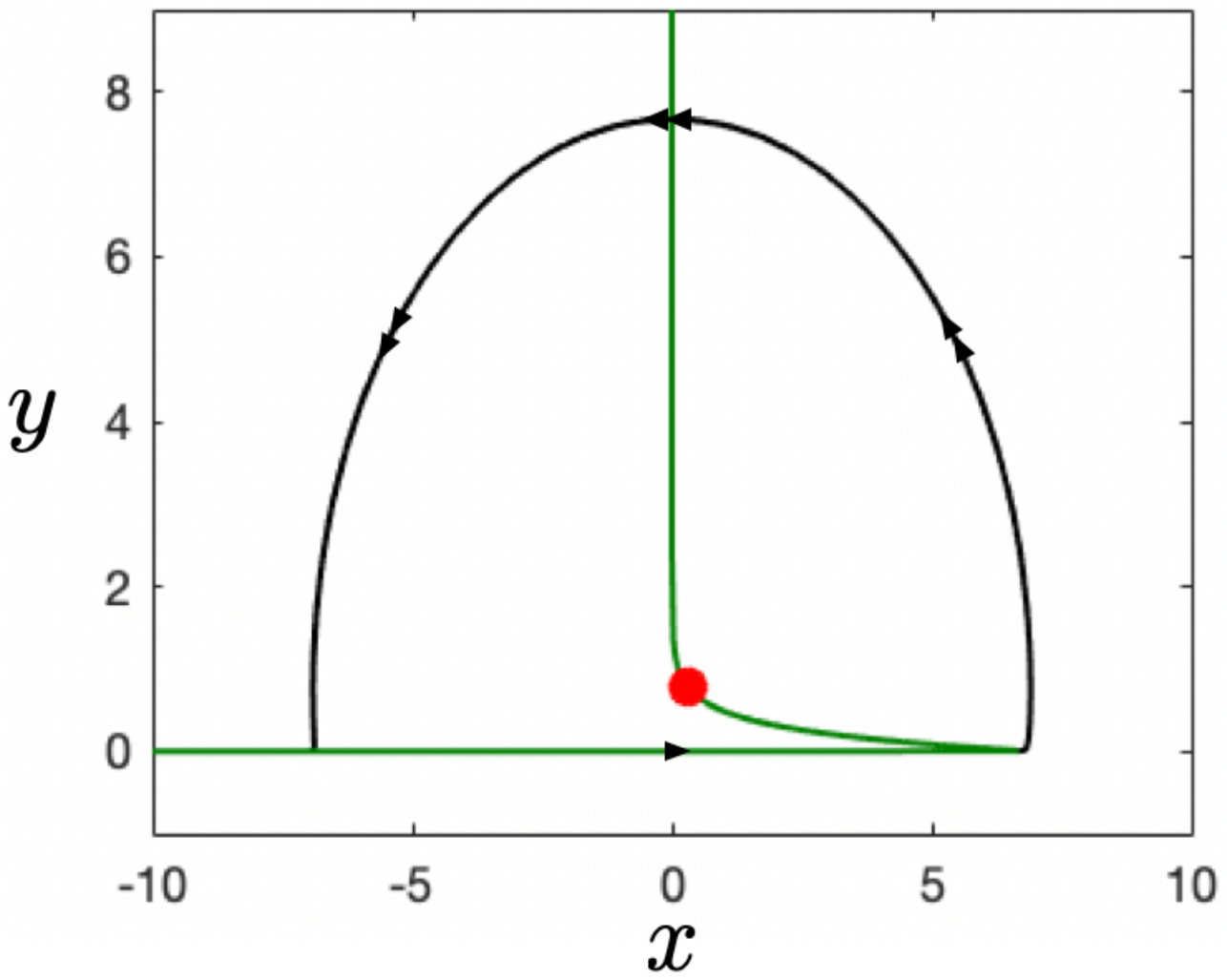}
  \caption{Phase space.}\label{ExpPhaseFig}
  \label{fig:sub5}
\end{subfigure}
\caption{Two-stroke oscillation in system \textcolor{black}{\eqref{ExpSys1}}, for the same parameter values as Figure \ref{EMFig} and $\epsilon=10^{-3}$; compare with Figure \ref{EMFig}.}\label{ExpAppFig}
\label{fig:test}
\end{figure}
The corresponding approximating system is given by
\begin{align}\label{ExpSys1}
\begin{array}{lcl}
x'=-{y}_\ast-y, \\
y'=x-R_{em,\epsilon}(y),
\end{array}
\end{align}
and the observed two-stroke oscillation and corresponding limit cycle are shown in Figure \ref{ExpAppFig}; compare with Figure \ref{EMFig}.
%
\begin{figure}[t]
\hspace{0em}\centerline{\includegraphics[trim={0 8.5cm 0 8.5cm},scale=.5]{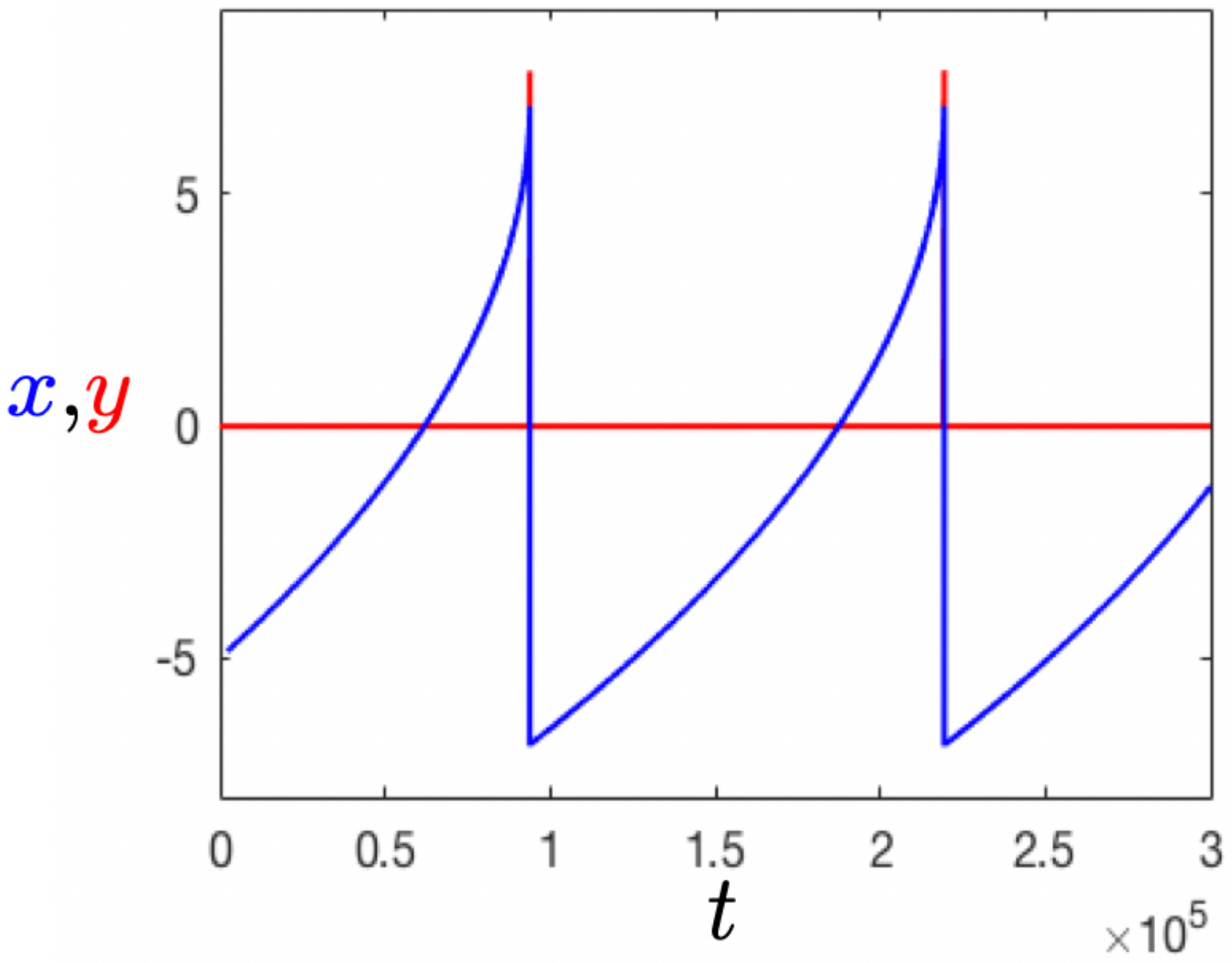}}
\caption{Two-stroke relaxation oscillation in \eqref{ExpSys} with the same parameters as Figure \ref{EMFig} and $\epsilon=10^{-3}$. Compare with Figure \ref{ExpTraceFig}, and note the time-scale difference.}\label{ExpTTraceFig}
\end{figure}

In order to reveal the underlying relaxation structure, we make the same \textcolor{black}{state-}dependent time desingularisation $dt = y\,d\bar t$ as in \eqref{desing1}, which leads to the system
\begin{align}\label{ExpSys}
\begin{array}{lcl}
x'=-({y}_\ast+y)y, \\
y'=(x-x_\ast e^{-ay})y+\epsilon ,
\end{array}
\end{align}
where the dash notation now denotes differentiation with respect to $\bar{t}$. System \eqref{ExpSys} is equivalent to \eqref{ExpSys1} on $\{y>0\}$  (and up to a change of orientation on $\{y<0\}$). The effect of the desingularisation \eqref{desing1} is that we observe relaxation-type two-stroke oscillations in system \eqref{ExpSys}; see Figure  \ref{ExpTTraceFig}. This system is in the general form of a singular perturbation problem \eqref{generalSP} with 
\begin{equation}\label{NFG2}
\begin{pmatrix}
N_1(x,y) \\
N_2(x,y)
\end{pmatrix}
=
\begin{pmatrix}
-y_\ast -y \\
x-x_\ast e^{-ay}
\end{pmatrix}\,,
\quad f(x,y)=y\,,
\quad
\begin{pmatrix}
G_1(x,y,\epsilon) \\
G_2(x,y,\epsilon)
\end{pmatrix}
=\begin{pmatrix}
0 \\
1
\end{pmatrix}
\,.
\end{equation}
Section~\ref{FrameworkSec} will provide the GSPT tools for the analysis of the observed two-stroke relaxation oscillations.

\subsection{ A stick-slip oscillator model}\label{SSSec}

Mechanical systems commonly exhibit two-stroke oscillations as a consequence of the so-called stick-slip phenomenon due to friction. The corresponding observations range from earthquake faulting or the sound of a violin, to the unwanted screeching of chalk on the chalk-board or sliding of machine parts (see \cite{Berger2002} and references therein). 

Many of the key dynamical features occurring in such complex mechanical systems with friction are captured by the simple spring-mass system shown in Figure \ref{SMFig}, in which one considers a mass $m$ on a conveyor moving with constant velocity $v_0$, attached to a wall by a spring of stiffness $k$. For small displacements $x$, the mass moves with the conveyor. In this case, the relative velocity between mass and belt is $v_r=\dot{x}-v_0=0$. This is referred to as the `stick mode', or `static phase'. As the mass moves with the belt, the restoring force of the spring increases linearly in accordance with Hookes law, and the mass starts to slip once this restoring force balances the maximum static friction: this is the `stick-slip transition'. Once slipping begins, we are in the `slip mode' or `dynamic phase'. Finally, the spring counteracts the sliding motion until static friction takes hold again, and the process starts over.
\begin{figure}[t]
\hspace{0em}\centerline{\includegraphics[trim={0 2.5cm 0 2.5cm},scale=.35]{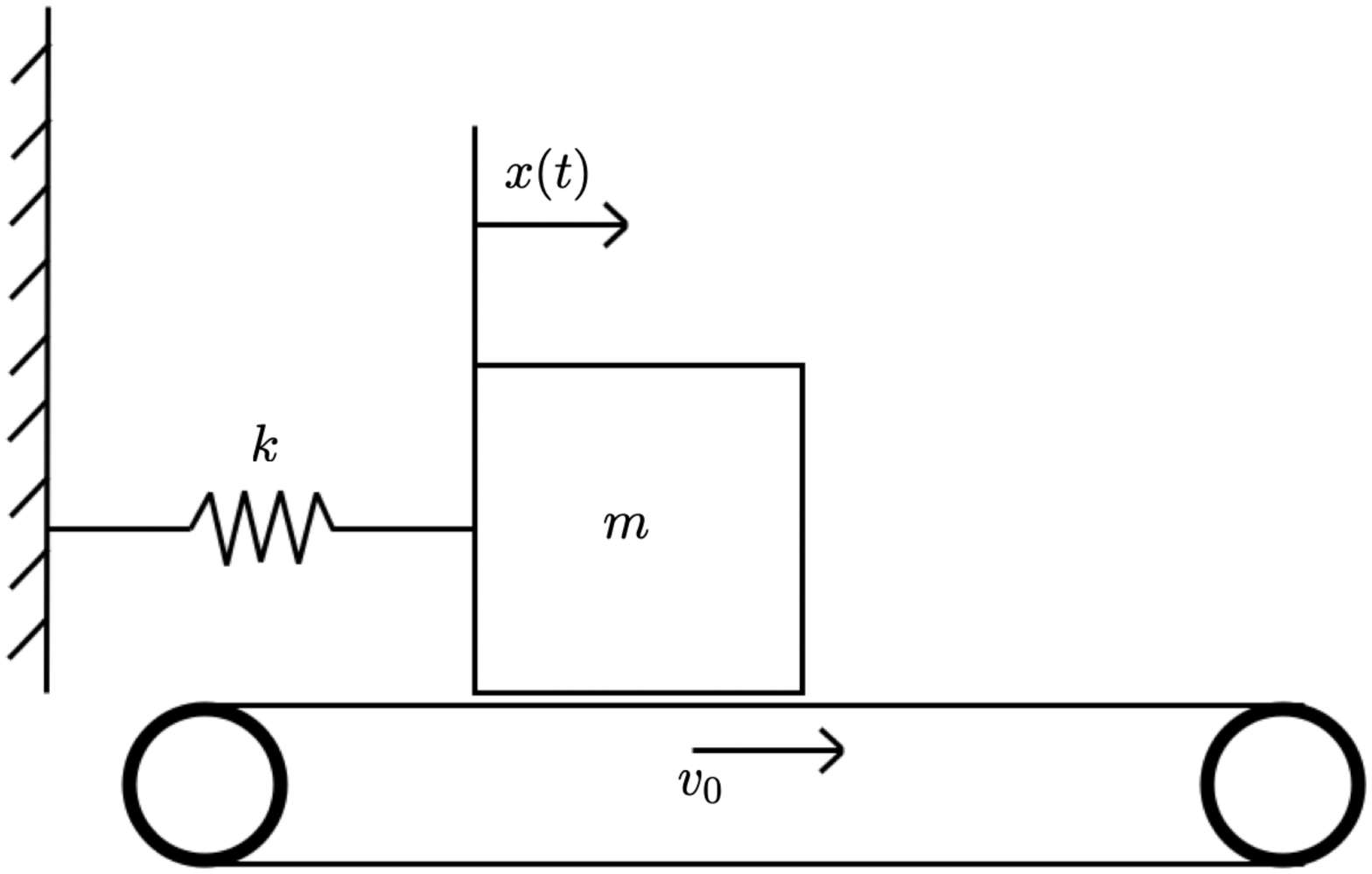}}
\caption{Schematic of a spring-mass oscillator.}\label{SMFig}
\end{figure}

This mechanical system is described by the following (dimensionless) equation of motion
\begin{equation}\label{SSEqn}
\ddot{x}+F_f(v_r)+x=0,
\end{equation}
where $x$, $t$ and $v_r=\dot{x}-v_0$ denote displacement, time and relative velocity, respectively. The stick-slip problem is usually modelled as a discontinuous system, with static and dynamic phases treated independently in accordance with
\begin{equation}\label{Fric}
F_f(v_r)=
\begin{cases}
- x, & v_r=0, \\
sgn(v_r)\mu(v_r), & v_r\neq0,
\end{cases}
\end{equation}
where the function $\mu(v_r)$ denotes the {\em coefficient of friction}. The transition from stick to slip is determined by the {\em stiction law}, which asserts that the stick phase ($v_r=0$) is maintained as long as
\begin{equation}\label{Stiction}
\vert F_f(0)\vert=\vert x\vert \leq \mu_s,
\end{equation}
where $\mu_s$ denotes the maximal value of static friction capable of preventing the onset of the slipping motion. The function $\mu(v_r)$ defines the dynamic friction law, while the $sgn(v_r)$ term ensures that the force due to friction opposes the direction of relative motion. As is typical, we assume $\mu(v_r)$ is an even and strictly positive function, noting that in general the specific form depends on the application. It is crucial for the existence of oscillations that $F_f'(v_r)<0$ for small $|v_r|$. This is known as the negative slope criterion in the stick-slip literature, and is observed in many applications exhibiting the so-called `Stribeck effect' \cite{Berger2002}. Following \cite{Won2016}, we consider the two forms for $\mu(v_r)$ shown in Figure \ref{SSCharsFig} as important examples: 

\begin{itemize}
\item\
$\mu(v_r)$ decays exponentially toward a minimum value $\mu_m$, as in Figure \ref{ExpCharFig}. Such a dependence is typical for the case of dry friction between solid surfaces, and an appropriate form is suggested in \cite{Berger2002} (see also \cite{Won2016}) as follows,
\begin{equation}\label{ExpSSChar}
\mu(v_r)=\mu_m+(\mu_s-\mu_m)e^{-a\vert v_r|},
\end{equation}
where $a>0$ is a fitting parameter which controls the slope of the characteristic.

\item
$\mu(v_r)$ is modelled as a polynomial function which decays initially towards a minimum value $\mu_m$ at relative velocity $v_m$, and increases for $|v_r|>v_m$, as in Figure \ref{PolyCharFig}. This is suitable for systems exhibiting dry friction for small $\vert v_r\vert$, and liquid or `viscous' friction for larger $\vert v_r\vert$. An appropriate form appears in, e.g. \cite{Thomsen2003,Ibrahim1994,Panovko1965} (see also \cite{Won2016,Chen2014}):
\begin{equation}\label{PolySSChar}
\mu(v_r)=\mu_s-\frac{3(\mu_s-\mu_m)}{2v_m}\vert v_r\vert+\frac{(\mu_s-\mu_m)}{2v_m^3}\vert v_r\vert^3\,.
\end{equation}

\end{itemize}

\begin{figure}[t]
\captionsetup{format=plain}
\centering
\begin{subfigure}{.5\textwidth}
  \centering
  \includegraphics[trim={3.5cm 3cm 0 3cm},scale=0.35]{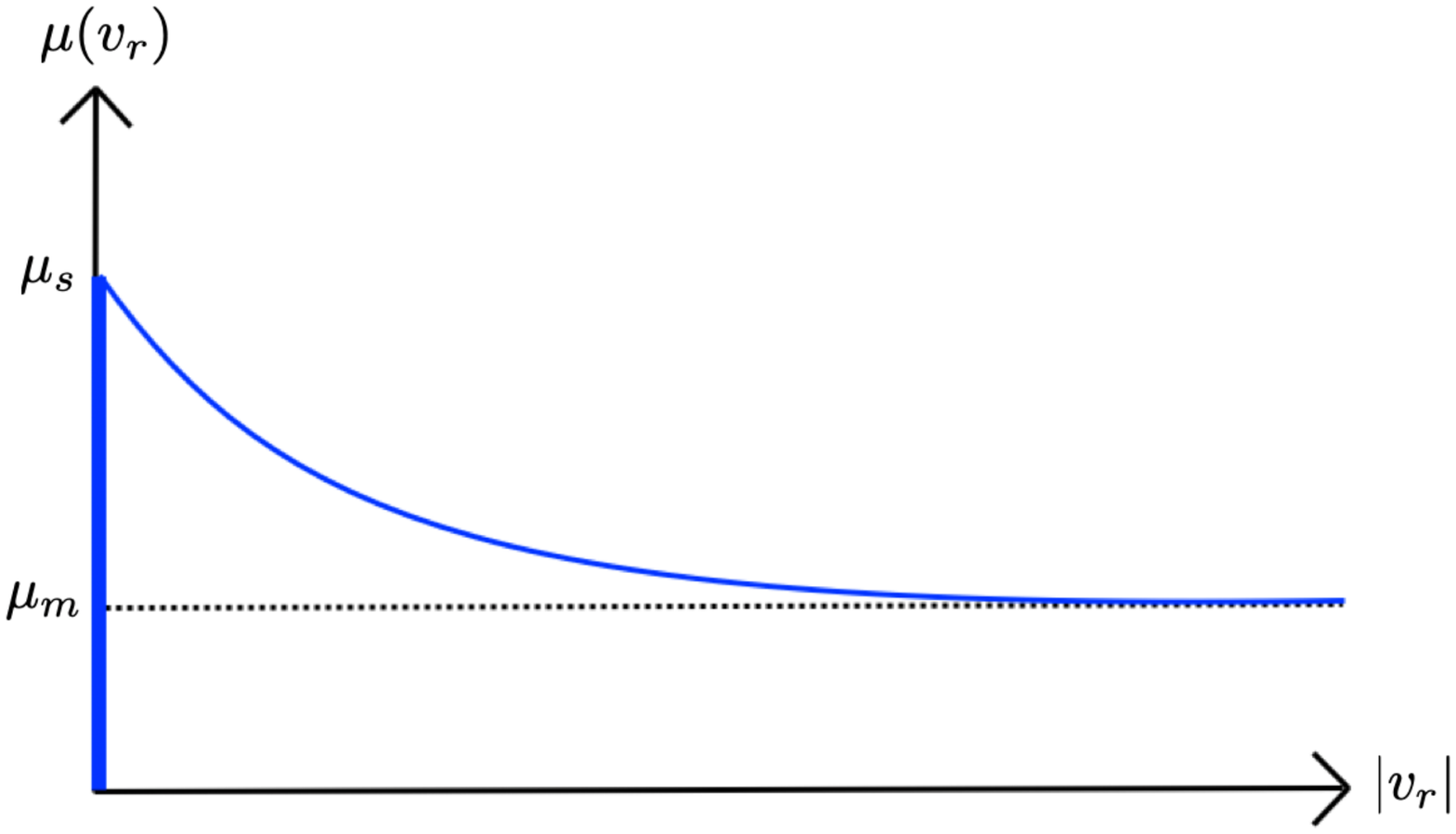}
  \caption{Exponential-type.}\label{ExpCharFig}
  \label{fig:sub1}
\end{subfigure}%
\begin{subfigure}{.5\textwidth}
  \centering
  \includegraphics[trim={3.5cm 3cm 0 3cm},scale=0.35]{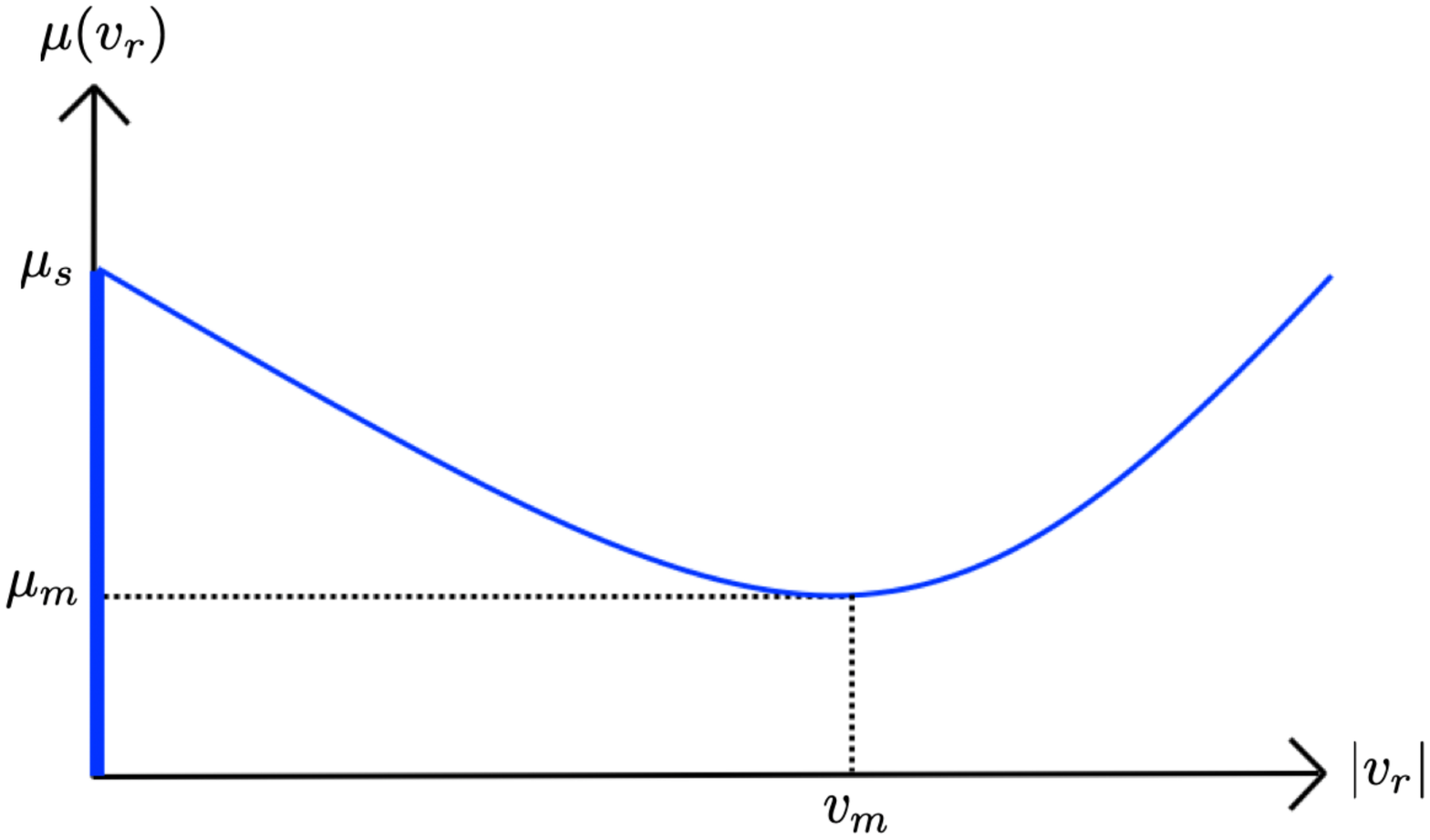}
  \caption{Polynomial-type.}\label{PolyCharFig}
  \label{fig:sub5}
\end{subfigure}
\caption{Friction-velocity curves for exponential and polynomial-type laws in (a) and (b) respectively. The thick line for $v_r=0$, $\mu<\mu_s$ indicates the regime where the static friction dominates.}\label{SSCharsFig}
\label{fig:test}
\end{figure}

\begin{rem}
Note that positivity of $\mu(v_r)$ and the presence of the $sgn(v_r)$ term in \eqref{Fric} leads to a jump discontinuity at $v_r=0$. This (discontinuous) modelling choice serves as an approximation of the real (smooth) mechanical system. Two-stroke cycles in such discontinuous models are obtained by phase space constructions in which segments from static and dynamic phases are concatenated, recalling that the transition from stick to slip is enforced when the threshold in \eqref{Stiction} is reached; see e.g.~\cite{Berger2002,Popp1990,Won2016,Chen2014} and the references therein.
\end{rem}

In the absence of external forcing, the velocity of the mass never exceeds the belt velocity, and so stick-slip oscillations occur only for $v_r\leq0$ \cite{Thomsen2003}. We intend to `smooth' or `regularise' the discontinuous dynamical system \eqref{SSEqn} in the domain relevant for our analysis, i.e~for $v_r\le 0$. This is achieved by replacing \eqref{Fric} with the $\epsilon$-dependent characteristic
\begin{equation}\label{SSAppChar}
F_{f,\epsilon}(v_r)=sgn(v_r)\mu_\epsilon(v_r)=sgn(v_r)\bigg(\mu(v_r)-\frac{\epsilon}{\vert v_r\vert}\bigg), \qquad 0<\epsilon\ll1.
\end{equation}
This approximation effectively smooths out the corner at $(v_r,\mu(v_r))=(0,\mu_s)$ and incorporates an asymptote along the line $v_r=0$. Figure \ref{SSCharsCompFig} shows that a sufficiently accurate approximation is obtained in the case of both polynomial and exponential-type characteristics.
We emphasise that the regularisation in \eqref{SSAppChar} is valid only for $v_r< 0$ (or $v_r> 0$). In fact, this smooth approximation leads to a `stick mode' $v_r\approx 0$  since the asymptotic speed $v_r=0$ cannot be reached for $0<\epsilon\ll 1$.

	\begin{rem}
		\label{rem:pws}
		\textcolor{black}{For background on the theory of regularisation in non-smooth problems, we refer the reader to \cite{Bernardo2008,Jeffrey2018} and the many references therein. Our choice of regularisation \eqref{SSAppChar} valid only over the half-plane $v_r < 0$ simplifies the analysis of this autonomous `stick-slip' problem \textit{sufficiently}. 
		On the other hand, we would like to emphasise that non-smooth problems involving external forcing or bifurcations occurring near $v_r = 0$ require regularisations which are valid in an entire \textit{neighbourhood} of $v_r = 0$,  and the analysis becomes (necessarily) more involved. We refer the reader to, e.g., \cite{Bossolini2017,Bonet2016,Kristiansen2015a,Kristiansen2015b,Kristiansen2019,Kristiansen2019c} for rigorous treatments of such problems using techniques from GSPT and blow-up.}
	\end{rem}

\begin{figure}[t]
\captionsetup{format=plain}
\centering
\begin{subfigure}{.5\textwidth}
  \centering
  \includegraphics[trim={0cm 7.5cm 0 7.5cm},scale=0.4]{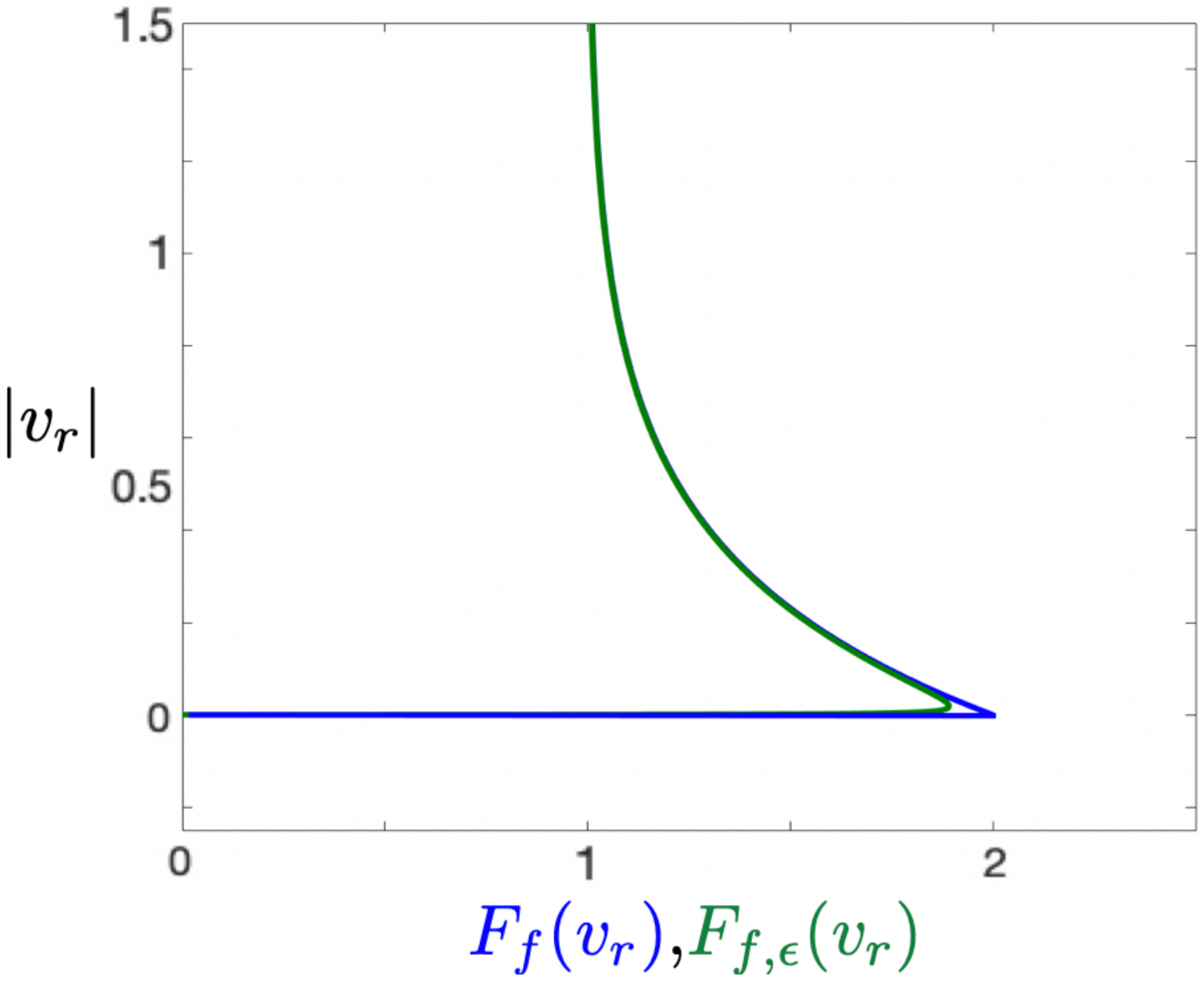}
  \caption{}\label{SSExpAppFig}
  \label{fig:sub1}
\end{subfigure}%
\begin{subfigure}{.5\textwidth}
  \centering
  \includegraphics[trim={0cm 7.5cm 0 7.5cm},scale=0.4]{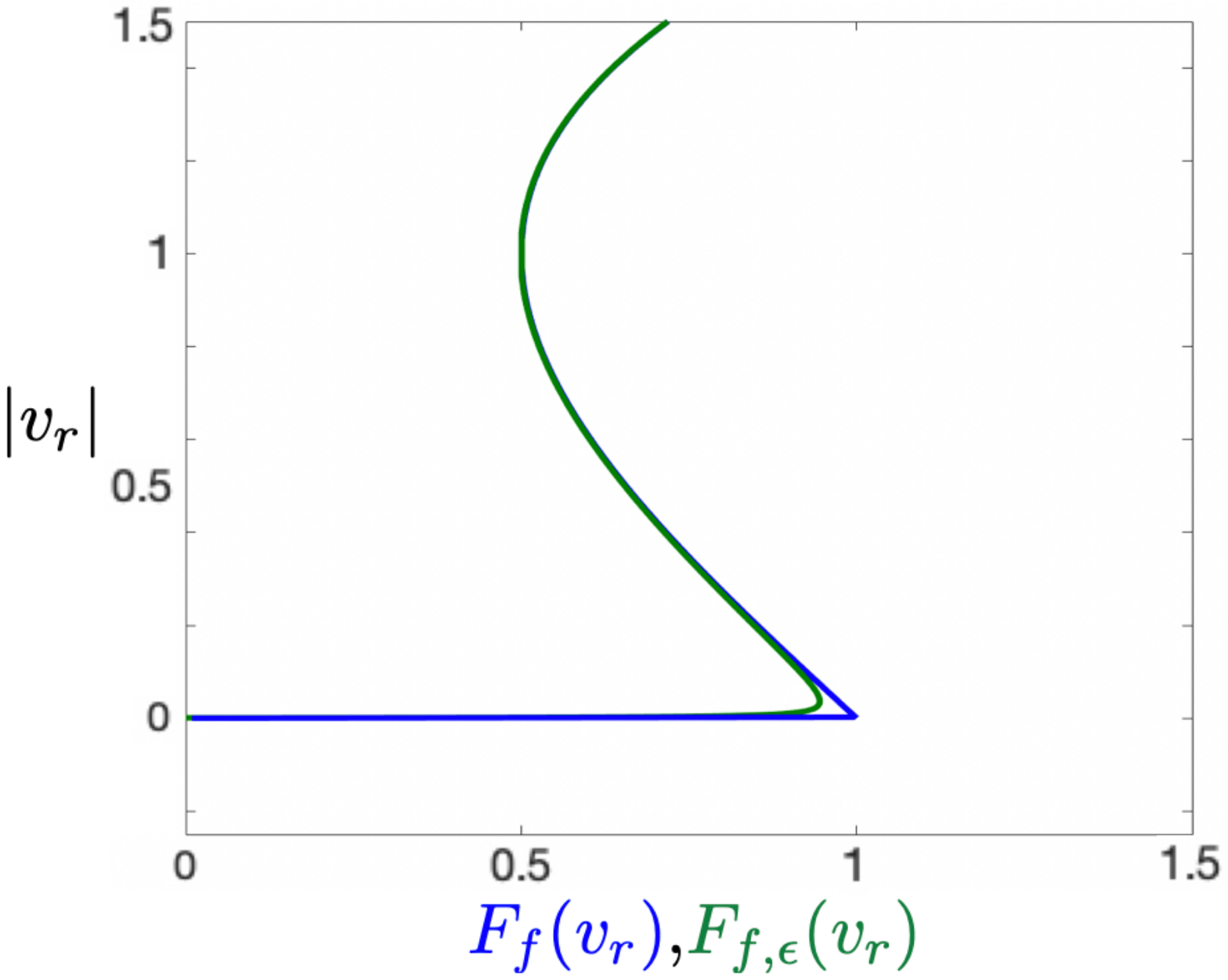}
\caption{}\label{SSStribeckFig}
  \label{fig:sub5}
\end{subfigure}
\caption{In (a): the characteristic \eqref{SSAppChar} with \eqref{ExpSSChar} plotted against the exponential-type (piecewise) characteristic. Parameters: $(\epsilon,v_0,\mu_m,\mu_s,a)=(10^{-3},0.5,1,2,3)$. In (b): the characteristic \eqref{SSAppChar} with \eqref{PolySSChar} plotted against the polynomial-type (piecewise) characteristic. Parameters: $(\epsilon,v_0,v_m,\mu_m,\mu_s)=(10^{-3},0.25,1,0.5,1)$. In each case the original characteristic is plotted in blue, and the approximating characteristic in green. Note that $y=-v_r$.}\label{SSCharsCompFig}
\label{fig:test}
\end{figure}

We introduce a new variable $y=-v_r=v_0-\dot{x}$ in the friction model \eqref{SSEqn} to obtain the following dynamical system,
\begin{align}\label{SSAppSys}
\begin{array}{lcl}
\dot x=v_0-y, \\
\dot y=x-\mu(y)+\frac{\epsilon}{y},
\end{array}
\end{align}
where we used the smooth approximation \eqref{SSAppChar} of the friction characteristic \eqref{Fric} restricted to $v_r<0$ (i.e.~$y>0$) together with the symmetry $\mu(y)=\mu(-y)$.
Again, after making the same time desingularisation $dt = y\,d\bar t$ as in \eqref{desing1} we obtain the system
\begin{align}\label{SS}
\begin{array}{lcl}
x'=(v_0-y)y, \\
y'=(x-\mu(y))y+\epsilon,
\end{array}
\end{align}
where dash denotes differentiation with respect to the new time $\bar t$. System \eqref{SS} is in the general form of a singularly perturbed system \eqref{generalSP} with
\begin{equation}\label{NFG3}
N(x,y)=
\begin{pmatrix}
{v_0-y} \\
{x-\mu(y)}
\end{pmatrix},
\qquad 
f(x,y)=y, \qquad
G(x,y,\epsilon)=
\begin{pmatrix}
{0} \\
{1}
\end{pmatrix},
\end{equation}
and it is equivalent to system \eqref{SSAppSys} in the relevant domain $y>0$. 
With the specific choice of $\mu (v_r)$ given by \eqref{ExpSSChar} or \eqref{PolySSChar} we observe
two-stroke oscillations in \eqref{SSAppSys}, and these show up as two-stroke relaxation oscillations in the corresponding system \eqref{SS}. 
Figure \ref{SSPolyFig} shows the time trace and corresponding relaxation cycle for the polynomial-type characteristic  \eqref{PolySSChar}. 

\begin{rem}
The exponential-type characteristic \eqref{ExpSSChar} has exactly the same features as the characteristic  
\eqref{ExpCharApp} describing two-stroke relaxation oscillation in the transistor oscillator model. 
Consequently, the time trace and corresponding relaxation cycle are similar to those in Figures \ref{ExpTTraceFig} and \ref{ExpPhaseFig}, respectively. 
\end{rem}

\begin{figure}[t]
\captionsetup{format=plain}
\centering
\begin{subfigure}{.5\textwidth}
  \centering
  \includegraphics[trim={2.5cm 8.5cm 0 8.5cm},scale=0.5]{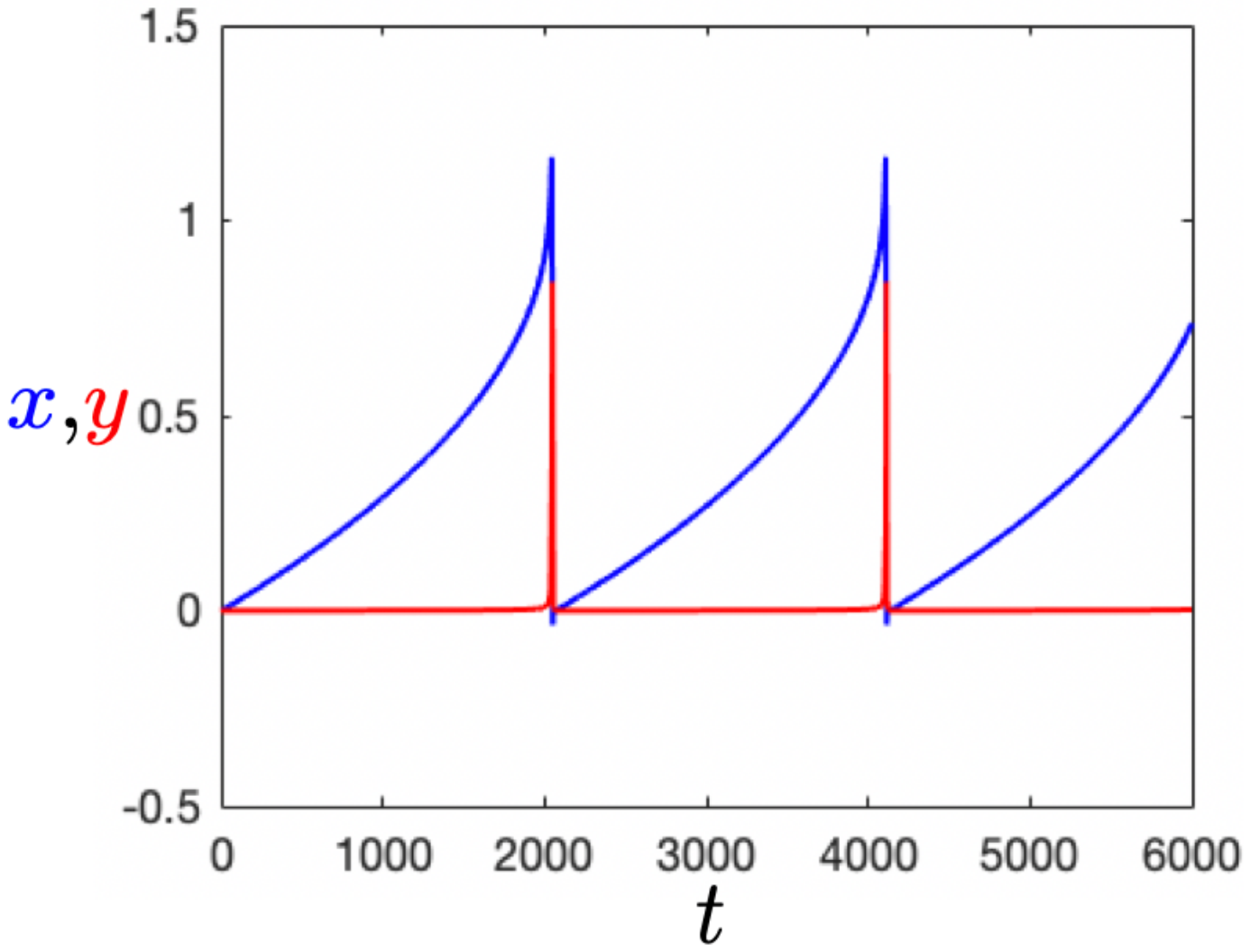}
  \caption{Time trace.}\label{SSExpPhaseFig}
  \label{fig:sub1}
\end{subfigure}%
\begin{subfigure}{.5\textwidth}
  \centering
  \includegraphics[trim={2.5cm 8.5cm 0 8.5cm},scale=0.5]{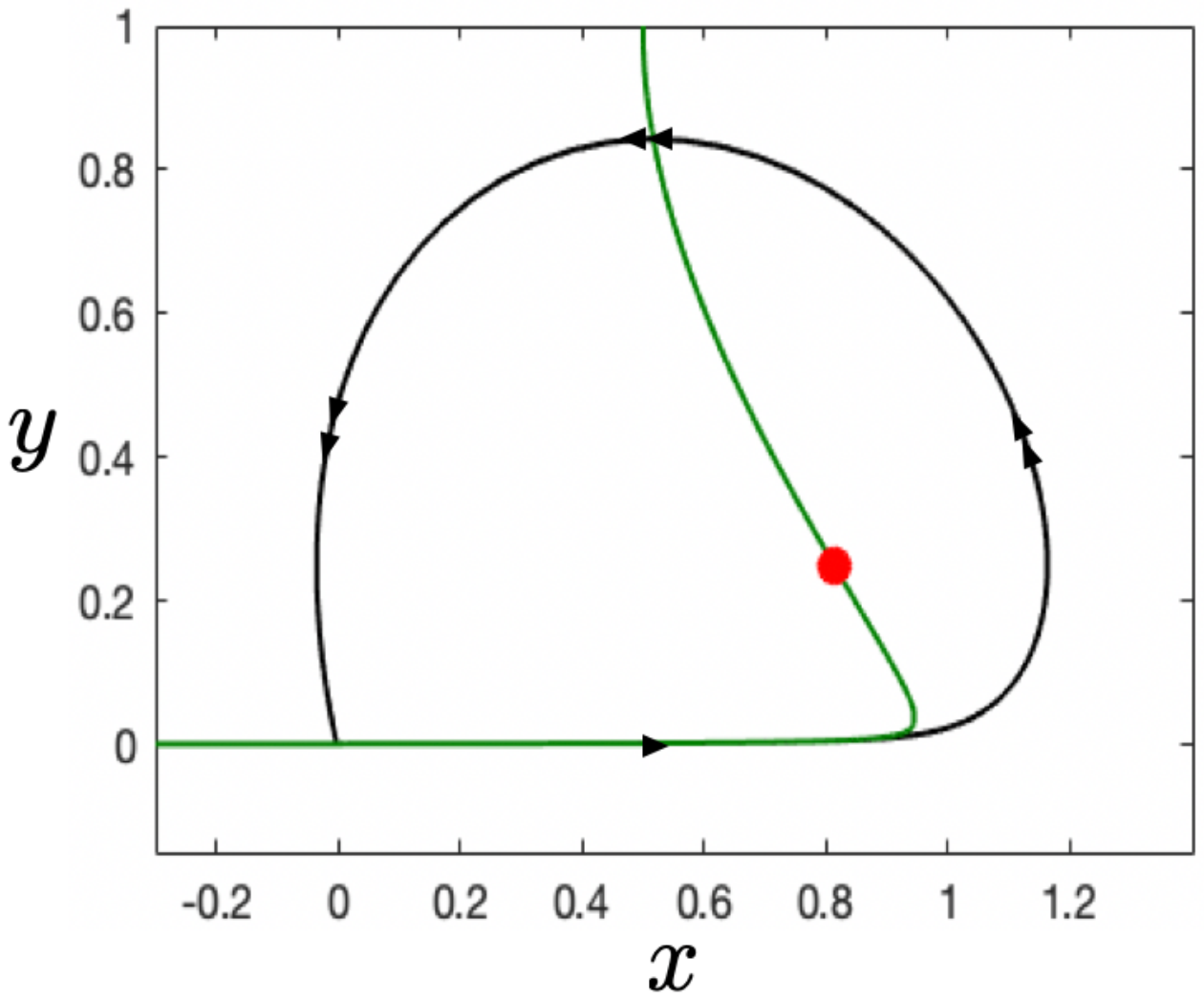}
  \caption{Phase space.}\label{SSExpPolyFig}
  \label{fig:sub5}
\end{subfigure}
\caption{Two-stroke relaxation oscillation in \eqref{SS} with polynomial-type dynamic friction \eqref{PolySSChar}, with parameters identical to those in Figure \ref{SSStribeckFig}.}\label{SSPolyFig}
\label{fig:test}
\end{figure}

\begin{rem}\label{SSMinRem} The dynamic friction coefficients \eqref{ExpSSChar} and \eqref{PolySSChar} have a common linearisation near $v_r=0$,
\[
\mu(v_r)=a-b\vert v_r \vert,
\]
with positive coefficients $a,b$. Since the representative model \eqref{SS2} is in the form \eqref{SS} with
\begin{equation}\label{MinSSChar}
\mu(y)=1-y, \qquad v_0=1,
\end{equation}
it can be viewed as a local minimal `stick-slip' model. More importantly, it serves as a mathematical `canonical' model for two-stroke relaxation oscillations.
\end{rem}

\section{A general GSPT framework}\label{FrameworkSec}

We consider planar perturbation problems of the general form
\begin{equation}\label{Gen1}
z'=H(z;\epsilon)=N(z)f(z)+\epsilon G(z;\epsilon), \qquad z\in\mathbb{R}^2, \qquad 0<\epsilon\ll1,
\end{equation}
where $N(z)=(N_1(z),N_2(z))^T$ and $G(z;\epsilon)=(G_1(z;\epsilon),G_2(z;\epsilon))^T$ are sufficiently smooth vector fields, $f:\mathbb{R}^2\to\mathbb{R}$ is a sufficiently smooth function, and the $'$ notation denotes differentiation with respect to (fast) time $t$. We will frequently denote variables componentwise by $z=(x,y)^T$. We also consider the equivalent problem on a slow timescale $\tau=\epsilon t$:
\begin{equation}\label{Gen2}
\dot{z}=\frac{1}{\epsilon}H(z;\epsilon)=\frac{1}{\epsilon}N(z)f(z)+G(z;\epsilon),
\end{equation}
where the dot notation denotes differentiation with respect to (slow) time $\tau=\epsilon t$. Notice that \eqref{Gen1} and \eqref{Gen2} are equivalent for $\epsilon>0$, but not in the limit $\epsilon\to 0$. In order to study systems of form \eqref{Gen1} respectively \eqref{Gen2}, we require coordinate independent analogues for all the notions of standard GSPT \cite{Kuehn2015,Jones1995}.

\begin{rem}
We emphasise that standard form problems \eqref{SF} can always be written in the general form \eqref{Gen1}: 
\begin{equation}\label{OurSF}
\begin{pmatrix}
{x'} \\
{y'}
\end{pmatrix}
=
\begin{pmatrix}
{f(x,y,\epsilon)} \\
{\epsilon g(x,y,\epsilon)}
\end{pmatrix}
=
\begin{pmatrix}
{1} \\
{0}
\end{pmatrix}
f_0(x,y)+\epsilon
\begin{pmatrix}
{f_R(x,y,\epsilon)} \\
{g(x,y,\epsilon)}
\end{pmatrix},
\end{equation}
where $f_0(x,y)=f(x,y,0)$. In the remainder of this work, we refer to \eqref{OurSF} when referring to problems in the `standard form'.

The converse, however, is not true. In particular, if the vector field $N$ contains isolated singularities \W{(assumed to be bounded away from the set $\{f=0\}$) then there is no \textit{global} flowbox transformation rectifying orbits of $N$ and hence, no global transformation of \eqref{Gen1} into standard form \eqref{OurSF}.} All our example systems, \eqref{SS2}, \eqref{ExpSys} and \eqref{SS}, are of this kind. 
\end{rem}

\subsection{Layer problem}\label{LayerSec}

\begin{defn} (Layer problem). The system
\begin{equation}\label{Layer}
z'=H(z,0)=h(z)=N(z)f(z)
\end{equation}
obtained from \eqref{Gen1} in the limit $\epsilon\to0$ is called the layer problem.
\end{defn}

\textcolor{black}{We impose the following on the system \eqref{Layer} for the remainder of this work.}

\begin{assumption}\label{BasicAss}
The set of equilibria for the layer problem \eqref{Layer} takes the form of a \textcolor{black}{\textit{disjoint}} union \textcolor{black}{
	\[
	S_0=S\cup V_0 ,
	\]}
where
\begin{equation}\label{S}
S=\big{\{}z\in\mathbb{R}^2|f(z)=0\big{\}}
\end{equation}
is a one-dimensional \textcolor{black}{regularly embedded submanifold of $\mathbb R^2$ called the \textit{critical manifold}}, and
\begin{equation}\label{NSing}
V_0=\big{\{}z\in\mathbb{R}^2|N(z)=(0,0)^T \big{\}}
\end{equation}
is the (possibly empty) set of singularities of the vector field $N$. \textcolor{black}{Without loss of generality, we assume that $V_0$ contains only \textit{isolated} singularities.}
\end{assumption}

\begin{rem}\label{rem:S}
The existence of a smooth critical manifold $S$ \eqref{S} defines \textcolor{black}{the} system \eqref{Gen1} as a singular perturbation problem in the GSPT sense; see \cite{Fenichel1979}. \textcolor{black}{Accordingly, we will often refer to $\epsilon \to 0$ as the \textit{singular limit.}} The assumption that $V_0$ is discrete simplifies the analysis without being too restrictive for our purposes. 

The GSPT literature also deals with \textcolor{black}{the case that the `critical manifold' $S$ is an \textit{immersed submanifold}, allowing for self-intersections. By assuming that $S$ is a regularly embedded submanifold in Assumption \ref{BasicAss}, we are ruling out the possibility of self-intersections.} \W{Thus, $Df$ is assumed to be non-vanishing $\forall z\in S$ in this work}.
\end{rem}

\begin{lem}\label{LayerNEquivalence} 
Orbits of the layer problem \eqref{Layer} coincide with orbits of the auxiliary system $z'=N(z)$ on $\mathbb{R}^2\setminus S$.
\end{lem}
\begin{proof}
 These systems are equivalent via the time desingularisation $dt=f(z)d\bar{t}$ modulo a reversal of orientation on $\{z\in\mathbb{R}^2|f(z)<0\}$.
\end{proof}

Evaluating the Jacobian for \textcolor{black}{the} system \eqref{Layer} along $S$ gives
\[
Dh\big|_{S}=NDf\big|_{S}=
\begin{pmatrix}
{N_1D_xf} & {N_1D_yf} \\
{N_2D_xf} & {N_2D_yf} 
\end{pmatrix}\Bigg|_{S},
\]
which has  a single trivial eigenvalue $\lambda_0=0$ since $\det Dh\big|_{S}=0$, and a single non-trivial eigenvalue \textcolor{black}{at $z\in S$} given by
\begin{equation}\label{EV}
\lambda(z)=\text{Tr}\big(NDf\big|_{\textcolor{black}{z}}\big)=\langle \textcolor{black}{\nabla f},N\rangle\big|_{\textcolor{black}{z}}.
\end{equation}
The corresponding eigenspace of the trivial eigenvalue spans the tangent space $T_zS$ at $z\in S$, i.e. it is orthogonal to the gradient of $f$,
\[
(T_zS_n)^\perp=\text{span}\, \textcolor{black}{\nabla f}\big|_z\,.
\]
The corresponding eigenspace of the nontrivial eigenvalue is spanned by $N(z)$, since 
\[
(Dh)N=(NDf)N=N\langle  \textcolor{black}{\nabla} f,N\rangle=N\lambda\,,\qquad \forall z\in S\,.
\]

\begin{defn} (Normal Hyperbolicity). We say that $z\in S$ is normally hyperbolic if the non-trivial eigenvalue $\lambda(z)\neq0$, and likewise call any submanifold $S_n\subseteq S$ normally hyperbolic if $\lambda(z)\neq0$, $\forall z\in S_n$. A normally hyperbolic submanifold $S_n$ is called attracting if $\lambda(z)<0$ $\forall z\in S_n$, and repelling if $\lambda(z)>0$, $\forall z\in S_n$.
\end{defn}

Let $S_n\subseteq S$ be a normally hyperbolic submanifold of $S$ and $z\in S_n$. The inner product in \eqref{EV}, which is non-zero $\forall z\in S_n$, induces the pointwise splitting
\[
T_z\mathbb{R}^2\big|_{S_n}=T_zS_n\oplus \mathcal{N}_z.
\]
Here $\mathcal{N}_z$ denotes the linear transverse fiber which has base at $z\in S_n$ and is spanned by $N(z)$. The collection of all such fibers forms a linear transverse fiber bundle $\mathcal{N}$, leading to the splitting
\begin{equation}\label{Splitting}
T\mathbb{R}^2\big|_{S_n}=TS_n\oplus \mathcal{N},
\end{equation}
where $TS_n$ is the tangent bundle similarly obtained by collecting all the tangent spaces $T_zS_n$ with base at $z\in S_n$. We note that more generally, one can consider the existence of such a splitting as the defining property of normal hyperbolicity. 

\begin{defn}\label{CPDef} (Contact point). A point $F\in S$ such that
\begin{equation}\label{CPCond}
\lambda(F)=\langle \textcolor{black}{\nabla f},N\rangle\big|_F=0
\end{equation}
is called a contact point.
\end{defn}
A contact point $F\in S$ indicates a loss of normal hyperbolicity of the critical manifold $S$.

\begin{defn}\label{CODef} (Contact order). 
Let $U\subset\mathbb{R}^2$ be a neighbourhood of a contact point $F$ and denote by 
$\mathcal{F}$ \textcolor{black}{a} segment of the corresponding layer orbit through $F$ in $U$. 
\textcolor{black}{Given local parameterisations $c(s)$ and $p(s)$, $s \in [-s_0,s_0]$, of $\mathcal{F}$ and $S$, respectively, such that $c(0) = p(0) = F$. The layer flow has contact order $\sigma_F\in\mathbb{N}^+$ at $F$ with $S$ if
	\[
	c^{(m)}(0) = p^{(m)}(0) \text{ for } m = 1, \ldots \sigma_F, \qquad  c^{(\sigma_F+1)}(0) \neq p^{(\sigma_F+1)}(0).
	\]}
\end{defn}

A contact point $F\in S$ is a point of tangency between the layer flow and $S$, and the contact order $\sigma_F$ describes the degree of the tangency.

\begin{prop}\label{CPOrder}
Let $F\in S$ denote a contact point of system \eqref{Gen1}, and assume without loss of generality that $S$ is given as a graph $y=\varphi(x)$ locally near $F$. Then the contact order at 
$F$ is given by
\begin{equation}\label{CO1}
\sigma_F=\text{min}\big{\{}n\in\mathbb{N}^+\big|D^{(n)}_{x}\langle \textcolor{black}{\nabla} f (x,\varphi(x)),N(x,\varphi(x))\rangle\big|_{F}\neq0\big{\}}.
\end{equation}
\end{prop}

\begin{proof}
	\textcolor{black}{We begin by introducing coordinates for which $S$ is  locally rectified along the $x-$axis.}	

\begin{lem}\label{FlatSLem} Given $z=(x,y)\in S$ such that $D_yf|_z\neq0$. There exists a smooth change of coordinates such that \textcolor{black}{the system \eqref{Gen1} can be written locally as
\begin{equation}\label{FlatS}
\begin{split}
\begin{pmatrix}
{x'} \\
{u'}
\end{pmatrix}
&:=
\tilde{N}(x,u)u+\epsilon\tilde{G}(x,u,\epsilon) \\
&=
\begin{pmatrix}
{N_1(x,M(x,u))} \\
{\langle \textcolor{black}{\nabla} f(x,M(x,u),N(x,M(x,u))\textcolor{black}{)}\rangle}
\end{pmatrix}
u+\epsilon
\begin{pmatrix}
{G_1(x,M(x,u),\epsilon)} \\
{\langle  \textcolor{black}{\nabla} f(x,M(x,u)),G(x,M(x,u),\epsilon)\rangle}
\end{pmatrix}\,.
\end{split}
\end{equation}
The critical manifold $S=\{(x,u)\in\mathbb{R}^2\,|\,u=0\}$ is straightened locally along the $x$-axis.}
\end{lem}
\begin{proof}
Define $u=f(x,y)$, which has a locally well-defined inverse $y=M(x,u)$, since $D_yf|_z\neq0$. Transforming \eqref{Gen1} into the new $(x,u)$-coordinate system gives \textcolor{black}{the system \eqref{FlatS}.}
\end{proof}
\noindent
For system \eqref{FlatS} with a flat manifold $S$, the condition for tangency between the layer flow and $S$ at $F$ is
\[
\frac{du}{dx}\bigg|_F=\frac{\tilde{N}_2}{\tilde{N}_1}\bigg|_F=0\,,
\]
where $\tilde{N}_2|_F= \lambda(F)=0$ and $\tilde{N}_1|_F\neq 0$ (by Assumption \ref{BasicAss}). The contact order is
\[
\sigma_F=\text{min}\bigg{\{}n\in\mathbb{N}^+\bigg|D^{(n)}_x\bigg(\frac{du}{dx}\bigg)\bigg|_F\neq0\bigg{\}},
\]
which can be simplified by noting that
\[
D^{(n)}_x\bigg(\frac{du}{dx}\bigg)\bigg|_F=0 \qquad \iff \qquad D^{(n)}_x\tilde{N}_2(x,\varphi(x))\big|_F=0,
\]
and hence
\begin{equation}\label{COFlatS}
\sigma_F=\text{min}\big{\{}n\in\mathbb{N}^+\big|D^{(n)}_x\tilde{N}_2(x,\varphi(x))\big|_F\neq0\big{\}}.
\end{equation}
In terms of the original coordinates, we have
\[
\sigma_F=\text{min}\big{\{}n\in\mathbb{N}^+\big|D^{(n)}_{x}\langle  \textcolor{black}{\nabla} f(x,\varphi(x)),N(x,\varphi(x))\rangle\big|_{F}\neq0\big{\}},
\]
as required.
\end{proof}

\begin{rem}
Lemma \ref{FlatSLem} holds locally about any $z\in S$, i.e. one can locally rectify $S$ near contact points or normally hyperbolic points. 
\end{rem}

\begin{rem}\label{rem:lambda}
An order one contact point $F$ implies that the non-trivial eigenvalue $\lambda$ \eqref{EV}  switches sign along $S$ as one `crosses'  $F$. The converse is also true, i.e.~a switch in stability of $S$ can occur only via a contact with the layer flow.
\end{rem}

\begin{rem}\label{FoldRem} 
In standard \textcolor{black}{form} problems \eqref{OurSF}, \textcolor{black}{order one} contact points are {\em fold points}. A fold point is generic if it satisfies the non-degeneracy and transversality conditions,
\begin{equation}\label{FoldNondegeneracy}
D^2_xf_0\big|_F\neq0 \qquad \text{and} \qquad D_yf_0\big|_F\neq0.
\end{equation}
Since
\[
D_x^2f_0\big|_F=D_x\langle  \textcolor{black}{\nabla} f,N\rangle\big|_F\neq0,
\]
generic fold points are order one contact points, i.e. $\sigma_F=1$. Conversely, if $F$ is an order one contact point in \eqref{OurSF}, then $\lambda(F)=D_xf_0|_F=0$ which implies $D_yf_0|_F\neq0$ (\textcolor{black}{c.f. Remark \ref{rem:S}}), and $D_x\langle  \textcolor{black}{\nabla} f,N\rangle|_F=D_x^2f_0|_F\neq0$, since $\sigma_F=1$. Hence order one contact points are generic folds in standard form problems \eqref{OurSF}.
\end{rem}

\begin{app} (Two-stroke relaxation oscillator model \eqref{SS}).
Recall that all our two-stroke oscillator models introduced are of the general form \eqref{Gen1} with $N(z)$, $f(z)$ and $G(z)$ defined in \eqref{NFG3}. The distinguishing feature is the choice of $v_0$ and $\mu(y)$ as highlighted in Table~\ref{table:models}. 
\begin{table}[t]
\centering
\begin{tabular}{| c | c | c |} 
 \hline
  model & $v_0$ & $\mu (y)$  \\
 \hline
 minimal model  \eqref{NFG1} & $1$  & $1-y$   \\
 electronic model \eqref{NFG2} & $-y_\ast$ & $x_\ast e^{-ay}$\\
  stick-slip model \eqref{NFG3}, exponential -type \eqref{ExpSSChar} & $v_0$ & $\mu_m+(\mu_s-\mu_m)e^{-a y}$  \\
  stick-slip model \eqref{NFG3}, polynomial-type \eqref{PolySSChar} & $v_0$ & $\mu_s-\frac{3(\mu_s-\mu_m)}{2v_m}y+\frac{(\mu_s-\mu_m)}{2v_m^3}y^3$  \\
 \hline
\end{tabular}
\caption{The two-stroke oscillator \eqref{SS}  for the different models.}
\label{table:models}
\end{table}
For all these models, the critical manifold is given by
\[
S=\big{\{}(x,y)\in\mathbb{R}^2\;|\;y=0\big{\}},
\]
and the set $V_0$ contains the single point
\[
p_0=\big(\mu(v_0),v_0 \big)\,.
\]
Assumption \ref{BasicAss} is satisfied for $v_0\neq0$, since all models assume $v_0>0$. The Jacobian of the layer problem at $p_0\notin S$ evaluates to
\[
Dh(p_0)=DN(p_0)f(p_0)= 
\begin{pmatrix}
{0} & {-v_0} \\
{v_0} & {-v_0\mu'(v_0)}
\end{pmatrix}.
\]
We have $\det Dh(p_0)=v_0^2>0$, and the trace is given by
\begin{equation}\label{pTrace}
\text{tr}Dh(p_0)=-v_0\mu'(v_0).
\end{equation}
The expression \eqref{pTrace} is always positive for all models except the stick-slip oscillator \eqref{SS} with polynomial-type characteristic \eqref{PolySSChar}, for which \eqref{pTrace} is positive only for $v_0\in(0,v_m)$. We restrict to values in this regime in this work (see Section \ref{3ScaleSec}). Thus the equilibrium $p_0$ is an unstable node or focus (in all cases). The non-trivial eigenvalue along $S$ is given by
\[
\langle  \textcolor{black}{\nabla} f,N\rangle\big|_{y=0}=x-\mu(0),
\]
and so the critical manifold decomposes into two normally hyperbolic branches
\[
S^a=\big{\{}(x,0)\,\big|\,x<\mu(0) \big{\}}, \qquad S^r=\big{\{}(x,0)\,\big|\,x>\mu(0)\big{\}},
\]
which are attracting and repelling, respectively. For all models, we have $\mu(0)>0$.
We also identify a single contact point
\[
F=(x_F,y_F)=(\mu(0),0),
\]
which we can classify as order one by noting that
\begin{equation}\label{Order1}
D_x\langle  \textcolor{black}{\nabla} f,N\rangle\big|_F=1\neq0 \qquad \implies \qquad \sigma_F=1.
\end{equation}
The layer problem dynamics near the contact point $F$ are sketched in Figure \ref{ExpAppSingFig1} for the case of the electronic two-stroke relaxation oscillator system \eqref{ExpSys}, which is obtained by substituting $\mu(y)=x_*e^{-ay}$ into \eqref{SS}; see Table \ref{table:models}.
\end{app}
\begin{figure}[t]
  \centering
  \includegraphics[trim={0 8.5cm 0 8.5cm},scale=.6]{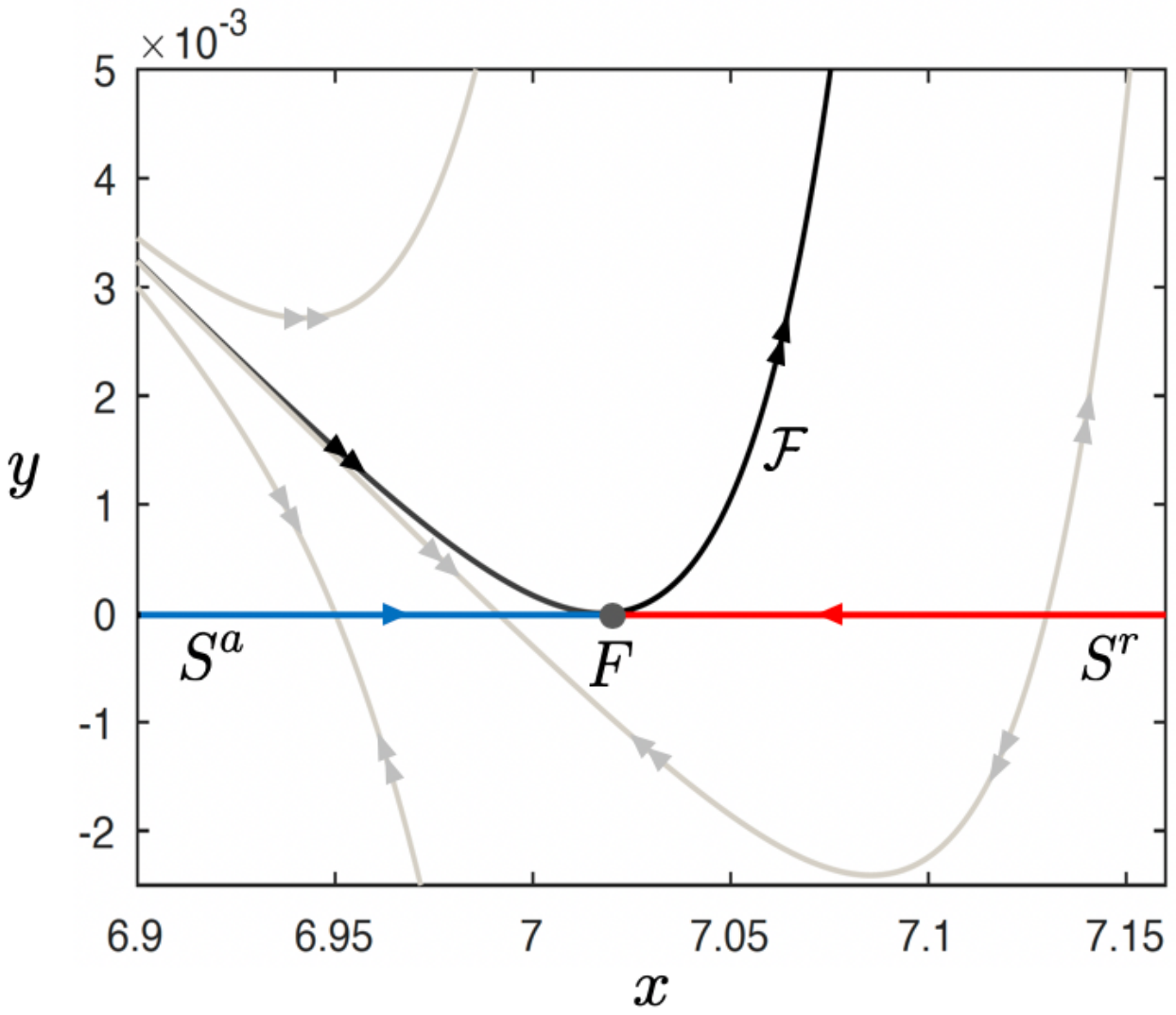}
\caption{Singular limit dynamics near the jump-off point $F$ of system \eqref{ExpSys}. Parameter values are the same as Figure \ref{EMCompFig} with $\epsilon=0$.}
\label{ExpAppSingFig1}
\end{figure}

\subsection{Reduced problem}\label{ReducedProbSec}

Consider system \eqref{Gen2} which evolves on the slow timescale $\tau$. Taking the singular limit $\epsilon\to 0$ becomes a non-trivial task:
\begin{equation}\label{LimGen2}
\dot{z}=
\lim_{\epsilon\to0}\bigg(\frac{1}{\epsilon}N(z)f(z)+G(z;\epsilon)\bigg) .
\end{equation}
Observe that in order for \eqref{LimGen2} to be well defined, the phase space must be restricted to $S$, i.e.~we need $f(z)$ to vanish, which implies that the vector field of \eqref{LimGen2} must lie in the tangent bundle $TS$ of the critical manifold $S$. 
\textcolor{black}{Note that the vector $G(z;0)\in T \mathbb R^2$ does, in general, \textit{not} lie in $TS$. Hence one must determine the component of the vector $G(z;0)$ in $TS$ to define the corresponding vector field.}

\textcolor{black}{Given a normally hyperbolic submanifold $S_n \subseteq S$,} the existence of the splitting \eqref{Splitting} provides the means to define an appropriate vector field \textcolor{black}{in $TS_n$}, because it guarantees the existence of a unique projection operator
\[
\Pi^{S_n}_N:T\mathbb{R}^2\big|_{S_n}=TS_n\oplus\mathcal{N}\to TS_n\,,
\]
i.e. $\Pi^{S_n}_N$  projects a vector with base point $z\in S_n$ along $\mathcal{N}_z$ onto $T_zS_n$; see Figure~\ref{ProjFig}.  This allows for the following definition, which is originally due to Fenichel \cite{Fenichel1979} \textcolor{black}{(see also \cite{Goeke2014})}.

\begin{defn}\label{ReducedDef} (Reduced Problem). For $z\in S_n$, the singular limit problem of system \eqref{Gen2}
is defined by
\begin{equation}\label{RP1}
\dot{z}=
\Pi_N^{S_n}\frac{\partial}{\partial \epsilon}H(z;\epsilon)\bigg|_{S_n\times\{0\}}=\Pi_N^{S_n} G(z;0)\bigg|_{S_n}\,,
\end{equation}
and it is called the reduced problem for \eqref{Gen2}.
\end{defn}

\begin{prop}\label{RPProp} 
The projection operator in \eqref{RP1} is given by
\begin{equation}\label{Proj}
\Pi_N^{S_n}=I_2-\frac{NDf}{\langle  \textcolor{black}{\nabla} f,N\rangle}\bigg|_{S_n}.
\end{equation}
An equivalent formulation of the reduced problem \eqref{RP1} is given by
\begin{equation}\label{RP2}
\dot{z}=\left[\frac{\det(N|G)}{\langle  \textcolor{black}{\nabla} f,N\rangle}
\begin{pmatrix}
{-D_yf} \\
{D_xf}
\end{pmatrix}\right]\Bigg|_{S_n},
\end{equation}
where $\det(N|G)$ denotes the determinant of the matrix with columns $N$ and $G$.
\end{prop}

\begin{figure}[h!]
	\captionsetup{format=plain}
	\hspace{0em}\centerline{\includegraphics[trim={0 6cm 0 6cm},scale=.5]{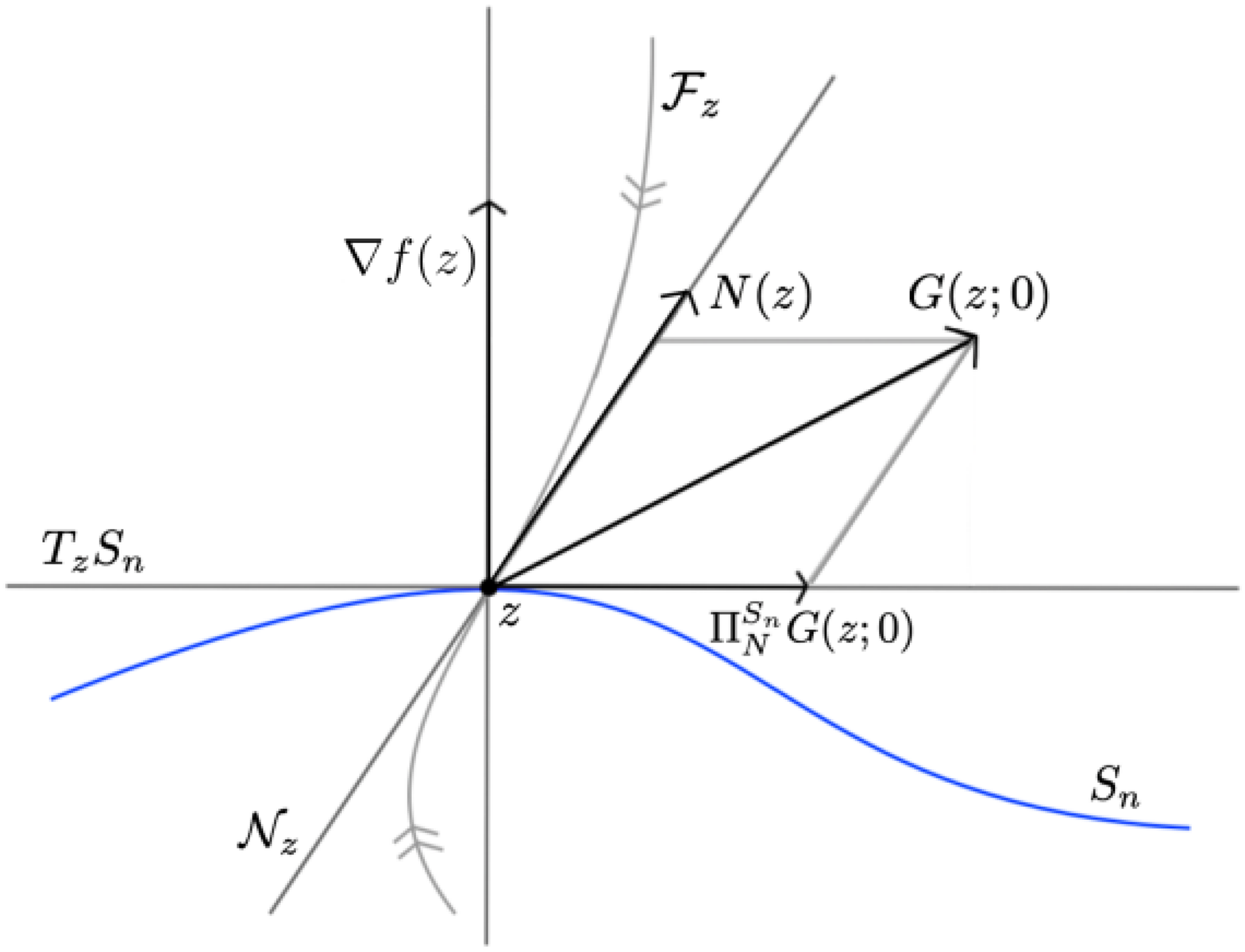}}
	\caption{Oblique projection of the vector $G(z;0)$ along $\mathcal{N}_z$ onto $\textcolor{black}{T_zS_n}$ at a normally hyperbolic point $z\in S_n$.}
	\label{ProjFig}
\end{figure}

\begin{proof} The form of the projection operator $\Pi_N^{S_n}$ reflects the definition of an oblique projection as shown in Figure \ref{ProjFig}, \textcolor{black}{i.e.~based on the splitting \eqref{Splitting}, the vector $G(z;0)$ has a unique representition
	\[
	G(z;0) = \Pi_N^{S_n} G(z;0) + \alpha N(z) 
	\]
	where $\alpha\in\mathbb{R}$ is obtained by noting that
	\[
	\frac{||\text{proj}_{\nabla f(z)} G(z;0)||}{||\text{proj}_{\nabla f(z)} N(z)||} = \alpha = \frac{\langle \nabla f(z), G(z;0) \rangle}{\langle \nabla f(z), N(z) \rangle} .
	\]
Note that $\text{proj}_{\nabla f(z)} N(z) \neq 0$ given normal hyperbolicity and Assumption~\ref{BasicAss}. Rearranging gives
	\[
	\begin{split}
	\Pi_N^{S_n} G(z;0) &= G(z;0) - \left( \frac{\langle \nabla f(z), G(z;0) \rangle}{\langle \nabla f(z), N(z) \rangle} \right) N(z) \\
	&= \left[I_2 - \frac{N(z) Df(z)}{\langle \nabla f(z) , N(z) \rangle} \right] G(z;0) ,
	\end{split}
	\]
and the expression in \eqref{Proj} follows.} Equivalence of \eqref{RP1} and \eqref{RP2} then follows from
\[
\begin{split}
\langle  \textcolor{black}{\nabla} f,N\rangle\big|_{S_n}\dot{z}&=\left[\big(\langle  \textcolor{black}{\nabla} f,N\rangle I_2-NDf\big)G\right]\big|_{S_n} \\
&=
\left[\begin{pmatrix}
{\langle  \textcolor{black}{\nabla} f,N\rangle-N_1D_xf} & {-N_1D_yf} \\
{-N_2D_xf} & {\langle  \textcolor{black}{\nabla} f,N\rangle-N_2D_yf}
\end{pmatrix}
\begin{pmatrix}
{G_1} \\
{G_2}
\end{pmatrix}\right]
\Bigg|_{S_n} \\
&=\left[\det(N|G)
\begin{pmatrix}
{-D_yf} \\
{D_xf}
\end{pmatrix}\right]\Bigg|_{S_n},
\end{split}
\]
as required.
\end{proof}

\begin{rem} For a slow-fast system in standard form \eqref{OurSF},
\[
\Pi_N^{S_n}=
\left[\begin{pmatrix}
{1} & {0} \\
{0} & {1}
\end{pmatrix}
-\frac{1}{D_xf_0}
\begin{pmatrix}
{D_xf_0} & {D_yf_0} \\
{0} & {0}
\end{pmatrix}\right]\Bigg|_{S_n}
=
\begin{pmatrix}
{0} & {-(D_xf_0)^{-1}D_yf_0} \\
{0} & {1}
\end{pmatrix}\Bigg|_{S_n},
\]
and so the reduced problem is
\begin{equation}\label{StndRP}
\begin{array}{lcl}
\dot{x}=-(D_xf_0)^{-1}(D_yf_0)g(x,y,0), \\
\dot{y}=g(x,y,0),
\end{array}
\end{equation}
with $(x,y)\in S_n$. System \eqref{StndRP} is the expression for the reduced vector field for standard form slow-fast problems\textcolor{black}{; see e.g. \cite{Kuehn2015,Jones1995}}.
\end{rem}

\begin{rem} Since $D\!f|_z\neq(0,0)$ $\forall z\in S$, it is clear from the form of system \eqref{RP2} that equilibria in the reduced problem occur if and only if $\det(N|G)=0$. Thus, the reduced dynamics on $S$ can be entirely characterised via the scalar functions $\langle  \textcolor{black}{\nabla} f,N\rangle$ and $\det(N|G)$.
\end{rem}

Note that the projection operator \eqref{Proj} is \textcolor{black}{not defined} where the splitting \eqref{Splitting} breaks down, i.e.~at contact points $F\in S$ where $S$ loses normal hyperbolicity. In order to study the reduced problem \eqref{RP1} near contact points $F$, we make a time desingularisation $d\tau=-\langle  \textcolor{black}{\nabla} f,N\rangle|_{S}\,d\bar{\tau}$, obtaining the {\em desingularised problem},
\begin{equation}\label{DesingRP}
\dot{z}
=
-\left[\big(\langle  \textcolor{black}{\nabla} f,N\rangle I_2-NDf \big)G\right]\big|_{S}=\left[\det(N|G)
\begin{pmatrix}
{D_yf} \\
{-D_xf}
\end{pmatrix}\right]\Bigg|_{S},
\end{equation}
where the overdot notation now denotes differentiation with respect to $\bar{\tau}$. The desingularised problem \eqref{DesingRP} is equivalent to the reduced problem \eqref{RP1} modulo a reversal of orientation when $\langle  \textcolor{black}{\nabla} f,N\rangle|_{S}>0$, i.e. on repelling submanifolds $S^r$ of $S$. Importantly, the desingularised problem \eqref{DesingRP} is well defined in a neighbourhood of a contact point $F$, which makes it a valuable tool for analysing the reduced problem in the case of loss of normal hyperbolicity.

\begin{defn}\label{RCPDef} (Regular Contact Point). Let $F\in S$ be a contact point with $\sigma_F=1$. $F$ is called a regular contact point if it satisfies the following equivalent conditions:
\begin{equation}\label{CPRegCond}
\det(N|G)\big|_{(F,0)}\neq0  
\qquad \iff \qquad 
\langle  \textcolor{black}{\nabla} f,G\rangle\big|_{(F,0)}\neq 0.
\end{equation}
\end{defn}

\begin{figure}[t]
\captionsetup{format=plain}
\centering
\begin{subfigure}{.5\textwidth}
  \centering
  \includegraphics[trim={9.5cm 8cm 0 8cm},scale=0.75]{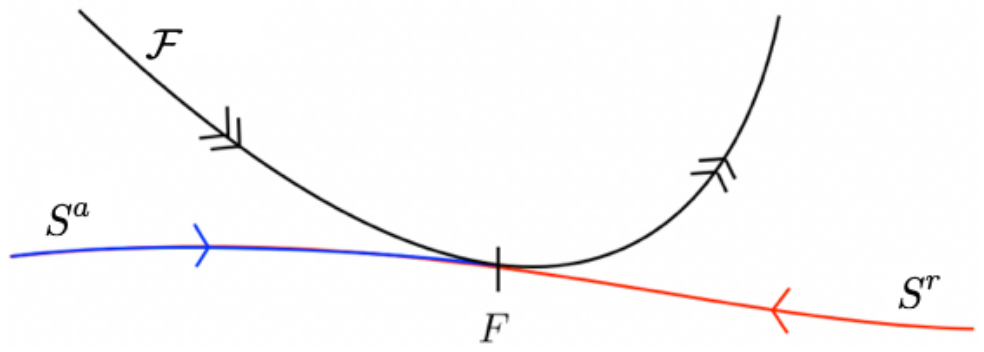}
  \caption{Jump-off}\label{JumpOff}
  \label{fig:sub1}
\end{subfigure}%
\begin{subfigure}{.5\textwidth}
  \centering
  \includegraphics[trim={9.5cm 8cm 0 8cm},scale=0.75]{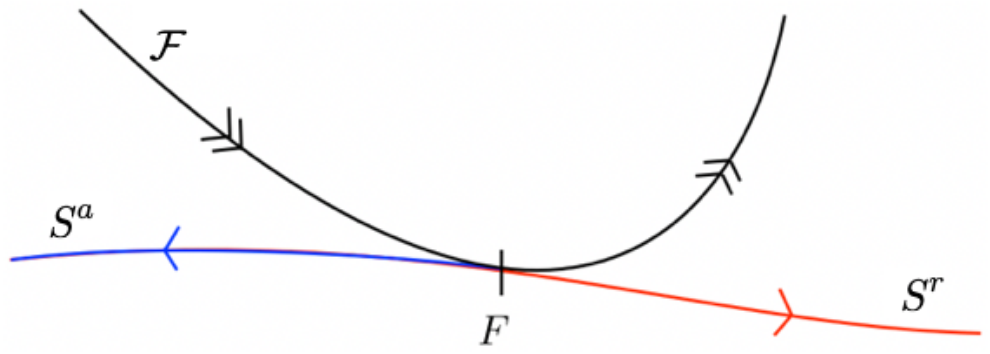}
  \caption{Jump-on}\label{JumpOn}
  \label{fig:sub5}
\end{subfigure}
\caption{The reduced flow near a regular contact point $F$ of order one.}\label{JumpFig}
\label{fig:test}
\end{figure}

A regular contact point implies that solutions of the desingularised problem \eqref{DesingRP} reach a contact point $F$ in finite (forward or backward) time which implies a finite (forward or backward) time blow-up of solutions near $F$ in the reduced problem \eqref{RP1}, i.e. solutions of the reduced problem cease to exist. Since a regular contact point is also an order one contact point, the stability property of the critical manifold changes near $F$; see Remark~\ref{rem:lambda}. Hence, the reduced flow is either towards or away from a regular contact point $F$ as shown in Figure~\ref{JumpFig}.

\begin{defn} (Jump-Off/On Point).
A regular contact point $F$ is called a jump-off point if the reduced flow is towards $F$ (see, e.g., Figure \ref{JumpOff}) or a jump-on point if the reduced flow is away from $F$ (see, e.g., Figure \ref{JumpOn}). 
\end{defn}
\W{
\begin{rem}
Jump-off points play an important role in relaxation oscillations insofar as they mark a transition from slow to fast motion, which is a defining feature for relaxation oscillations.
\end{rem}
}
\begin{rem}\label{RegFoldRem}
A point $F\in S$ in standard form problems \eqref{OurSF} is called a regular fold point if it satisfies $g|_{(F,0)}\neq0$ in addition to the fold conditions $f_0|_F=0$, $D_xf_0|_F=0$ and the nondegeneracy conditions in \eqref{FoldNondegeneracy} (see, e.g. \cite{Kuehn2015}). Hence at a regular fold point $F$ we have
\[
 g|_{(F,0)}\neq0 \qquad \implies \qquad\langle  \textcolor{black}{\nabla} f,G\rangle\big|_{(F,0)}=(D_yf_0)g\big|_{(F,0)}\neq0,
\]
i.e. regular fold points are regular contact points with $\sigma_F=1$.

Conversely, let $F$ be a regular contact point with $\sigma_F=1$ in \eqref{OurSF}. Then by analogous reasoning as Remark \ref{FoldRem} we have $D_xf_0|_F=0$, $D_yf_0|_F\neq0$ and $D_x^2f_0|_F\neq0$. Moreover, $\det(N|G)|_{(F,0)}=g|_{(F,0)}\neq0$. Hence regular contact points with $\sigma_F=1$ in system \eqref{OurSF} are regular fold points.
\end{rem}

\begin{app} (Two-stroke relaxation oscillator model \eqref{SS} continued).
For this model, we obtain a projection operator \eqref{Proj} of the form
\begin{equation}\label{ProjSS}
\Pi_N^{S}=
\begin{pmatrix}
{1} & {0} \\
{0} & {1}
\end{pmatrix}
-\frac{1}{x-\mu(0)}
\begin{pmatrix}
{0} & {v_0} \\
{0} & {x-\mu(0)}
\end{pmatrix}
=
\begin{pmatrix}
{1} & {-v_0/(x-\mu(0))} \\
{0} & {0}
\end{pmatrix}\,,
\end{equation}
and a reduced problem \eqref{RP1}, 
\begin{equation}
\begin{pmatrix}
\dot{x} \\
\dot{y}
\end{pmatrix}
= \Pi_N^{S} G(z;0)|_{S}
=
\begin{pmatrix}
{-v_0/(x-\mu(0))} \\
{0}
\end{pmatrix}
\,.
\end{equation}
Since $S$ loses normal hyperbolicity at the regular contact point $F$ for $x=\mu (0)$, we study the corresponding desingularised problem \eqref{DesingRP},
\begin{equation}
\begin{pmatrix}
\dot{x} \\
\dot{y}
\end{pmatrix}
= v_0
\begin{pmatrix}
1 \\
0
\end{pmatrix}
\,,
\end{equation}
with $v_0>0$. Thus $F$ is a regular jump-off point by Definition \ref{CPDef}, i.e.~the reduced flow is towards $F$ and a finite (forward) time blow-up of solutions occurs.
The combined layer and reduced problem dynamics near the jump-off point $F$ are shown in Figure \ref{ExpAppSingFig1} for the case of system \eqref{ExpSys}. 
\end{app}

\subsection{Local GSPT results in non-standard form}\label{LocalSec1}

\textcolor{black}{While any general slow-fast problem \eqref{Gen1} can be locally transformed to a standard form problem \eqref{OurSF} as shown in Appendix \ref{app:rectifying}, it is only a theoretical result in nature, i.e. in most cases this transformation cannot be calculated explicitly. In fact, it is not desirable to make any coordinate transformations in applications when one can apply GSPT directly to a non-standard problem \eqref{Gen1}. Thus in the following, we present local results which are well-known in standard GSPT literature in their non-standard form.}

\begin{thm}\label{FenThm} (Fenichel Theorems \cite{Fenichel1979,Fenichel1974,Fenichel1977}. See also \cite{Kuehn2015,Tikhonov1952}). Consider \eqref{Gen1} and suppose \textcolor{black}{$S_n \subseteq S$ is normally hyperbolic and compact}. Then $\exists \epsilon_0>0$ such that $\forall \epsilon\in (0,\epsilon_0)$ the following assertions are true:
\begin{itemize}
\item[(F1)] There exists a $C^r$-smooth locally invariant manifold $S_{n,\epsilon}$, called the slow manifold, which \textcolor{black}{is $\mathcal O(\epsilon)$ from $S_n$ in the Hausdorff distance}.
\item[(F2)] The flow on $S_{n,\epsilon}$ converges to the reduced flow on $S_n$ as $\epsilon\to 0$.
\item[(F3)] The manifold $S_{n,\epsilon}$ is normally hyperbolic, and exhibits the same stability properties with respect to the fast dynamics as $S_n$ (i.e. $S_{n,\epsilon}$ is attracting or repelling if $S_n$ is attracting or repelling respectively).
\item[(F4)] The manifold $S_{n,\epsilon}$ is usually not unique, but all manifolds satisfying (F1)-(F3) lie within a Hausdorff distance which is $\mathcal{O}(e^{-K/\epsilon})$ for some constant $K>0$.
\item[(F5)] Statements (F1)-(F4) also hold locally for the stable and unstable manifolds \textcolor{black}{given by the foliations}
\[
W^s_{loc}(S_n)=\bigcup_{z\in S_n}W^s_{loc}(z), \qquad  W^u_{loc}(S_n)=\bigcup_{z\in S_n}W^u_{loc}(z),
\]
\textcolor{black}{where $W^s_{loc}(z)$ and $W^u_{loc}(z)$ denote the local stable and unstable one-dimensional manifolds associated with $z\in S_n$. The manifolds $W^s_{loc}(S_n)$ and $W^s_{loc}(S_n)$} persist as manifolds $W^s_{loc}(S_{n,\epsilon})$ and $W^u_{loc}(S_{n,\epsilon})$ with foliations consisting of $C^r$-smooth leaves $W_{loc}^s(z)$ and $W_{loc}^s(z)$ with base points $z\in S_{n,\epsilon}$. In particular, we have the following:
\begin{itemize}
\item[(i)] 
\begin{equation}\label{U/SFoliations}
W^s_{loc}(S_{n,\epsilon})=\bigcup_{z\in S_{n,\epsilon}}W^s_{loc}(z), \qquad  W^u_{loc}(S_{n,\epsilon})=\bigcup_{z\in S_{n,\epsilon}}W^u_{loc}(z).
\end{equation}
\item[(ii)] The foliations in \eqref{U/SFoliations} are positively and negatively invariant respectively, i.e. $W_{loc}^s(z)\cdot t\subset W_{loc}^s(z\cdot t)$ $\forall t\geq0$ such that $z\cdot t\in S_{n,\epsilon}$, where $\cdot t$ denotes the forward evolution of $z$ in time, and similarly $W^u_{loc}(z)\cdot t\subset W^u_{loc}(z\cdot t)$ $\forall t\leq0$ such that $z\cdot t\in S_{n,\epsilon}$.
\item[(iii)] If $S_n$ is attracting and $\lambda(z)<\alpha_s<0$ $\forall z\in S_n$, there exists a constant $\kappa_s>0$ such that if $z\in S_{n,\epsilon}$ and $v\in W^s_{loc}(z)$, then
\[
\Vert v\cdot t-z\cdot t\Vert\leq \kappa_se^{\alpha_s t}
\]
$\forall t\geq0$ such that $z\cdot t\in S_{n,\epsilon}$.
Similarly, if $S_n$ is repelling and $\lambda(z)>\alpha_u>0$ $\forall z\in S_n$, there exists a constant $\kappa_u>0$ such that if $z\in S_{n,\epsilon}$ and $v\in W^u_{loc}(z)$, then
\[
\Vert v\cdot t-z\cdot t\Vert\leq \kappa_ue^{\alpha_u t}
\]
$\forall t\leq0$ such that $z\cdot t\in S_{n,\epsilon}$.
\end{itemize}
\end{itemize}
\end{thm}

\begin{rem}\label{Non-uniqueness} 
Although Theorem \ref{FenThm} actually implies existence of an entire family of slow manifolds, by \textcolor{black}{(F4)} all such slow manifolds are exponentially close in $\epsilon$, so fixing a choice of slow manifold is rarely problematic in calculations.
\end{rem}

\begin{rem}
Fenichel's original work in \cite{Fenichel1979} does not depend on the system being expressible in the standard form \eqref{OurSF}, and can be applied directly for general systems \eqref{Gen1}.
\end{rem}

In the case that normal hyperbolicity breaks down, Theorem~\ref{FenThm} no longer applies. Thus, we still require a description of the perturbed dynamics near a contact point $F\in S$. We consider here only the least degenerate case, i.e.~the dynamics near a regular contact point $F$.  Let $U_F\subset\mathbb{R}^2$ denote a neighbourhood of a regular jump-off point $F=(x_F,y_F)$, and let $\mathcal{F}$ denote the layer problem orbit segment in $U_F$ that has contact of order one with $S$ at $F$. We define $\Sigma_{in/out}\in U_F$,
\begin{equation}\label{RCPSecs}
\Sigma_{in}=\big{\{}(x_F-\rho,y)|y\in J \big{\}}\,, \qquad \Sigma_{out}=\big{\{}(x_F+\rho,y)|y\in J \big{\}}\,,
\end{equation}
and assume without loss of generality that these vertical sections are transverse to both $S$ and the layer flow for sufficiently small $\rho>0$ and a suitably defined real interval $J$; see Figure \ref{RCPFig}. 

\begin{figure}[t]
\captionsetup{format=plain}
\hspace{0em}\centerline{\includegraphics[trim={0 6.5cm 0 6.5cm},scale=.6]{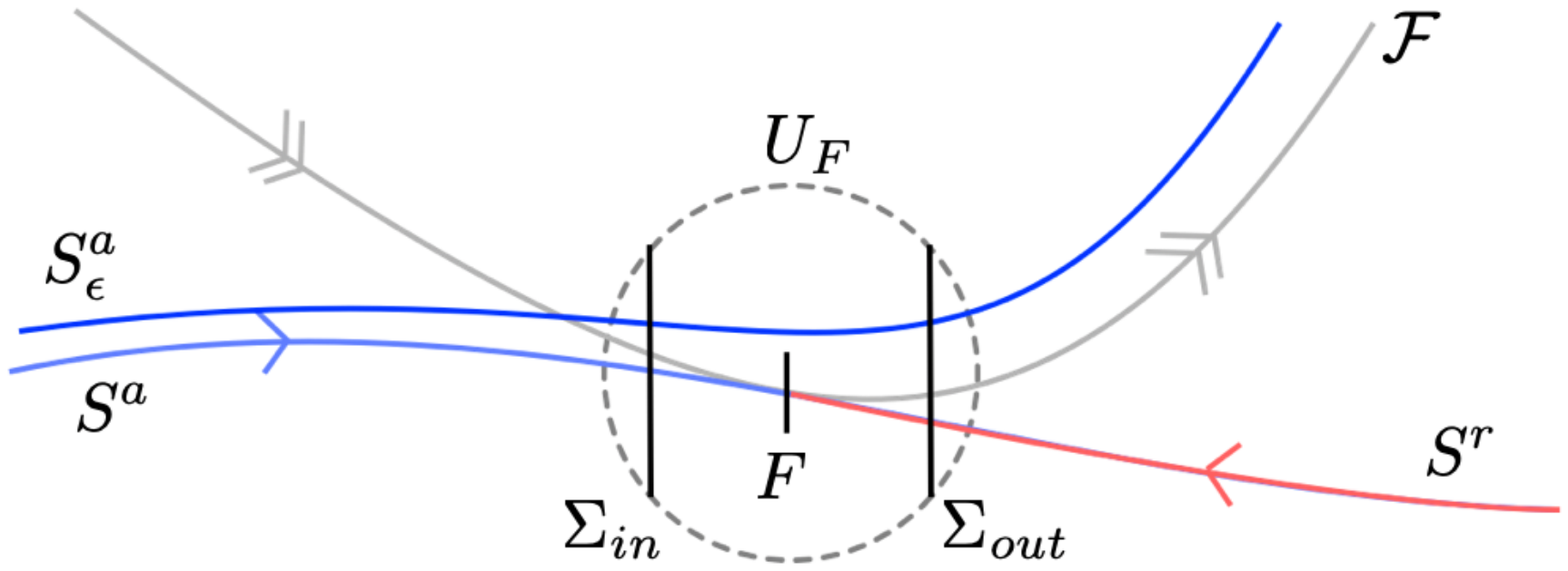}}
\caption{Behaviour near a regular contact point $F$, as described by Theorem \ref{RCPThm}.}\label{RCPFig}
\end{figure}

\begin{thm}\label{RCPThm} Let $F\in S$ be a jump-off point of system \eqref{Gen1}, and assume without loss of generality that $D_yf|_F\neq0$. Then $\exists\epsilon_0>0$ such that $\forall\epsilon\in(0,\epsilon_0]$ the following assertions hold:
\begin{itemize}
\item[(1)] The attracting slow manifold $S^a_{\epsilon}$ leaves the neighbourhood of the contact point via $\Sigma_{out}$, and
\[
\vert y_s-y_l\vert=\mathcal{O}(\epsilon^{2/3})\,,
\]
where $y_s$ respectively $y_l$ denote the $y$-coordinate of $S_\epsilon^a\cap\Sigma_{out}$ respectively $\mathcal{F}\cap\Sigma_{out}$.
\item[(2)] The transition map $\pi:\Sigma_{in}\to \Sigma_{out}$ is a contraction with contraction rate $\mathcal{O}(e^{-c/\epsilon})$, for some constant $c>0$.
\end{itemize}
\end{thm}

\begin{proof}
See Appendix \ref{CPProof}.
\end{proof}

\begin{rem}
The distance $\mathcal{O}(\epsilon^{2/3})$ in Theorem \ref{RCPThm} is consistent with the result in \cite{Krupa2001a}, where it is shown that the slow manifold $S_\epsilon^a$ leaves a neighbourhood of a regular fold point at a Hausdorff distance which is $\mathcal{O}(\epsilon^{2/3})$ from the critical fiber. The only difference between Theorem \ref{RCPThm} and the result in \cite{Krupa2001a} is that the distance $\mathcal{O}(\epsilon^{2/3})$ in Theorem \ref{RCPThm} is stated for general systems \eqref{Gen1} in terms of simple transversals in the original coordinates.
\end{rem}

\section{Existence of two-stroke relaxation oscillations}\label{ROSec}

In this section we present existence, uniqueness and stability results for two-stroke relaxation oscillations in non-standard singular perturbation problems \eqref{Gen1} which apply to our model systems \eqref{SS2}, \eqref{ExpSys}, and \eqref{SS}.

First, we note that the presence of a regular contact point and the associated finite time blow-up in the reduced problem allows one to concatenate segments of layer and reduced problems.

\begin{defn} (Reciprocal point, cf. \cite{Kaleda2011}). A point $L_z\in S$ is reciprocal to $z\in S$ if they belong to the endpoints of a heteroclinic orbit of the layer problem. We also say that the pair $(z,L_z)$ is reciprocal.
\end{defn}

\begin{rem}
For standard form problems \eqref{OurSF}, if $z=(x_z,y_z)\in S$ has a reciprocal point $L_z=(x_{L_z},y_{L_z})$, then $y_z=y_{L_z}$ since the fast fibers are parallel to the $x$-axis. For general systems \eqref{Gen1} however, the relationship between $z\in S$ and a corresponding reciprocal point $L_z$ is non-trivial. 
\end{rem}

\begin{defn}\label{SingRODef} (Singular relaxation cycle). A closed singular orbit consisting of at least one segment from the layer and the reduced problem is called a singular relaxation cycle; see Figure \ref{ROFig}.
\end{defn}

\begin{figure}[t]
\captionsetup{format=plain}
\hspace{0em}\centerline{\includegraphics[trim={0 2.5cm 0 2.5cm},scale=.35]{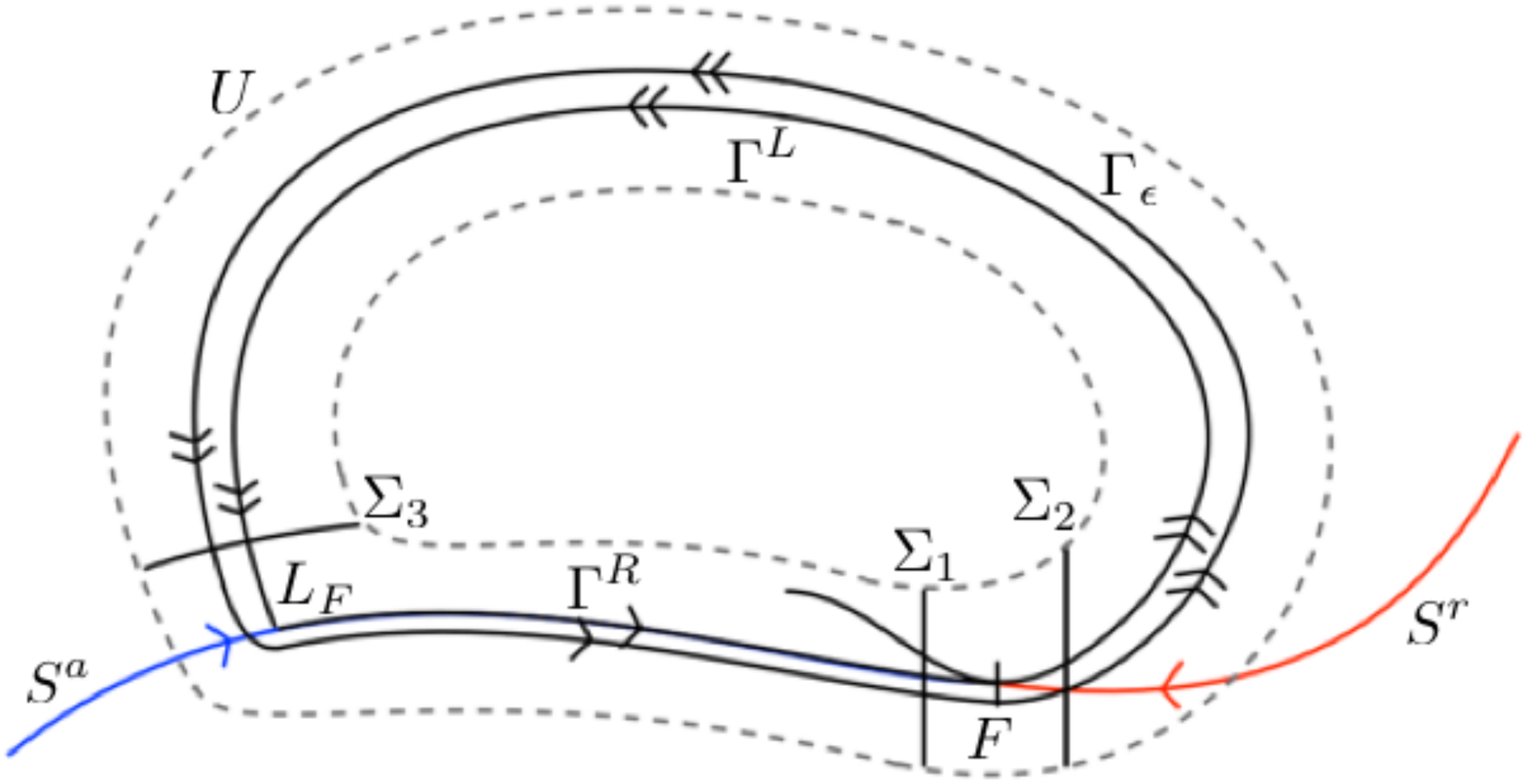}}
\caption{Two-stroke singular relaxation cycle $\Gamma=\Gamma^R\cup\Gamma^L$ for a system \eqref{Gen1} satisfying Assumption \ref{ROAss}. Also shown are the segments $\Sigma_i$ defined for the proof of Theorem \ref{ROThm}, and the relaxation oscillation $\Gamma_\epsilon$.}\label{ROFig}
\end{figure}

We are now able to provide minimal conditions on system \eqref{Gen1} for the existence of a singular two-stroke singular relaxation cycle and state the main result about the persistence of a two-stroke relaxation cycle under sufficiently small perturbations $\epsilon\ll 1$.

\begin{assumption}\label{ROAss} System \eqref{Gen1} has the following properties:
\begin{itemize}
\item[($A1$)] The critical manifold $S$ has precisely one jump-off point $F$, and therefore decomposes $S=S^a\cup\{F\}\cup S^r$, where $S^a$ ($S^r$) is attracting (repelling) and normally hyperbolic. Without loss of generality, $D_yf|_F\neq0$ is satisfied.
\item[($A2$)] The jump-off point $F$ has a reciprocal point $L_F\in S^a$, implying the existence of a singular relaxation cycle $\Gamma=\Gamma^L\cup\Gamma^R$, where the segment $\Gamma^R$ is a trajectory segment of the reduced problem from $L_F\in S^a$ to $F$ (see Figure \ref{ROFig}).
\end{itemize}
\end{assumption}

\begin{thm}\label{ROThm} \W{Given system \eqref{Gen1} under Assumptions~\ref{BasicAss} and~\ref{ROAss}}. Let $U$ denote a fixed tubular neighbourhood of the singular relaxation cycle $\Gamma$ such that $U\cap V_0=\O$. Then $\exists\epsilon_0>0$ such that $\forall\epsilon\in(0,\epsilon_0)$, system \eqref{Gen1} has a unique two-stroke relaxation cycle $\Gamma_\epsilon\subset U$. The relaxation cycle $\Gamma_\epsilon$ is attracting with Floquet exponent bounded above by $-K/\epsilon$ for some constant $K>0$, and converges to $\Gamma$ in Hausdorff distance as $\epsilon\to0$.
\end{thm}

\begin{proof}
Under Assumption~ \ref{ROAss} and assuming a sufficiently small tubular neighbourhood $U$ of the singular relaxation cycle $\Gamma$, we define cross sections $\Sigma_i\in U$, $i\in\{1,2,3\}$, as shown in Figure~ \ref{ROFig}. Let $\Pi :  \Sigma_3 \to \Sigma_3$ denote a global return map defined by the composition
$$
\Pi=\Pi_{23}\circ\Pi_{12}\circ\Pi_{31}
$$
where $\Pi_{31}:\Sigma_3\to\Sigma_1$, $\Pi_{12}:\Sigma_1\to\Sigma_2$ and $\Pi_{23}:\Sigma_2\to\Sigma_3$ are the corresponding transition maps. We show that $\Pi$ is a contraction, from which the existence of a unique stable limit cycle follows by the {\em contraction mapping principle}. 

First, consider the map $\Pi_{31}$ which describes the flow near the normally hyperbolic attracting branch $S^a$ of the critical manifold $S$. By Fenichel Theorem \ref{FenThm} (F1)-(F3), $S^a$ perturbs to an attracting slow manifold $S_\epsilon^a$. Moreover, Theorem \ref{FenThm} (F5) implies that initial conditions in $\Sigma_3$ are exponentially attracted to the slow manifold $S_\epsilon^a$ with rate $\mathcal{O}(e^{-c_1/\epsilon})$ for some $c_1>0$, and they follow their base points on $S^a_{\epsilon}$  until exiting through $\Sigma_1$. Hence an (open) interval $I_3\subset\Sigma_3$ of size $O(1)$ about $\Gamma\cap\Sigma_3$ is mapped to an interval $\Pi_{31}(I_3)\subset\Sigma_2$ of width $\mathcal{O}(e^{-c_1/\epsilon})$ about $S_\epsilon^a\cap\Sigma_1$, i.e. the map $\Pi_{31}$ is exponentially contracting.

Second, consider the map $\Pi_{12}$ which describes the flow \textcolor{black}{past} the contact point $F$ where normal hyperbolicity is lost. This transition map is covered by Theorem \ref{RCPThm} which states that an (open) interval $I_1\subset\Sigma_1$ of size $O(1)$ about $\Gamma\cap\Sigma_1$ is mapped to an interval $\Pi_{12}(I_1)\subset \Sigma_2$ of width $\mathcal{O}(e^{-c_2/\epsilon})$ about $S_\epsilon^a\cap\Sigma_2$ for some $c_2>0$,  i.e. the map $\Pi_{12}$ is also exponentially contracting.

Third, consider the map $\Pi_{23}$ which describes regular flow from $\Sigma_2$ to $\Sigma_3$ in $U$. The flow-box theorem (see e.g. \textcolor{black}{\cite{Arnold1973,Chicone1999,Hirsch2012}}) and regular perturbation theory imply that $\Pi_{23}$ is a diffeomorphism with at most algebraic growth, i.e. there exists an (open) interval $I_2\subset\Sigma_2$ of size $O(1)$ about $\Gamma\cap\Sigma_2$ such that $\Pi_{23}(I_2)\subset \Sigma_3$. 

Finally, we take the composition of the three transition maps. Since $\Pi_{31}$ and $\Pi_{12}$ are exponentially contracting while $\Pi_{23}$ has (at most) algebraic growth, the return map $\Pi(I)\subset I$ is an interval of width $\mathcal{O}(e^{-c/\epsilon})$ for some $c>0$, i.e. $\Pi$ is a contraction with rate $\mathcal{O}(e^{-c/\epsilon})$. This guarantees the existence of a unique fixed point corresponding to a unique stable limit cycle $\Gamma_{\epsilon}\in U$.

By Theorem \ref{FenThm} (F1)-(F2), $S^a_{\epsilon}$ is $\mathcal{O}(\epsilon)$ from $S^a$ and converges to $S^a$ in the Hausdorff distance as $\epsilon\to0$. By Theorem \ref{RCPThm} (i), $S_\epsilon^a\cap\Sigma_2$ is $\mathcal{O}(\epsilon^{2/3})$ from the intersection of $\Sigma_3\cap\Gamma$ and converging to $\Gamma\cap\Sigma_3$ as $\epsilon\to0$. Hence $\Gamma_{\epsilon}$ approaches $\Gamma$ in the Hausdorff distance as $\epsilon\to 0$.
\end{proof}

\begin{app}(Application of Theorem~\ref{ROThm} to two-stroke oscillator model \eqref{SS1}).
\begin{figure}[t]
\captionsetup{format=plain}
\centering
\begin{subfigure}{.32\textwidth}
  \centering
  \includegraphics[trim={3.5cm 9cm 0 9cm},scale=0.375]{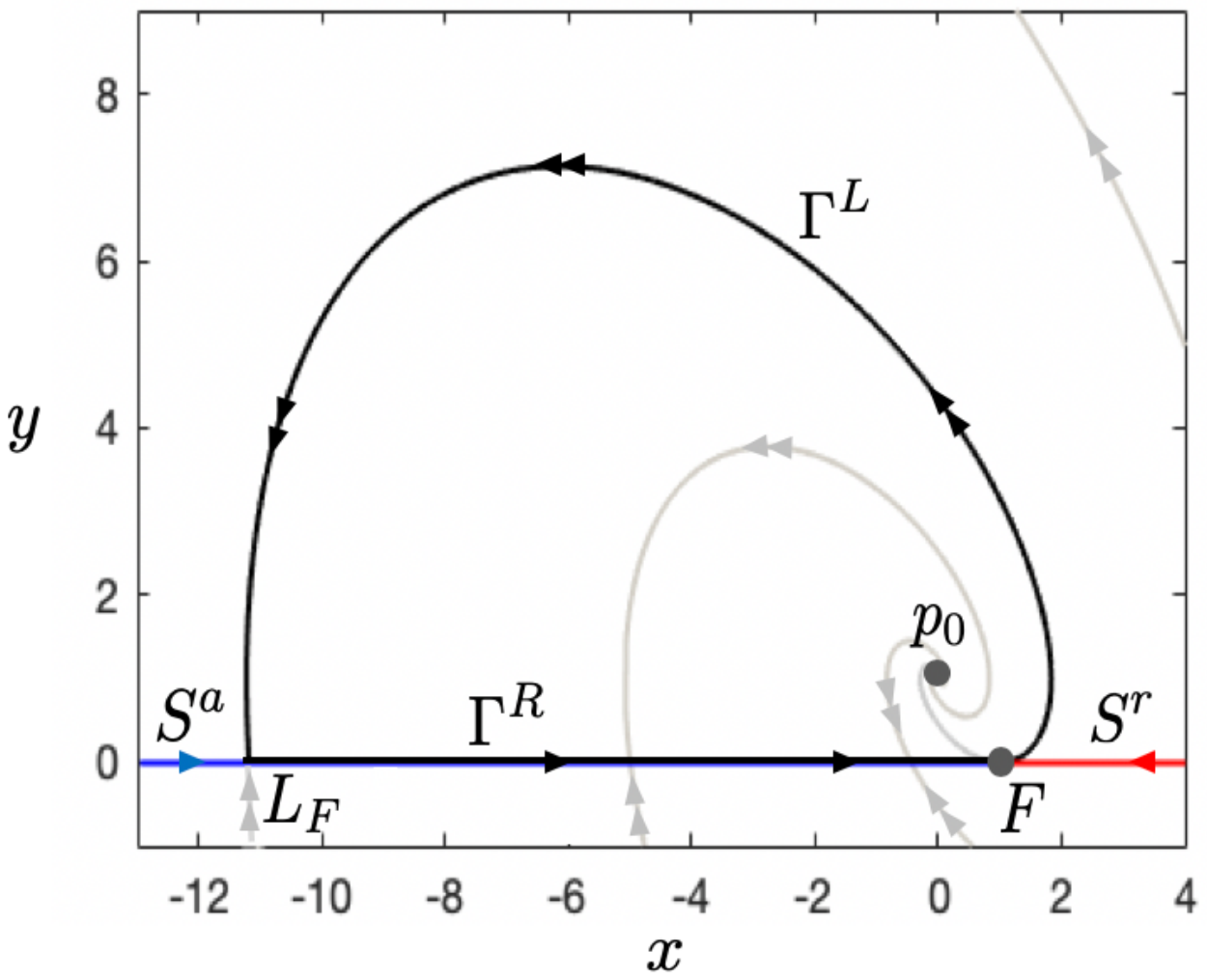}
  \caption{}\label{SingRatFig2}
  \label{fig:sub1}
\end{subfigure}%
\begin{subfigure}{.32\textwidth}
  \centering
  \includegraphics[trim={3.5cm 9cm 0 9cm},scale=0.375]{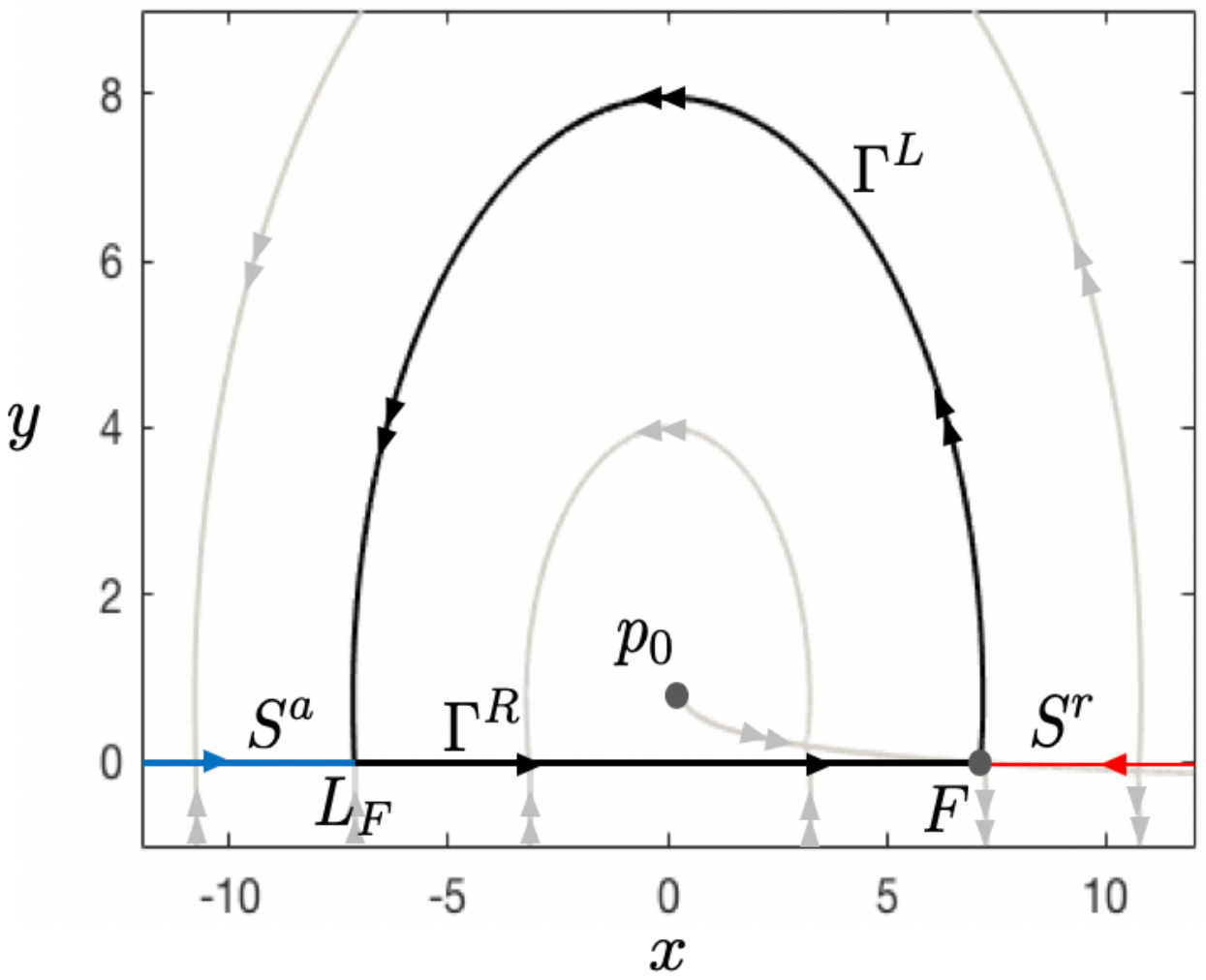}
  \caption{}\label{ExpAppSingFig}
  \label{fig:sub5}
\end{subfigure}
\begin{subfigure}{.32\textwidth}
  \centering
  \includegraphics[trim={3.5cm 9cm 0 9cm},scale=0.375]{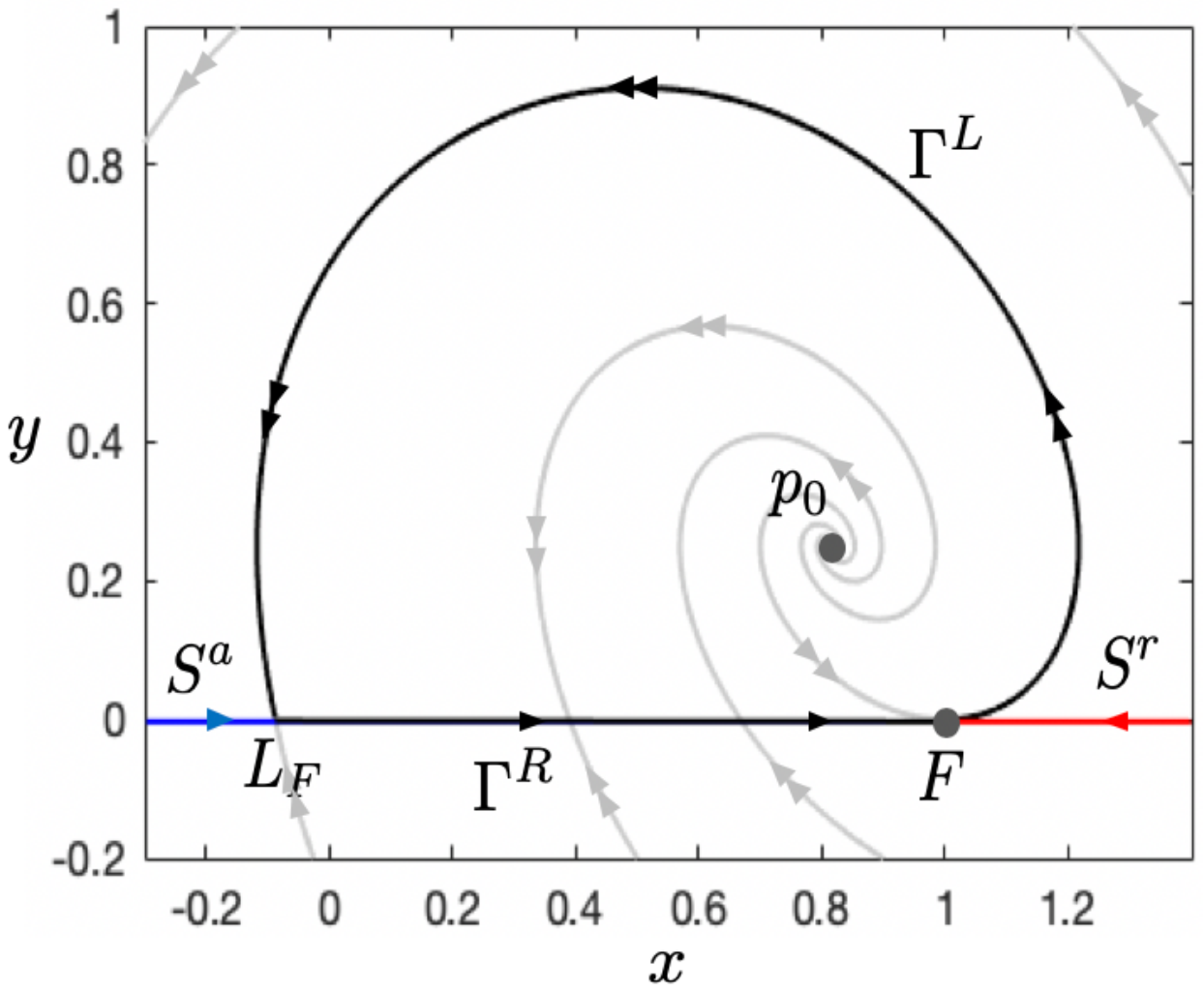}
  \caption{}\label{SSCubicSing}
  \label{fig:sub5}
\end{subfigure}
\caption{Singular limit dynamics and singular relaxation cycle $\Gamma=\Gamma^L\cup\Gamma^R$ for (a) system \eqref{SS2}; (b) system \eqref{SS} with characteristic \eqref{ExpSSChar}; (c) system \eqref{SS} with polynomial-type characteristic \eqref{PolySSChar}. Parameter values (except $\epsilon$) in (b) and (c) are the same as in Figure \ref{SSCharsCompFig}.}\label{SingRatFig}
\label{fig:test}
\end{figure}
%
System \eqref{SS1} is equivalent to system \eqref{SS2}.
The existence of a jump-off point in system \eqref{SS2} was shown in Section 3.
What is left to show is that system \eqref{SS2} has a singular relaxation cycle $\Gamma$, i.e. it remains to find a reciprocal point $L_F\in S^a$ of $F$. Note that the auxiliary layer problem of system \eqref{SS2}, $(x',y')^T=N(x,y)=(1-y,x-1+y)^T$, is linear with an unstable focus at $p_0=(0,1)$. Let $(x(t),y(t))$ denote the unique solution for this auxiliary system with $(x(0),y(0))=(1,0)=F$. Since there are no other singularities of $N(x,y)$, the expansion and rotation due to the unstable focus at $p_0$ guarantees the existence of a (unique) reciprocal point $L_F=(x(T),0)$ for some time $T>0$, where $x(T)<1$. Numerically we obtain an estimate $L_F\approx(-11.2,0)$ (Figure \ref{SingRatFig2}).
Hence we can construct a singular relaxation cycle $\Gamma=\Gamma^R\cup\Gamma^L$, where
$\Gamma^{R}=\big{\{}(x,0)\in S^a|x\in[x(T),1]\big{\}}$ and $\Gamma^{L}=\{(x(t),y(t))|t\in[0,T]\}$.
Hence Assumption \ref{ROAss} holds and by Theorem \ref{ROThm} there exists a strongly attracting two-stroke relaxation cycle $\Gamma_\epsilon$ converging to $\Gamma$ in the Hausdorff distance as $\epsilon\to0$. Thus we have proven existence and stability of the observed two-stroke oscillations in system \eqref{SS1}.
\end{app}

\begin{rem} 
\textcolor{black}{We omit the details involved with proving the existence of reciprocal points $L_F$ for system \eqref{SS} with either polynomial or exponential-type characteristics, which is clear from numerical simulations. See Figures \ref{ExpAppSingFig} and \ref{SSCubicSing}, where for our specific choice of parameters we find $L_F\approx (-6.86,0)$ and $L_F\approx (-0.09,0)$, respectively.}
\end{rem}

For completeness, we also include results on singular relaxation cycles for which $F$ is a jump-on point; see Figure~\ref{JumpOn}.

\begin{assumption}\label{ROAss2} System \eqref{Gen1} has the following properties:
\begin{itemize}
\item[($\bar{A1}$)] The critical manifold $S$ has precisely one jump-on point $F$, and therefore decomposes $S=S^a\cup\{F\}\cup S^r$, where $S^a$ ($S^r$) is attracting (repelling) and normally hyperbolic. Without loss of generality, $D_yf|_F\neq0$ is satisfied.
\item[($\bar{A2}$)] The jump-on point $F$ has a reciprocal point $L_F\in S^r$, implying the existence of a singular relaxation cycle $\Gamma=\Gamma^L\cup\Gamma^R$, where the segment $\Gamma^R$ is a trajectory segment of the reduced problem from $L_F\in S^r$ to $F$.
\end{itemize}
\end{assumption}

\begin{thm}\label{ROThm2} \W{Given system \eqref{Gen1} under Assumptions~\ref{BasicAss} and~\ref{ROAss2}}. Let $U$ denote a fixed tubular neighbourhood of the singular relaxation cycle $\Gamma$ such that $U\cap V_0=\O$. Then $\exists\epsilon_0>0$ such that $\forall\epsilon\in(0,\epsilon_0)$, system \eqref{Gen1} has a unique two-stroke relaxation cycle $\Gamma_\epsilon\subset U$. The relaxation cycle $\Gamma_\epsilon$ is repelling with Floquet exponent bounded below by $K/\epsilon$ for some constant $K>0$, and converges to $\Gamma$ in Hausdorff distance as $\epsilon\to0$.
\end{thm}

\begin{proof}
Reversing time and applying the same arguments given in the proof of Theorem \ref{ROThm} yields the desired result.
\end{proof}

\section{Comparison of two and four-stroke relaxation oscillators}\label{DiscussionSec}

We reiterate that the models presented in this work show that in general, two-stroke oscillations can occur given an $S$-shaped characteristic with two turning points (e.g.~system \eqref{SS} with polynomial-type characteristic), a characteristic with one turning point and a vertical asymptote (e.g.~system \eqref{LRDyn}), or, minimally, a $C$-shaped characteristic with a single turning point and no vertical asymptote (e.g.~system \eqref{SS2}). Thus, in general, only one turning point is necessary for two-stroke oscillation, in contrast to the four-stroke oscillation typified by the vdP oscillator, which requires two. These observations are summarised in Table \ref{24CompTab}.

\begin{table}[t]
\centering
\begin{tabular}{||c | c | c |} 
 \hline
 & Two-stroke & Four-stroke \\
 \hline
 Characteristic   & (S), (SC), or (C)    & (S) only  \\
 \# contact points &  $1$  & $2$   \\
 \# equilibria on $S$ & $0$ & $1$  \\
 \# equilibria of $N$ & $1$ & $0$  \\
 \hline
\end{tabular}
\caption{Comparison of minimal features for two and four-stroke relaxation oscillation. Here the notation ($C$), ($SC$), ($S$) refers to the shape of the characteristic: ($C$) means one turning point and no vertical asymptote (e.g. system \eqref{SS2}); ($SC$) means one turning point and a vertical asymptote (e.g. system \eqref{ExpSys}); ($S$) means two turning points (e.g. \eqref{VdPDyn} or \eqref{SS} with polynomial-type characteristic).}
\label{24CompTab}
\end{table}

We emphasise that the singularly perturbed two-stroke oscillator models presented here cannot be globally put into the standard form \eqref{SF}, i.e. they are genuine singular perturbation problems of the more general form \eqref{Gen1}. 

We would like to point out, though, that two-stroke relaxation oscillations have been studied in singularly perturbed problems in the standard form \eqref{SF} in the context of a model for aircraft-ground dynamics \cite{Rankin2011} (see also \cite{Kristiansen2017}), as well as in a model of discontinuous plastic deformation of metals \cite{Brons2005}. The characteristics in these models have one turning point and a vertical asymptote, similar to the characteristic of system \eqref{EM2} shown in Figure \ref{ExpTTraceFig1}. What distinguishes these two-stroke relaxation oscillations from those presented in this work is that the amplitude of the relaxation oscillation approaches infinity in the singular limit $\epsilon\to0$. The unbounded growth in cycle amplitude in these systems is a consequence of the global separation of slow and fast variables which implies a layer flow along straight fast fibers and, hence, no return mechanism is possible on a compact domain. Instead, the return mechanism in the case of the oscillators presented in \cite{Kristiansen2017,Rankin2011} occurs due to a `loss of normal hyperbolicity at infinity', i.e.~an alignment of the fast fibers with the vertical asymptote of the characteristic. A similar feature occurs in an autocatolator model, which exhibits three-timescale relaxation oscillations, and has been studied \textcolor{black}{using non-standard form techniques} in detail in \cite{Gucwa2009}. 

It is worth noting that system \eqref{EM2} could be analysed as a standard singular perturbation problem by artificially introducing a perturbation parameter $\epsilon\ll1$ so that \eqref{EM2} becomes
\begin{align}\label{EM3}
\begin{array}{lcl}
\dot{x}=\epsilon(-y- {y}_\ast), \\
\dot{y}=x-R_{em}(y),
\end{array}
\end{align}
where $R_{em}(y)$ is given by \eqref{EMChar2}. System \eqref{EM3} also `loses normal hyperbolicity at infinity' due to an alignment of the critical curve with the fast fibers (see \cite{Kristiansen2017} for details on how to deal with this). This approach has the advantage that the characteristic does not need to be approximated, but the disadvantage of extreme sensitivity of the cycle amplitude as a function of $\epsilon$ (as mentioned above, the amplitude tends to infinity as $\epsilon\to0$). Furthermore, one must be able to understand perturbations all the way up to $\epsilon=1$ (homotopy argument) in order to completely capture the original two-stroke oscillation observed in \eqref{EM2}.

Table \ref{24CompTab1} categorises the different approaches, showing which characteristic types allow for a standard form analysis, whether it can be undertaken on a compact domain, and those cases for which one must move beyond the standard form. Note that two-stroke relaxation oscillation in systems with a $C$-shaped characteristic cannot occur in standard form problems. 

\begin{table}[t]
\centering
\begin{tabular}{||c | c | c | c | c |} 
 \hline
 & Two-stroke (C) & Two-stroke (SC) & Two-stroke (S) & Four-stroke (S) \\
 \hline
 Stnd compact   & N  & N & N & Y  \\
 Stnd noncompact & N & Y & N & Y  \\
 Non-stnd compact & Y & Y & Y & Y  \\
 \hline
\end{tabular}
\caption{Different relaxation oscillation types, and whether or not they can be analysed as standard form problems on a compact or non-compact domain. As in Table \ref{24CompTab}, the notation ($C$), ($SC$), ($S$) refers to the shape of the characteristic}
\label{24CompTab1}
\end{table}

\subsection{Transition from two to four-stroke}

Consider a stick-slip oscillator model with a pole and polynomial-type approximation for the characteristic:
\[
\ddot{X}+F_{f,\epsilon}(v_r)+X=0, \qquad F_{f,\epsilon}(v_r)=\tilde N\bigg(-\mu_s+a_1\vert v_r\vert-a_3\vert v_r\vert^3+\frac{\epsilon}{\vert v_r\vert}\bigg),
\]
where $\tilde N$ is the (dimensionless) normal force associated with the mass. As before, $F_{f,\epsilon}(v_r)$ coincides with a common choice of dynamic friction law for $v_r\neq0$ (see, e.g. \cite{Thomsen2003,Won2016,Chen2014}), and the limit $\epsilon,\vert v_r\vert\to0$ approximates the stick phase. Restricting to $v_r<0$, setting $X=\tilde Nx$, $y=-v_r$, and applying the time desingularisation
\[
d\tau=\frac{y}{\tilde N}dt,
\]
we obtain the dynamical system
\begin{align}\label{24Trans}
\begin{array}{lcl}
x'=\delta(v_0-y)y, \\
y'=(x-\mu_s+a_1y-a_3y^3)y+\epsilon ,
\end{array}
\end{align}
where $\delta:=\tilde N^{-2}$. System \eqref{24Trans} exhibits two and four-stroke relaxation oscillation in different (limiting) regions of $(\epsilon,\delta)$-parameter space; see Figures \ref{2TransFig} and \ref{4TransFig}:
\begin{figure}[t]
\captionsetup{format=plain}
\centering
\begin{subfigure}{.5\textwidth}
  \centering
  \includegraphics[trim={2.5cm 8.5cm 0 8.5cm},scale=0.5]{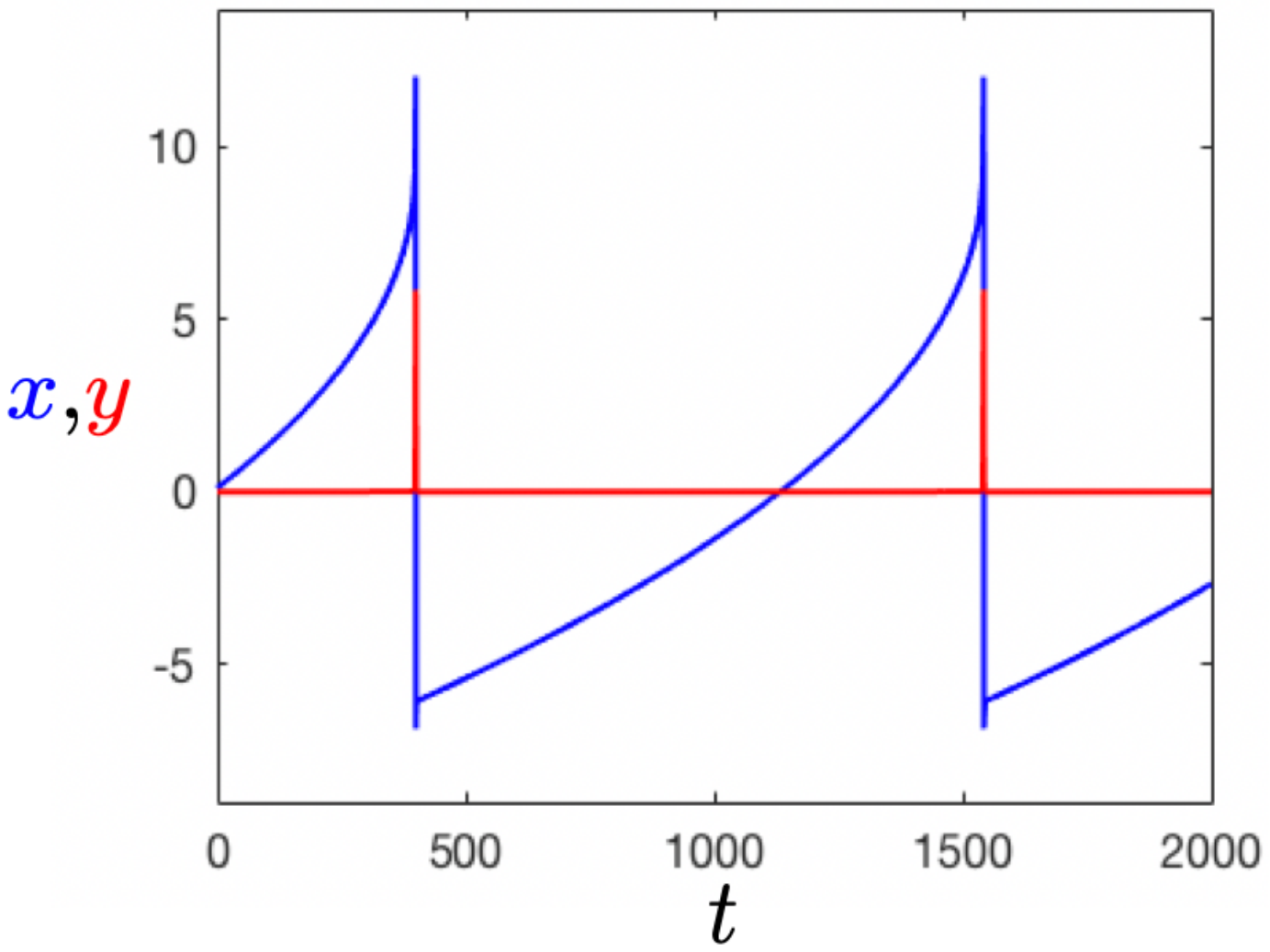}
  \caption{}\label{2STraceFig}
  \label{fig:sub1}
\end{subfigure}%
\begin{subfigure}{.5\textwidth}
  \centering
  \includegraphics[trim={2.5cm 8.5cm 0 8.5cm},scale=0.5]{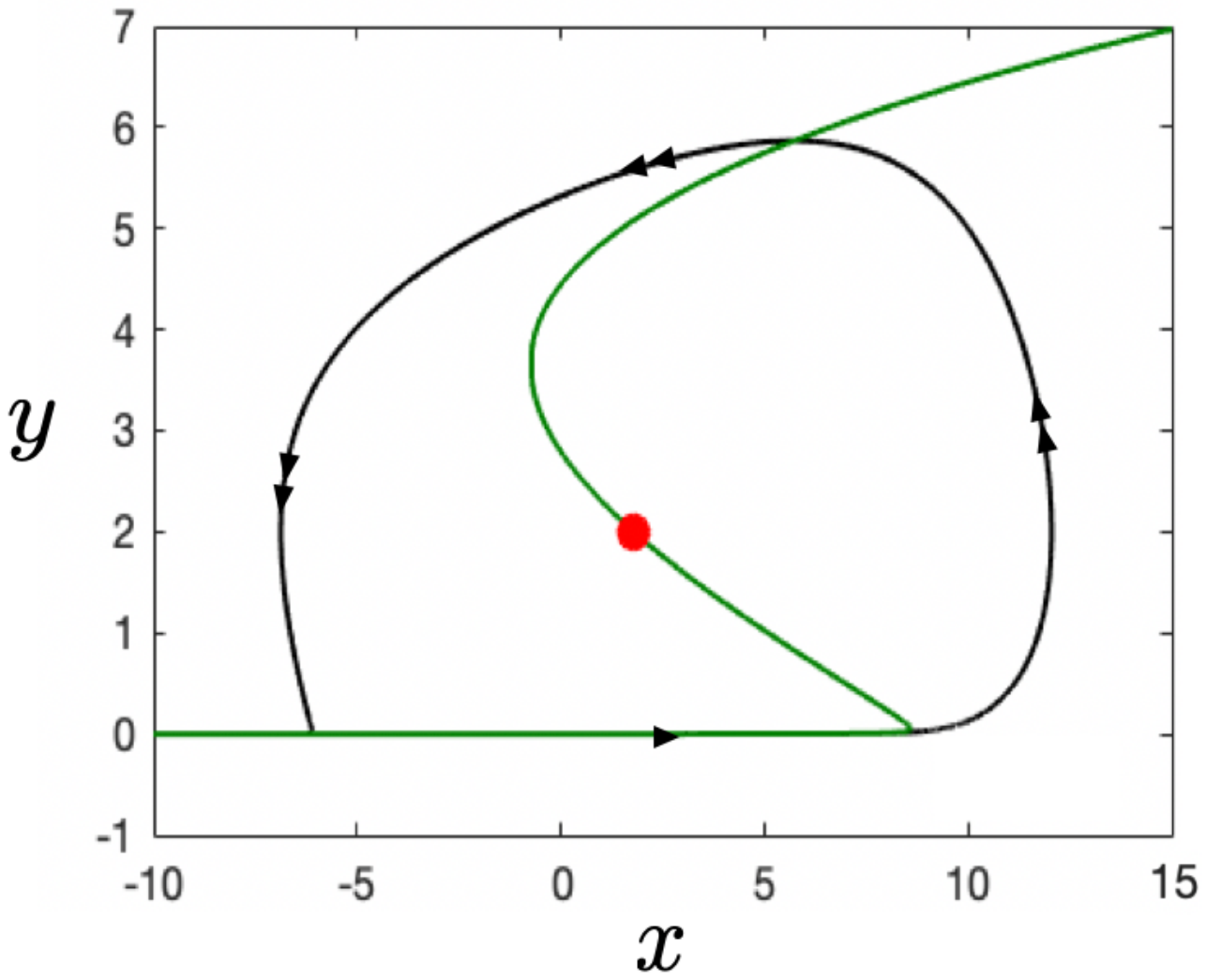}
  \caption{}\label{4STraceFig}
  \label{fig:sub5}
\end{subfigure}
\caption{Two-stroke relaxation oscillation in \eqref{24Trans} for $(\epsilon,\delta)=(10^{-2},5)$, and $(v_0,\mu_s,a_1,a_3)=(2,9,4,10^{-1})$.}\label{2TransFig}
\label{fig:test}
\end{figure}
%
\begin{figure}[t!]
\captionsetup{format=plain}
\centering
\begin{subfigure}{.5\textwidth}
  \centering
  \includegraphics[trim={2.5cm 8.5cm 0 8.5cm},scale=0.5]{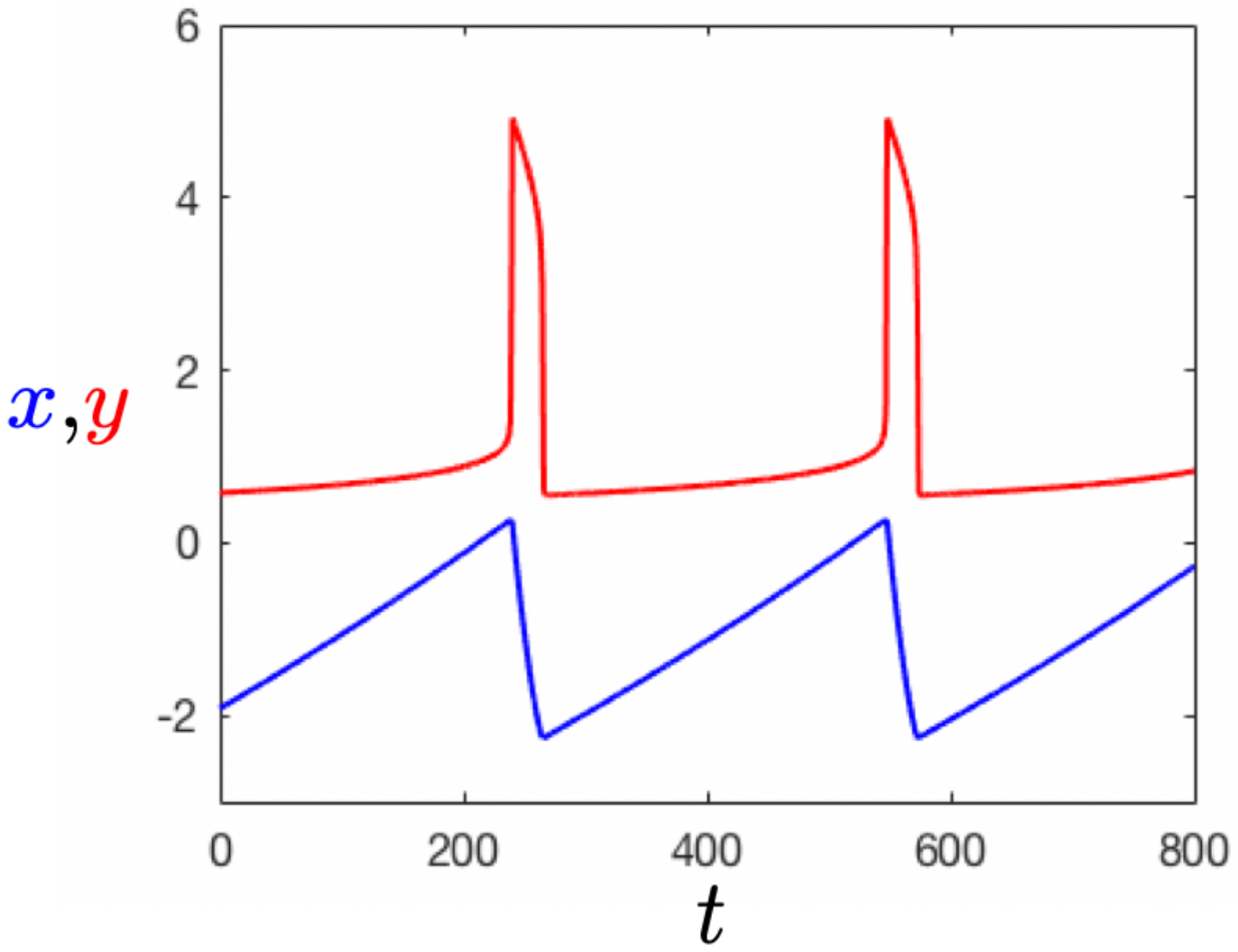}
  \caption{}\label{2STraceFig}
  \label{fig:sub1}
\end{subfigure}%
\begin{subfigure}{.5\textwidth}
  \centering
  \includegraphics[trim={2.5cm 8.5cm 0 8.5cm},scale=0.5]{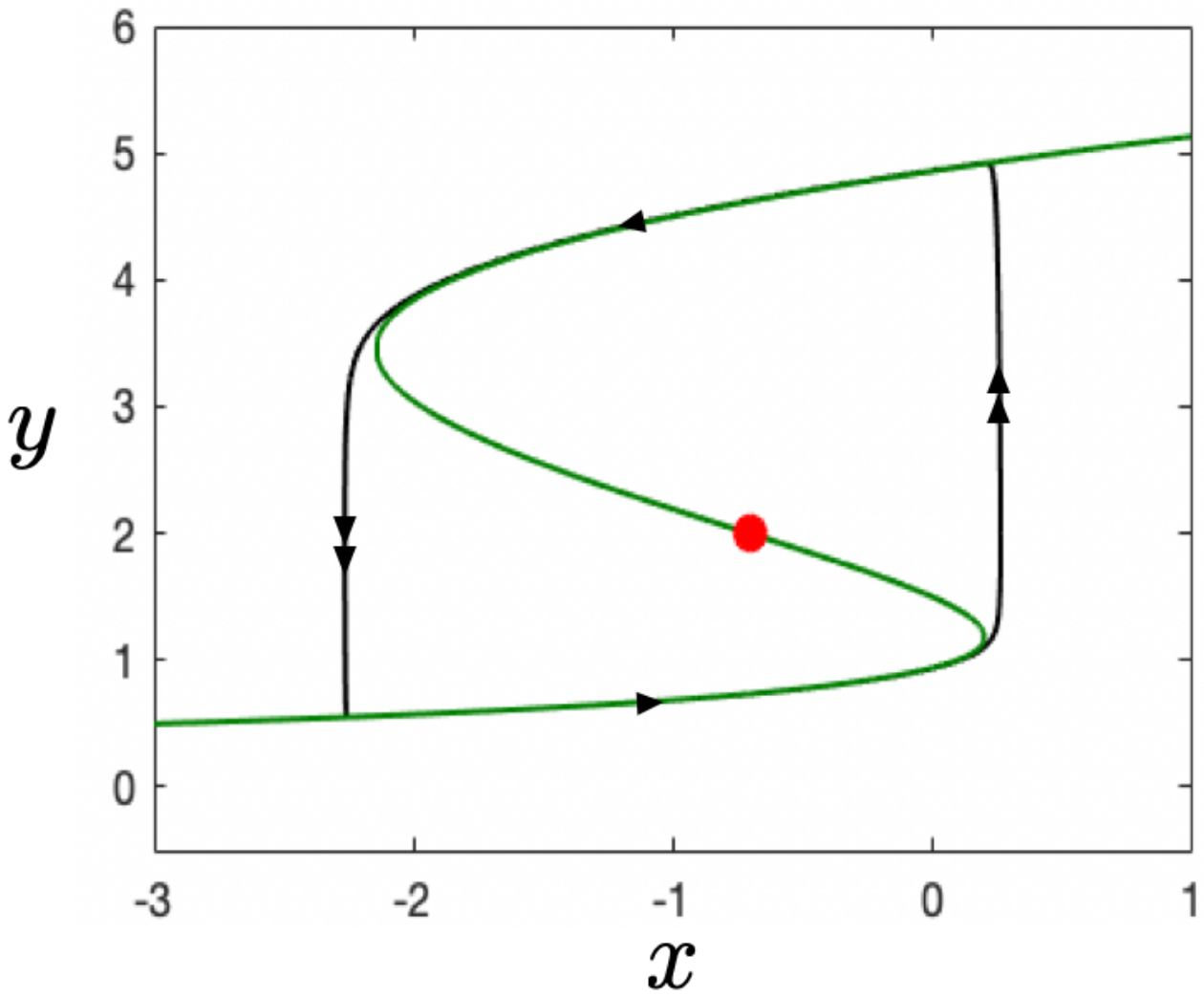}
  \caption{}\label{4STraceFig}
  \label{fig:sub5}
\end{subfigure}
\caption{Four-stroke relaxation oscillation in \eqref{24Trans} for $(\epsilon,\delta)=(5,10^{-2})$, with all other parameters the same as in Figure \ref{2TransFig}.}\label{4TransFig}
\label{fig:test2}
\end{figure}

\begin{itemize}
\item[R1.] $\epsilon\ll1$, $\delta=\mathcal{O}(1)$. Two-stroke relaxation oscillations;
\item[R2.] $\epsilon=\mathcal{O}(1)$, $\delta\ll1$. Four-stroke relaxation oscillations.
\end{itemize}
Hence, there is a transition from two- to four-stroke relaxation oscillations as one traverses a path in $(\epsilon,\delta)$-parameter space from region R1 to R2. 
This example shows that a single oscillator can exhibit both two- and four-stroke relaxation oscillations if there are multiple small perturbation parameters in the model which allows for different singular limits.

\begin{rem}
In the context of stick-slip oscillations, the case $0<\delta\ll\epsilon\ll1$ is non-physical: the limit $\epsilon\to0$ approximates a discontinuity, so one should have $\epsilon<\delta$ asymptotically.
\end{rem}

\subsection{Timescales and the two/four-stroke distinction }\label{3ScaleSec}

Consider system \eqref{24Trans} with $0<\epsilon\ll1$ and $\delta=1$. If we define the parameters $a_1$, $a_3$ as in \eqref{PolySSChar}, this system coincides with the two-stroke oscillator model \eqref{SS} with polynomial characteristic. Our analysis in Section \ref{LayerSec} showed that the equilibrium $p_0$ is stable for $v_0>v_m$, in which case oscillations are not possible. Physically, large belt speeds $v_0>v_m$ mean that the mass cannot `stick': its position stays fixed while the belt slides underneath it in an equilibrium state known as `steady sliding' \cite{Thomsen2003,Panovko1965,Chen2014}. As one decreases the belt speed, small amplitude oscillations known as `pure-slip' oscillations appear for $v_0<v_m$. These oscillations are not of stick-slip type, and exist only in a narrow parameter regime $(v_{ss},v_m)$; see e.g. \cite{Dankowicz2000}. Stick-slip oscillations occur only once the belt speed is decreased below $v_{ss}$; see Figures \ref{OscTypeFig} and \ref{delta1Fig}. The value of $v_m$ can be identified as a supercritical Andronov-Hopf bifurcation, but the value $v_{ss}$ is typically harder to identify and known only for a few specific cases. For example, the authors in \cite{Thomsen2003} show that under the assumption $\mu_s-\mu_m\ll1$,
\begin{equation}\label{ThomsenVals}
v_{ss}=\sqrt{\frac{4}{5}}v_m,
\end{equation}
for the stick-slip oscillator with characteristic \eqref{PolySSChar}.
\begin{figure}[t]
\captionsetup{format=plain}
\hspace{0em}\centerline{\includegraphics[trim={0 3cm 0 3cm},scale=.4]{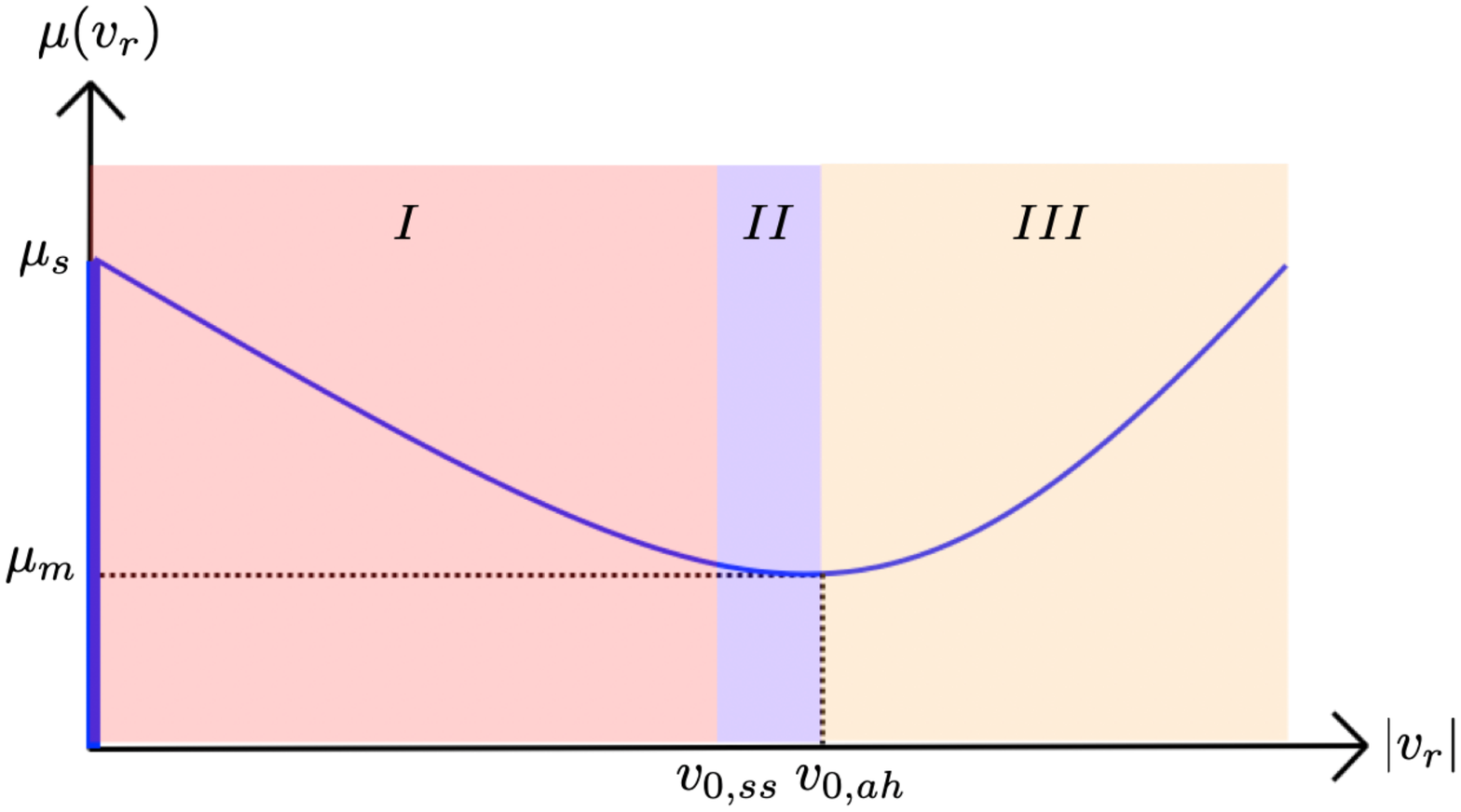}}
\caption{Behaviour for different belt speeds $v_0$. Region $I$: stick-slip oscillation; region $II$: pure-slip oscillation; region $III$: steady sliding (no oscillation).}\label{OscTypeFig}
\end{figure}

\begin{figure}[t]
\captionsetup{format=plain}
\centering
\begin{subfigure}{.32\textwidth}
  \centering
  \includegraphics[trim={2cm 7cm 0 7cm},scale=0.25]{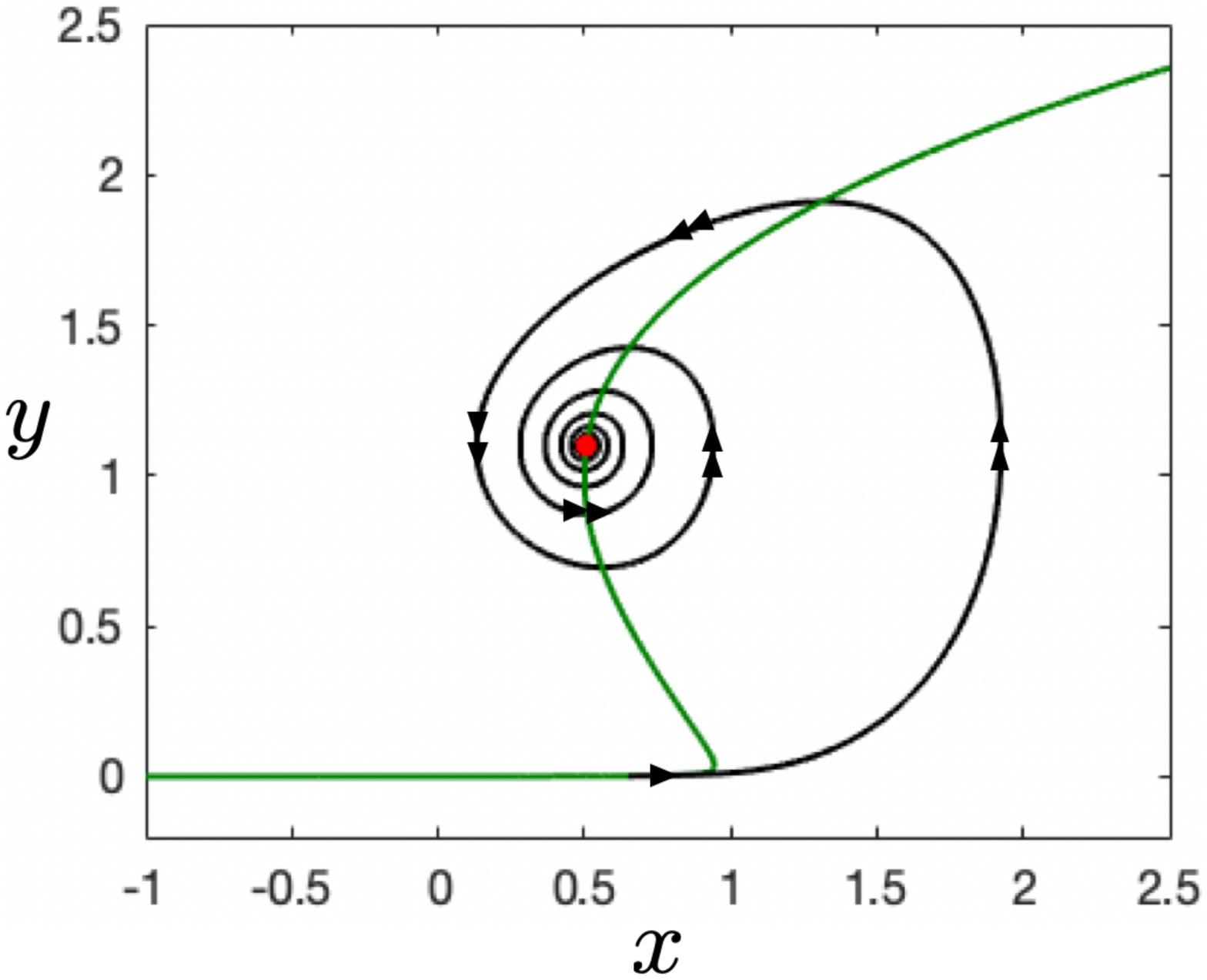}
  \caption{}
  \label{fig:sub1}
\end{subfigure}%
\begin{subfigure}{.32\textwidth}
  \centering
  \includegraphics[trim={2cm 7cm 0 7cm},scale=0.25]{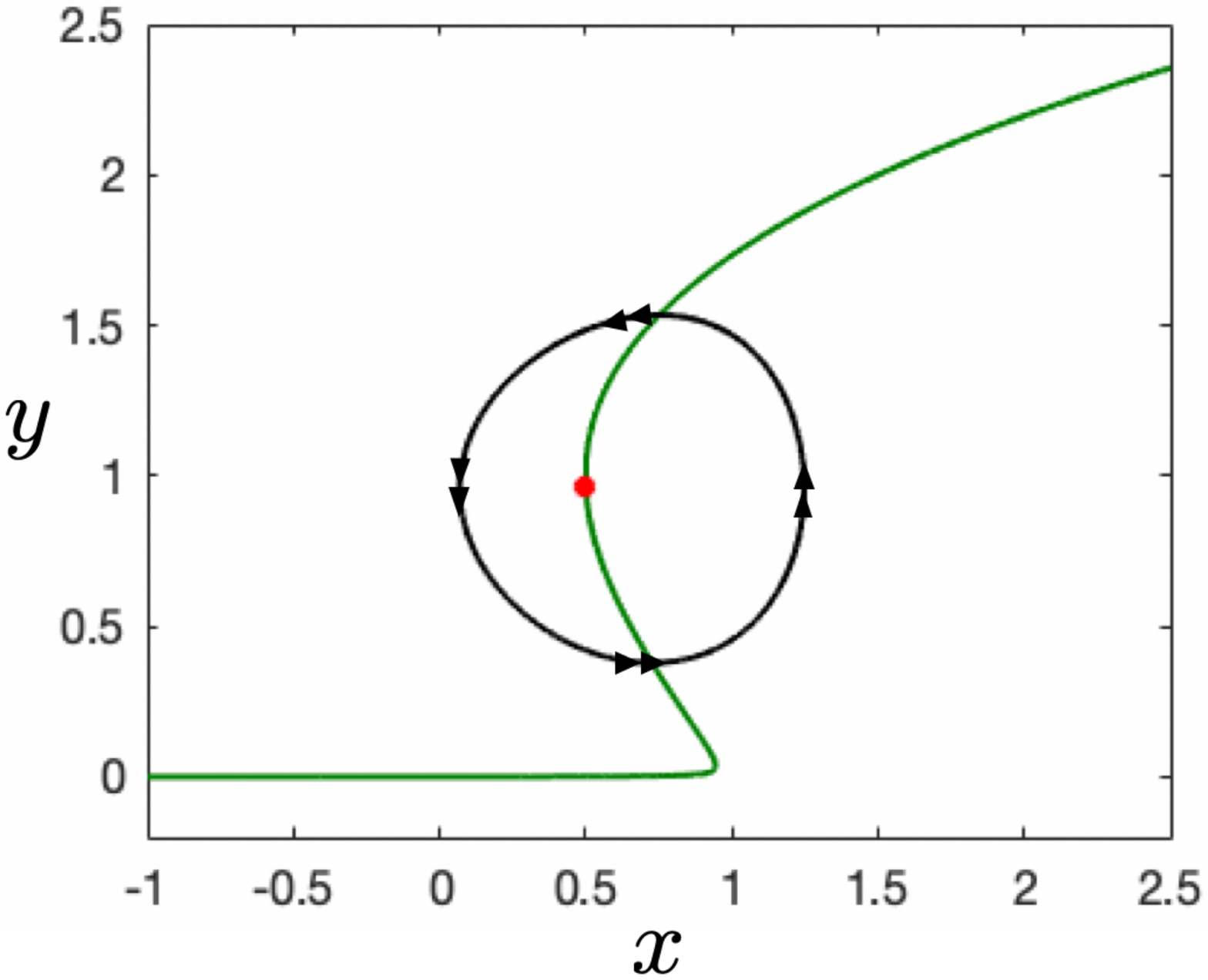}
  \caption{}
  \label{fig:sub5}
\end{subfigure}
\begin{subfigure}{.32\textwidth}
  \centering
  \includegraphics[trim={2cm 7cm 0 7cm},scale=0.25]{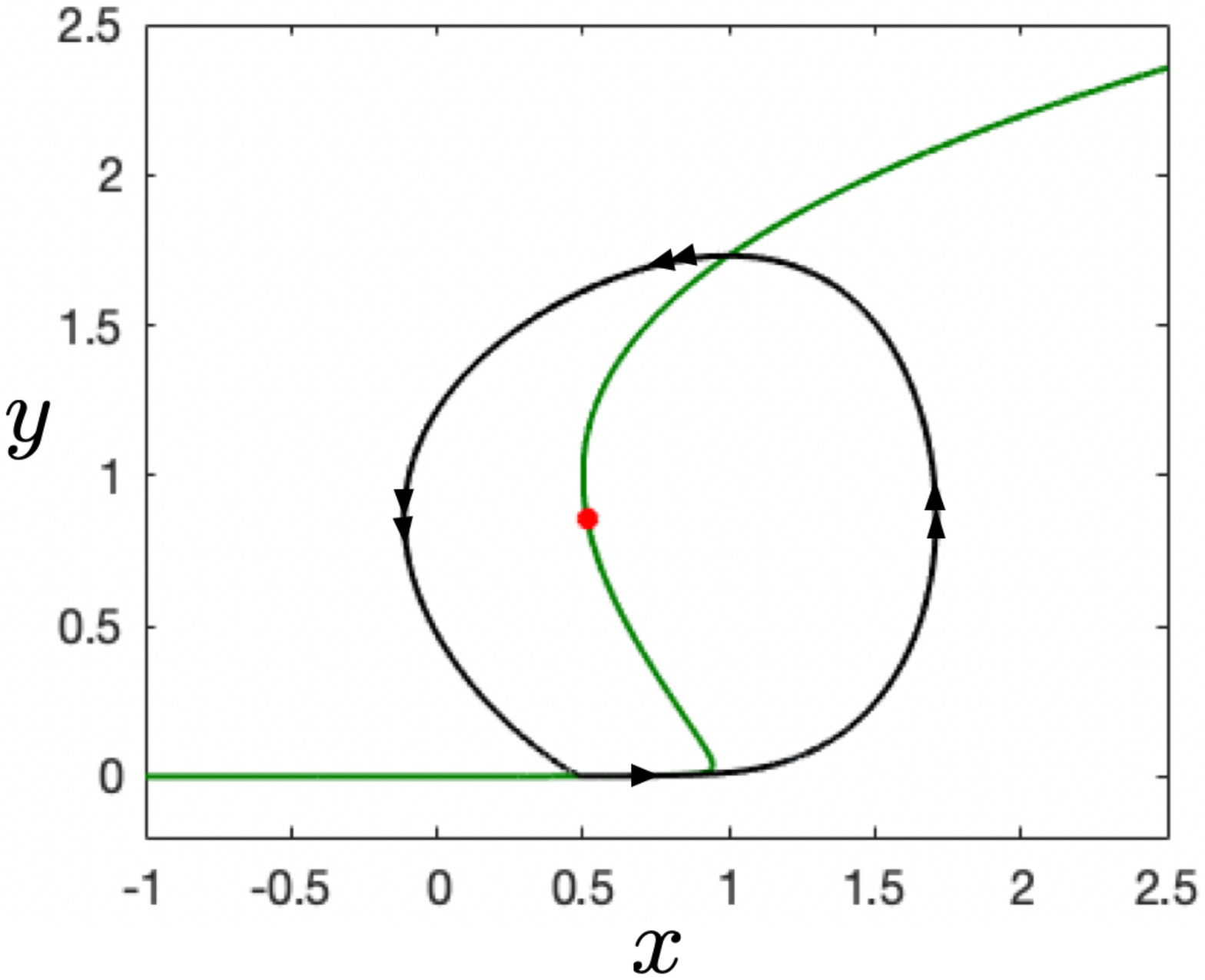}
  \caption{}
  \label{fig:sub5}
\end{subfigure}
\caption{System \eqref{24Trans} with $\delta=1$ and $(\mu_s,a_1,a_3,\epsilon)=(1,3/4,1/4,10^{-3})$. Shown are representative dynamics for different belt speeds: (a) steady sliding for $v_0=1.1$; (b) pure-slip oscillation for $v_0=0.96$; (c) stick-slip oscillation for $v_0=0.86$.}\label{delta1Fig}
\label{fig:test}
\end{figure}

\begin{figure}[t]
\captionsetup{format=plain}
\centering
\begin{subfigure}{.32\textwidth}
  \centering
  \includegraphics[trim={2cm 7cm 0 7cm},scale=0.25]{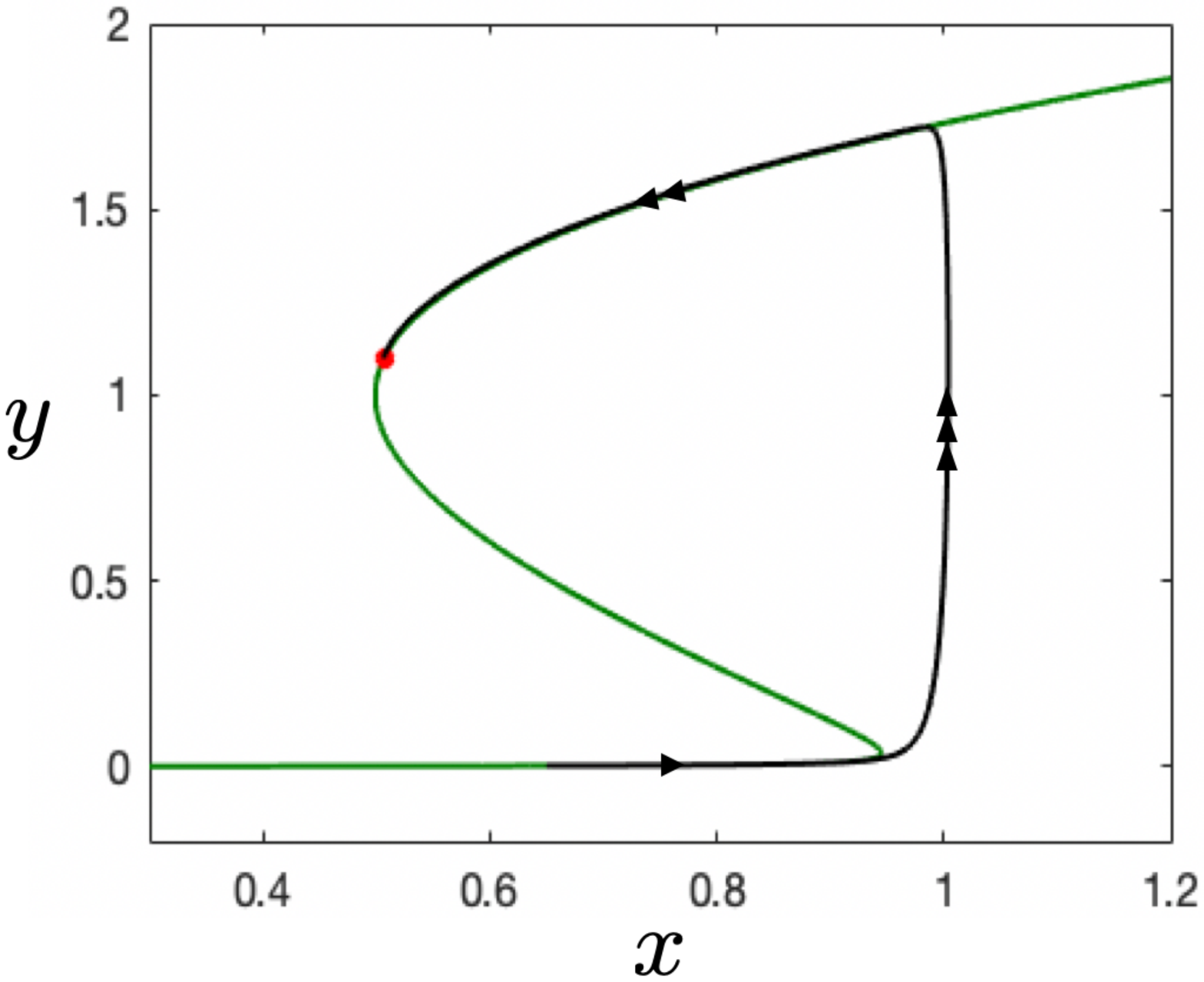}
  \caption{}
  \label{fig:sub1}
\end{subfigure}%
\begin{subfigure}{.32\textwidth}
  \centering
  \includegraphics[trim={2cm 7cm 0 7cm},scale=0.25]{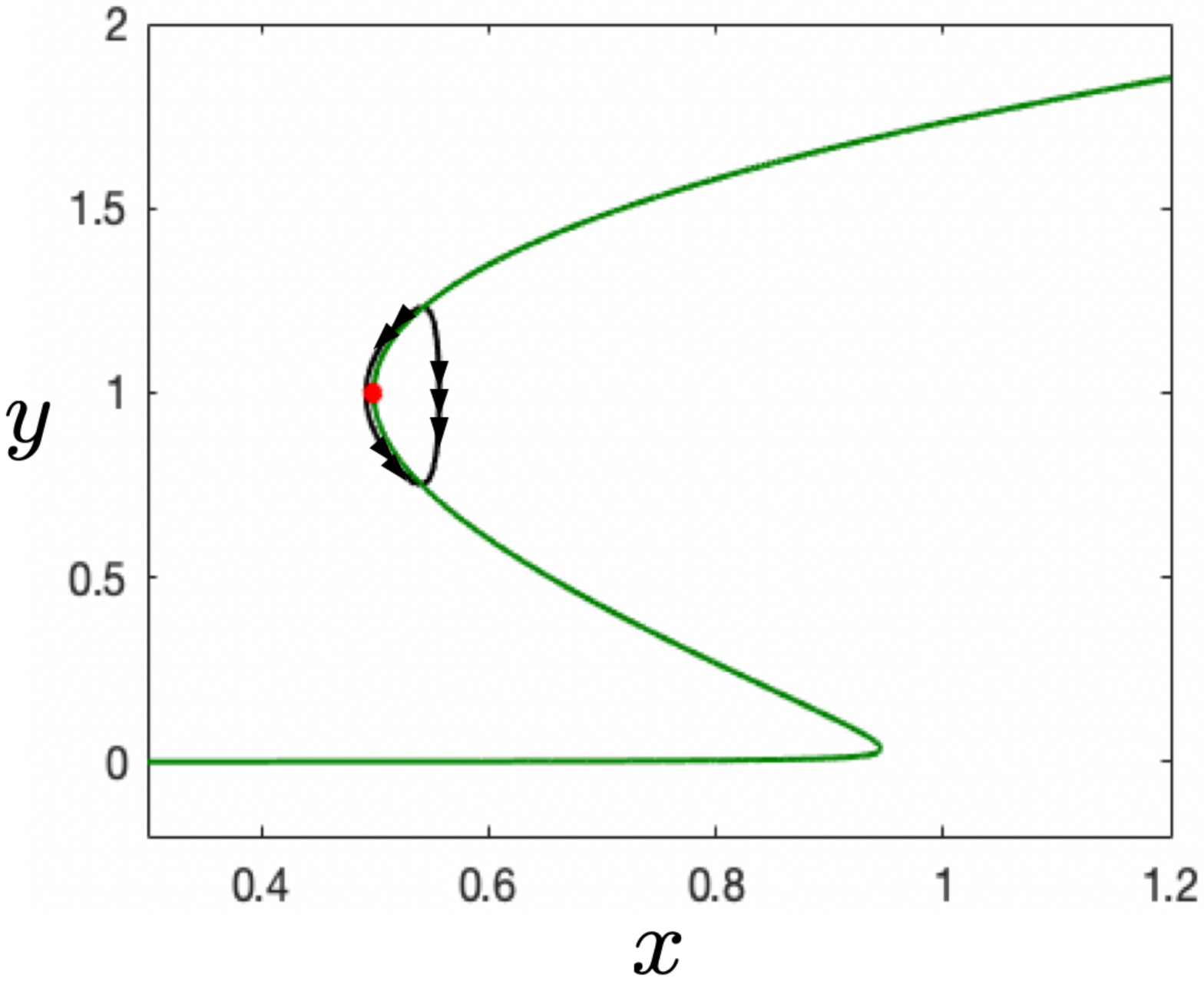}
  \caption{}
  \label{fig:sub5}
\end{subfigure}
\begin{subfigure}{.32\textwidth}
  \centering
  \includegraphics[trim={2cm 7cm 0 7cm},scale=0.25]{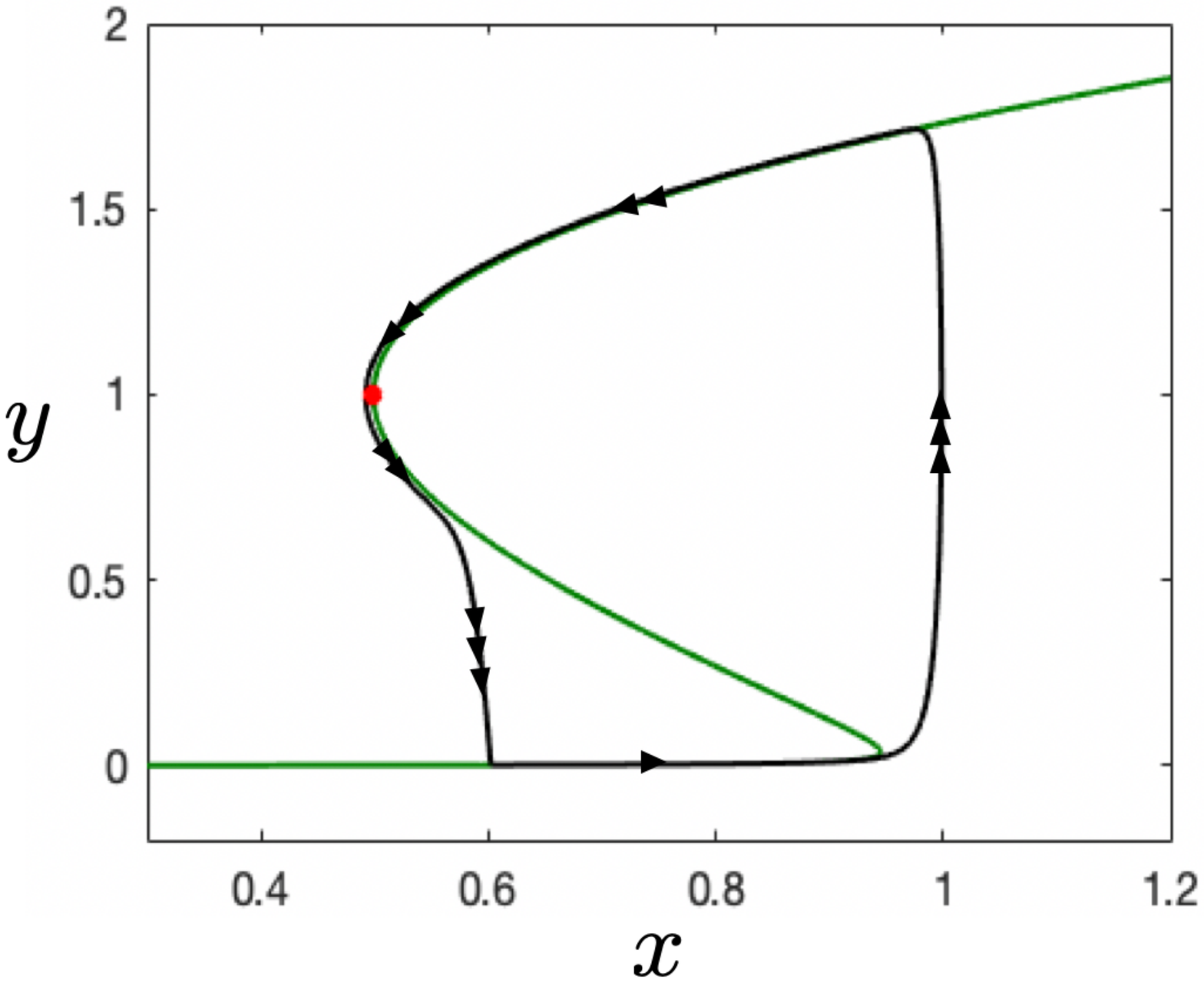}
  \caption{}\label{delta001cFig}
  \label{fig:sub5}
\end{subfigure}
\caption{System \eqref{24Trans} with $\delta=\textcolor{black}{10^{-2}}$ and $(\mu_s,a_1,a_3,\epsilon)=(1,3/4,1/4,10^{-3})$. Shown are representative dynamics for different belt speeds: (a) steady sliding for $v_0=1.1$; (b) pure-slip oscillation for $v_0=0.997166$; (c) stick-slip oscillation for $v_0=0.997165$.}\label{delta001Fig}
\label{fig:test}
\end{figure}

\

Now consider the case $0<\epsilon\ll\delta\ll1$, which corresponds physically to the case of a large normal force $(\delta=\tilde N^{-2})$. Figure \ref{delta001Fig} shows that by decreasing $\delta$ solutions begin to track the upper branch of the characteristic in phase space. The presence of two singular perturbation parameters indicates the presence of three timescales, however, which distinguishes the relaxation oscillations observed for $0<v_0<v_{ss}$ from vdP-type relaxation oscillations.

The additional timescale allows for an interpretation of the transition from steady sliding to stick-slip using standard GSPT. The layer problem obtained in the limit $\epsilon\to0$ has Jacobian at $p_0$ given by
\[
J=
\begin{pmatrix}
{0} & {-\delta v_0} \\
{v_0} & {(a_1-3a_3v_0^2)v_0}
\end{pmatrix},
\]
which has eigenvalues $\pm i\delta v_0^2$ at the Andronov-Hopf value $v_0=\sqrt{a_1/3a_3}$, i.e. there is a singular Andronov-Hopf bifurcation \cite{Krupa2001b,Dum1996,Wechselberger2015} as $\delta\to0$. Moreover, one observes a rapid decrease in the width of the interval $(v_{ss},v_m)$ as $\delta\to0$. This can be seen by comparing Figure \ref{delta001Fig}, which shows the dynamics for different belt speeds with $\delta=10^{-2}$, with Figure \ref{delta1Fig}. Pure-slip and stick-slip cases in Figure \ref{delta001Fig} are separated by a change in belt speed $v_0$ which is $\mathcal{O}(10^{-6})$. Both observations point to the existence of a \emph{canard explosion} in system \eqref{24Trans} with $0<\epsilon\ll\delta\ll1$; a dynamic phenomena characterised by a smooth transition from small-amplitude oscillations born in a singular Andronov-Hopf bifurcation to large-amplitude relaxation oscillation over an exponentially small interval in parameter space. In particular, determination of the parameter value $v_{ss}$ separating pure-slip and stick-slip solutions reduces to the problem of locating the solution which tracks the inner branch of the characteristic for the longest time.\footnote{In the language of standard GSPT, the value $v_{ss}$ corresponds to the location of the \textcolor{black}{\textit{maximal canard}}, which separates canard cycles with and without head.} A complete study of the dynamics associated with system \eqref{24Trans} constitutes future/ongoing work.

\section{Conclusion and future work}\label{ConcSec}

Two-stroke oscillation is an important phenomenon occurring across nature, as well as in a host of engineering problems. We have illustrated that two-stroke oscillators can be well described via an underlying singularly perturbed problem featuring two-stroke relaxation oscillation. Geometric singular perturbation theory can be adapted for the study of two-stroke relaxation oscillations, as we showed in the context of a number of applications deriving from the study of electronic oscillators, and mechanical oscillators with friction. Our study showed that relaxation oscillations can occur under minimal conditions which differ from those associated with the (four-stroke) vdP-type relaxation oscillations. This observation is made also in \cite{Kristiansen2017,Rankin2011,Brons2005}, where two-stroke relaxation oscillations are observed in singularly perturbed problems in standard form \eqref{SF}. The models proposed in these works describe two-stroke relaxation oscillations as perturbations of singular relaxation cycles containing segments at infinity. Our approach shows that by relaxing the requirement that the system is globally expressible in the standard form \eqref{SF} and hence allowing for non-zero curvature of the layer flow, two-stroke relaxation cycles can be described as perturbations of singular relaxation cycles residing in a compact region of phase space. Moreover, our approach is applicable to the analysis of systems like \eqref{SS2}, which cannot be understood as a standard form problem with a return mechanism at infinity.


The manuscript also raises a number of interesting questions in both theory and application. The remaining discussion focuses on a small number of these in turn. Each constitutes either current or future work.

\

{\bf Onset of stick-slip by canard explosion in a three-timescale problem.} As outlined in Section \ref{3ScaleSec}, the transition from `steady sliding'  to `stick-slip' oscillation in the model \eqref{24Trans} with $0<\epsilon\ll\delta\ll1$ appears to occur via a canard explosion, with the transition from small-amplitude sliding oscillations to large-amplitude stick-slip oscillations occurring at the transition from canard cycles without head to canard cycles with head. The observed dynamics is, more generally, a feature of all stick-slip oscillators with a characteristic $\mu(v_r)$ that has a turning point bounded away from $v_r=0$. We are interested in a complete study of the canard explosions exhibited by these systems, and in particular, with the transition from (two-timescale) canard cycles without head to (three-timescale) canard cycles with head; see Figure~\ref{delta001Fig}.

\

{\bf More general friction forces \textcolor{black}{and exponential nonlinearities}.} A natural progression from this work is to study two-stroke oscillators for which the friction has `rate-and-state' dependence. Such a generalised friction force $F_f(x,v_r,t)$ models rate-and-state dependence in a wide range of physical phenomena and engineering applications \cite{Berger2002,Dankowicz2000}. From a modelling point of view, incorporating rate dependence in particular leads one into the realm of non-autonomous dynamical systems (`open' models). If the friction force changes only slowly in time then singular perturbation techniques can still provide answers to understand the underlying non-autonomous dynamics; see e.g.~\cite{Wechselberger2013}.
 
In addition to the difficulties associated with non-autonomous dynamical systems, many of the existing models involving exponential or logarithmic type (friction) terms lead, in the context of singularly perturbed problems, to more general problems surrounding the difficulties associated with the breakdown of normal hyperbolicity at greater than algebraic rates. Significant progress has been made in this area by \textcolor{black}{in \cite{Kristiansen2017,Bossolini2017b,Kristiansen2019b}} by adapting the blow-up method in order to deal with such difficulties. \textcolor{black}{These techniques are also applicable in the study of certain two-stroke oscillations with exponential nonlinearities occurring in the aformentioned electrical oscillators originally introduced by Hester \cite{Hester1968} which is ongoing work.}

\

{\bf Canards and bifurcations two-stroke relaxation oscillators.} In the case of vdP-type (four-stroke) relaxation oscillation, the mechanism responsible for the onset of relaxation oscillation is a canard explosion, i.e. a rapid transition from small to large amplitude relaxation oscillations under an exponentially small variation in parameter space. Canard explosion in the vdP oscillator with constant forcing is well understood \cite{Krupa2001b}, and known to correlate with the passage of an equilibrium over a fold under additional parameter variation. Although the canard explosion itself is a global phenomenon, it is driven by the local dynamics, in particular the occurrence of a (singular) Andronov-Hopf bifurcation, and the presence of small (and large) canard cycles.

Given the local equivalence between general singularly perturbed problems \eqref{Gen1} and standard form problems \eqref{SF} discussed in Section \ref{LocalSec1}, then, a natural question is the following: what is the global effect of introducing analogous local dynamics in the case of a two-stroke relaxation oscillator? \W{Canard explosions observed in general systems \eqref{Gen1} capable of generating two-stroke relaxation oscillation may differ qualitatively from the vdP-type canard explosion in a number of ways (see, e.g., \cite{Wechselberger2019})}: among other distinguishing features, canard cycles can have arbitrarily large amplitudes which exceed that of the relaxation cycle significantly, and the onset of stable two-stroke relaxation oscillation is possible only if the associated (singular) Andronov-Hopf bifurcation is subcritical. In addition to the study of canard explosion under `minimal' conditions, we \W{currently} consider the possibilities for global dynamics associated with two-stroke relaxation oscillators with up to two additional equilibria on the slow manifold. The occurrence of a local (singular) Bogdanov-Takens bifurcation (see \cite{Wechselberger2015,Maesschalck2011b}) in this scenario leads to a wide variety of possibilities for global bifurcations and dynamics.

\

{\bf \textcolor{black}{Regularised boundary equilibrium bifurcation in stick-slip.}} \textcolor{black}{The bifurcations discussed in Section \ref{3ScaleSec} are known to occur in general stick-slip problems for belt speeds $v_0$ bounded uniformly away from zero. Bifurcations occurring in a neighbourhood of the regularised region $v_r = 0$ are necessarily more complex, and are associated with \textit{boundary equlibrium bifurcations} in the piecewise-smooth (PWS) literature \cite{Bernardo2008,Jeffrey2018,Kuznetsov2003}, which occur when an equilibrium collides with a discontinuity surface under variation of a parameter. In an ongoing work, we consider the case of a (regularised) `boundary focus' bifurcation which occurs in stick-slip problems as $v_0 \to 0^+$. In the PWS case, it is known that the amplitude of stick-slip cycles shrinks to zero and terminates at the point of collision. In the regularised analogue, we identify an Andronov-Hopf bifurcation as the bifurcation responsible for the termination of oscillations for $0 < \epsilon \ll1$, and describe the growth of these cycles to order one stick-slip type cycles observed for belt speeds $v_0$ bounded uniformly away from zero.}

\

{\bf The general framework for $\mathbb{R}^n$.} The framework for GSPT applied to general systems \eqref{Gen1} developed in this work extends beyond the study of planar two-stroke oscillations. In particular, the framework presented in this work generalises to higher dimensional singular perturbation problems in the general form
\begin{equation}\label{GenRn}
z'=N(z)f(z)+\epsilon G(z,\epsilon), \qquad z\in\mathbb{R}^n, \qquad 0<\epsilon\ll1,
\end{equation}
where $N(z)$ denotes a matrix of dimension $n\times (n-k)$ and $f(z)$ is a $(n-k)$-dimensional (vector-valued) function, and the existence of an $k$-dimensional \textcolor{black}{ regularly embedded submanifold} $S=\{z \in \mathbb{R}^n|f(z)=0\} \subset \mathbb R^n$ is assumed, $1\le k<n$. This work is developed in \cite{Wechselberger2019}.

\section*{Acknowledgments}
This work was supported by the Australian Research Council DP180103022 grant. \textcolor{black}{The authors would also like to thank Peter Szmolyan, Ilona Kosiuk, Kristian Uldall Kristiansen and the referees for helpful comments on an early version of the manuscript. In particular, we would like to acknowledge the work of Peter Szmolyan and Ilona Kosiuk as motivation for this work.}

\bibliographystyle{vancouver}
\bibliography{bibnew}

\begin{appendix}

\section{Local transformation of system \eqref{Gen1} to standard form}\label{app:rectifying}

\W{Due to the existence of a local fibration in a sufficiently narrow tubular neighbourhood $\mathcal B$ of $S$, the study of general slow-fast systems \eqref{Gen1} is related to the study of standard form problems \eqref{OurSF} by local (topological) equivalence. This is emphasised in the following theoretical result, which provides the means for locally rectifying the fibers near $S$.

\begin{lem}\label{GenToStnd} Given system \eqref{Gen1} with critical manifold $S$. Then there exists a local coordinate transformation such that system \eqref{Gen1} takes the standard form
\begin{equation}\label{SF2}
\begin{pmatrix}
{x'} \\
{v'}
\end{pmatrix}
=
\begin{pmatrix}
{1} \\
{0}
\end{pmatrix}
\tilde{f}(x,v)+\epsilon
\begin{pmatrix}
{g_1(x,v,\epsilon)} \\
{g_2(x,v,\epsilon)}
\end{pmatrix}
\end{equation}
in a local tubular neighbourhood $\mathcal{B}$ of $S$.
\end{lem}

\begin{proof}
Consider system \eqref{Gen1} with coordinates $z=(x,y)^T$. Assume without loss of generality that fast fibers in a local tubular neighbourhood $\mathcal{B}$ of $S$ can be described as level sets $L(x,y)=c\in I\subset\mathbb{R}$ for some smooth real-valued function $L(x,y)$ with $D_yL(x,y)|_{\mathcal B}\neq0$, i.e. each fast fiber can be locally written as a graph $y=Y^c(x)$.
We define a new local coordinate $v=L(x,y)$ which has a locally well-defined inverse $y=\tilde{M}(x,v)$, and obtain the standard form system
\[
\begin{pmatrix}
{x'} \\
{v'}
\end{pmatrix}
=
\begin{pmatrix}
{1} \\
{0}
\end{pmatrix}
N_1(x,\tilde{M}(x,v))f(x,\tilde{M}(x,v))+\epsilon
\begin{pmatrix}
{ G_1(x,\tilde{M}(x,v),\epsilon)} \\
{\langle DL(x,\tilde{M}(x,v)),G(x,\tilde{M}(x,v),\epsilon)\rangle}
\end{pmatrix},
\]
where we have used the fact that $\langle DL,N\rangle|_{\mathcal B}=0$ is the defining condition for a local fast foliation.
\end{proof}
}

\section{Proof of Theorem \ref{RCPThm}}\label{CPProof}

Without loss of generality, assume that $F=(x_F,y_F)=(0,0)$, and $D_yf|_F>0$ (one can always ensure this by adjusting $N$ if necessary). Making the preliminary transformation $u=f(x,y)$ and denoting the inverse of the transformation by $y=M(x,u)$ as in proof of Lemma \ref{FlatSLem}, we obtain the system
\begin{equation}\label{NGTilde1}
\begin{split}
\begin{pmatrix}
{x'} \\
{u'}
\end{pmatrix}
&:=
\tilde{N}(x,u)u+\epsilon\tilde{G}(x,u,\epsilon) \\
&=
\begin{pmatrix}
{N_1(x,M(x,u))} \\
{\langle  \textcolor{black}{\nabla} f(x,M(x,u),N(x,M(x,u))\textcolor{black}{)}\rangle}
\end{pmatrix}
u+\epsilon
\begin{pmatrix}
{G_1(x,M(x,u),\epsilon)} \\
{\langle  \textcolor{black}{\nabla} f(x,M(x,u)),G(x,M(x,u),\epsilon)\rangle}
\end{pmatrix},
\end{split}
\end{equation}
and the sections $\Sigma_{in/out}$ defined in \eqref{RCPSecs} with $(x_F,y_F)=(0,0)$ are mapped to
\[
\tilde{\Sigma}_{in}=\big{\{}(-\rho,M(-\rho,u))\big| u\in\tilde{J}\big{\}}, \qquad \tilde{\Sigma}_{in}=\big{\{}(\rho,M(\rho,u))\big| u\in\tilde{J}\big{\}},
\]
which can in turn be rewritten as
\begin{equation}\label{K1Segs}
\tilde{\Sigma}_{in}=\big{\{}(-\rho,u)\big| u\in\hat{J}_{in}\big{\}}, \qquad
\tilde{\Sigma}_{in}=\big{\{}(\rho,u)\big| u\in\hat{J}_{out}\big{\}},
\end{equation}
since $M(\rho,u)=\mathcal{O}(u)$. One can verify that the presence of a jump-off point at $(x,y)=(0,0)$ implies the presence of a jump-off point at $(x,u)=(0,0)$ in the system \eqref{NGTilde1}.\footnote{Since $D_yf|_F>0$, the coordinate transformation preserves orientation.} We consider the extended system obtained from \eqref{NGTilde1} by adding the trivial equation $\epsilon'=0$ and expanding about $(x,u,\epsilon)=(0,0,0)$,
\begin{equation}\label{ExpReg}
\begin{pmatrix}
{x'} \\
{u'} \\
{\epsilon'}
\end{pmatrix}
=
\begin{pmatrix}
{a_0+a_1x+a_2u+\cdots} \\
{b_1x+b_2u+\cdots} \\
{0}
\end{pmatrix}
u+\epsilon
\begin{pmatrix}
{c_0+c_1x+c_2u+c_3\epsilon+\cdots} \\
{d_0+d_1x+d_2u+d_3\epsilon +\cdots} \\
{0}
\end{pmatrix},
\end{equation}
where $a_0\neq0$, $b_1\neq0$, $d_0\neq0$ by the definition of regular contact (Definition \ref{RCPDef}). We assume in the following that $a_0$, $b_0$, $d_0$ are positive; the proof is analogous for different choices provided the relative orientations are consistent with $F$ being a jump-off point. We now define the blow-up $\Phi:S^2\times[0,\rho]\to\mathbb{R}^3$ by the mapping
\begin{equation}\label{BUMap}
(x,u,\epsilon)=\big{(}r\bar{x},r^2\bar{u},r^3\bar{\epsilon}\big{)}, \qquad \big{(}\bar{x},\bar{u},\bar{\epsilon},r\big{)}\in S^2\times[0,\rho],
\end{equation}
where in particular, since ${\epsilon}\geq0$, we need only consider the dynamics on and near the hemisphere
\[
(\bar{x},\bar{u},\bar{\epsilon})\in S^2_+=\{(\bar{x},\bar{u},\bar{\epsilon})|\bar{x}^2+\bar{u}^2+\bar{\epsilon}^2=1,\bar{\epsilon}\geq0\}.
\]
The dynamics are studied in different coordinate charts. We define entry and exit charts $K_1:\bar{x}=-1$ respectively $K_3:\bar{x}=1$, and a family rescaling chart $K_2:\bar{\epsilon}=1$. In line with conventions, we denote the image of an object $\gamma$ under the blow-up map \eqref{BUMap} by $\bar{\gamma}$, and \textcolor{black}{its} image in a specific chart $K_j$ by $\gamma_j$, for $j\in\{1,2,3\}$. Charts $K_1$ and $K_3$ allow one to describe the extension of the (extended) slow manifolds $\bar{S^a_\epsilon}$ and $\bar{S^r_\epsilon}$ respectively into the neighbourhood of $S_+^2$, while the flow on (and near) the interior of $S_+^2$ is understood in chart $K_2$.

\

In order to put the technicalities into context, we briefly summarise the key findings in charts $K_1$ and $K_3$ before considering a more detailed analysis. The main dynamical features are sketched in Figure \ref{RegBU1Fig}, which shows a birds-eye perspective of the dynamics in the blow-up, restricted to the hemisphere $S_+^2$. We refer to Figure  \ref{RegBU1Fig} throughout the proof; for now it suffices to note the existence of four singularities on the equator $S^1$. Those denoted $p_a$ and $p_r$ are partially hyperbolic and attracting/repelling along $S^1$ respectively, and $q_{in}$ and $q_{out}$ are hyperbolic singularities corresponding to the intersection with the (extended) critical fiber $\mathcal{\bar{F}}$. 

\begin{figure}[t]
\captionsetup{format=plain}
\hspace{0em}\centerline{\includegraphics[trim={0 6cm 0 6cm},scale=.4]{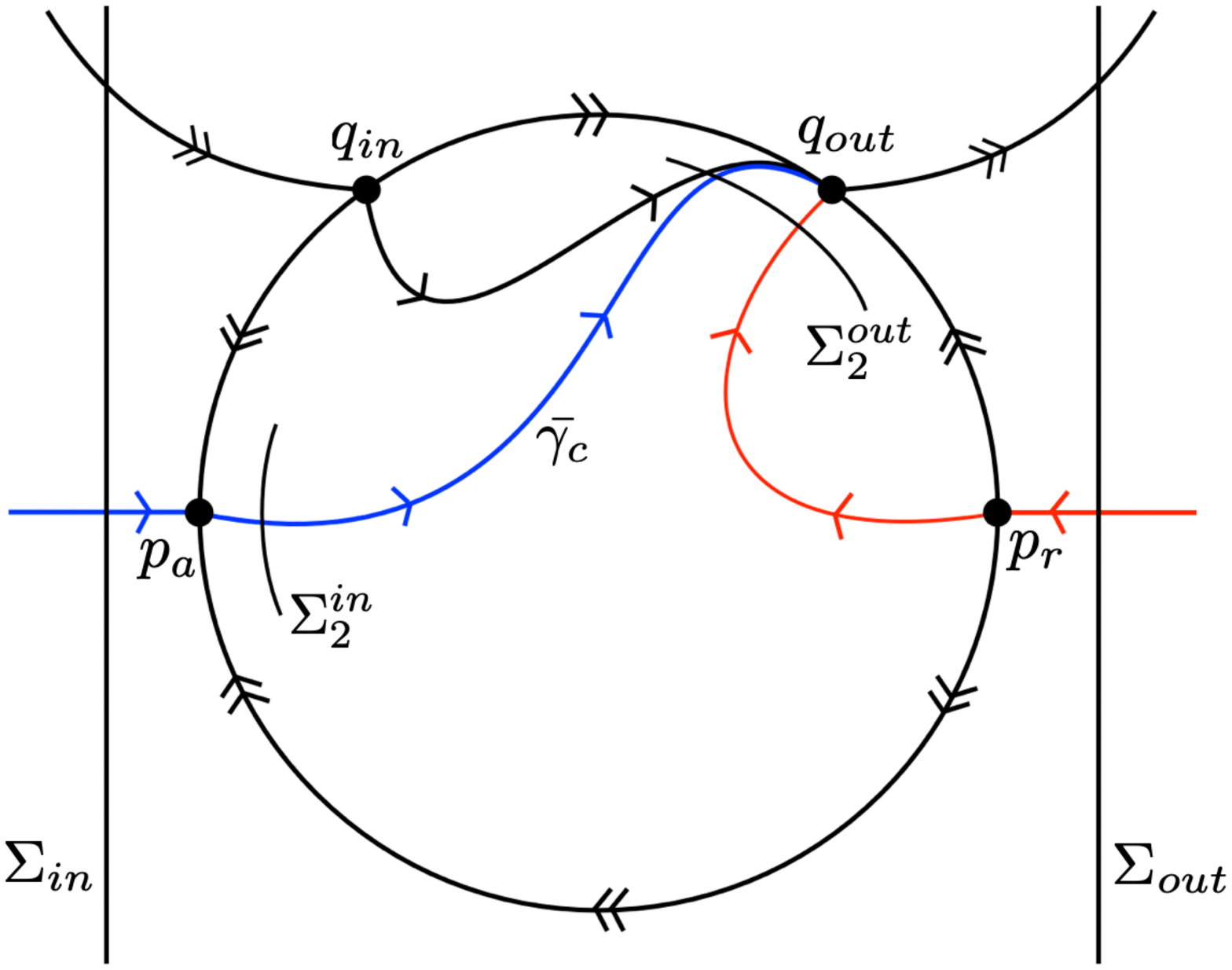}}
\caption{Dynamics on $S^{2,+}$.}\label{RegBU1Fig}
\end{figure}

\

We now present an analysis of the dynamics in charts $K_i$, $i\in\{1,2,3\}$ in turn. Most of the analysis is conceptually analogous to the work in \cite{Krupa2001a}, in which the authors study the dynamics of a regular fold point in standard form problems \eqref{OurSF}. We cite corresponding and relevant results contained in this work as the analysis proceeds.

\

\textbf{Chart $K_2$.} After an additional desingularisation $dt=r_2^{-1}d\bar{t}_2$, we obtain the following in the family rescaling chart $K_2$:
\begin{align}\label{K2Pert}
\begin{array}{lcl}
x_2'=a_0u_2+\mathcal{O}(r_2), \\
u_2'=d_0+b_1x_2u_2+\mathcal{O}(r_2), \\
r_2'=0,
\end{array}
\end{align}
where the subscript notation indicates the use of chart-specific coordinates
\[
(x,u,\epsilon)=(r_2x_2,r_2^2u_2,r_2^3).
\]
System \eqref{K2Pert} is a regular perturbation problem with perturbation parameter $r_2$, and so the dynamics are governed by the limiting system
\begin{align}\label{K2L}
\begin{array}{lcl}
x_2'=a_0u_2, \\
u_2'=d_0+b_1x_2u_2.
\end{array}
\end{align}

In the following we define the segments
\[
\Sigma_2^{in}=\big{\{}(x_2,u_2,r_2)\big|x=-\delta^{-1/3}, r_2\in[0,\delta^{1/3}\rho]\big{\}}, \qquad \Sigma_2^{out}=\big{\{}(x_2,u_2,r_2)\big|x=\delta^{-1/3}, r_2\in[0,\delta^{1/3}\rho]\big{\}},
\]
(see Figure \ref{RegBU1Fig}) and denote the transition map defined by the flow of \eqref{K2Pert} by $\Pi_2:\Sigma_2^{in}\to\Sigma_2^{out}$.

\begin{prop}\label{K2Prop} (cf. Proposition 2.3 and 2.4 in \cite{Krupa2001a}).
There exists a solution $\gamma_{c,2}$ for \eqref{K2L} such that for the initial condition $q\in\Sigma_2^{in}\cap\gamma_{c,2}$, we have the following:
\begin{itemize}
\item[(i)]
\[
\Pi_2(q)=\bigg(\delta^{-1/3},\bigg(\frac{b_1}{2a_0}\bigg)\delta^{-2/3}+\bigg(\frac{2d_0^2}{a_0b_1}\bigg)^{1/3}\Omega_0-\bigg(\frac{2d_0}{b_1}\bigg)\delta^{1/3}+\mathcal{O}(\delta),0\bigg)^T,
\]
\item[(ii)] A neighbourhood of $q$ is mapped diffeomorphically onto a neighbourhood of $\Pi_2(q)$.
\end{itemize}
\end{prop}

\begin{proof}
The key observation is that \eqref{K2L} is a Ricatti equation with special solution $\gamma_{c,2}$ (see Figure \ref{RegBU1Fig}). In particular, $\gamma_{c,2}$ can be parameterised as $(x_2,\zeta(x_2))$, where $\zeta(x_2)$ has known asymptotic behaviour described by
\[
\zeta(x_2)=\bigg(\frac{2d_0^2}{a_0b_1}\bigg)^{1/3}\Omega_0+\bigg(\frac{b_1}{2a_0}\bigg) x_2^2-\bigg(\frac{2d_0}{b_1}\bigg)\frac{1}{x_2}+\mathcal{O}\bigg(\frac{1}{x_2^3}\bigg), \qquad x_2\to\infty,
\]
where $\Omega_0$ is a known positive constant,\footnote{$\Omega_0$ is the smallest positive solution to $J_{-1/3}(2z^{3/2}/3)+J_{1/3}(2z^{3/2}/3)$, where $J_{-1/3}$ and $J_{1/3}$ are Bessel functions of the first kind. See Remark 2.4. in \cite{Krupa2001a} and also \textcolor{black}{\cite{Mishchenko1975}}.} and
\begin{equation}\label{Ricatti2}
\zeta(x_2)=-\bigg(\frac{d_0}{b_1}\bigg)\frac{1}{x_2}+\mathcal{O}\bigg(\frac{1}{x_2^4}\bigg), \qquad x_2\to-\infty.
\end{equation}
The result (i) follows, and (ii) follows by standard results from regular perturbation theory.
\end{proof}

\

\textbf{Chart $K_1$.} After an additional desingularisation $dt=r_1^{-1}d\bar{t}_1$, we obtain the following in the entry chart $K_1$:
\begin{align}\label{K1Pert}
\begin{array}{lcl}
u_1'=-b_1u_1+d_0\epsilon_1+2u_1\theta_1(u_1,\epsilon_1,r_1)+r_1((b_3+b_2u_1)u_1-d_1\epsilon_1)+\mathcal{O}(r_1^2), \\
\epsilon_1'=3\epsilon_1\theta_1(u_1,\epsilon_1,r_1), \\
r_1'=-r_1\theta(u_1,\epsilon_1,r_1),
\end{array}
\end{align}
where the subscript notation indicates the use of chart-specific coordinates
\[
(x,u,\epsilon)=(-r_1,r_1^2u_1,r_1^3\epsilon_1),
\]
and
\[
\theta_1(u_1,\epsilon_1,r_1)=a_0u_1+r_1(-a_1u_1+c_0\epsilon_1)+\mathcal{O}(r_1^2).
\]
The portion of the equator $S^1$ visible in chart $K_1$ can be identified with the invariant line $\{(u_1,0,0)|u_1\in\mathbb{R}\}$, along which we identify the partially hyperbolic singularity $p_a=(0,0,0)$, and hyperbolic saddle $q_{in}=(b_1/2a_0,0,0)$ shown in Figure \ref{RegBU1Fig}. Our main task here is to understand the dynamics near $p_a$, so we restrict our analysis to the set
\[
D_1:=\big{\{}(u_1,\epsilon_1,r_1)\big|(\epsilon_1,r_1)\in[0,\delta]\times[0,\rho] \big{\}}.
\]
The situation is sketched in Figure \ref{K1RCPFig}, which shows all the relevant objects in the analysis near $p_a$.

\begin{figure}[t]
\captionsetup{format=plain}
\hspace{0em}\centerline{\includegraphics[trim={0 1.5cm 0 1.5cm},scale=.4]{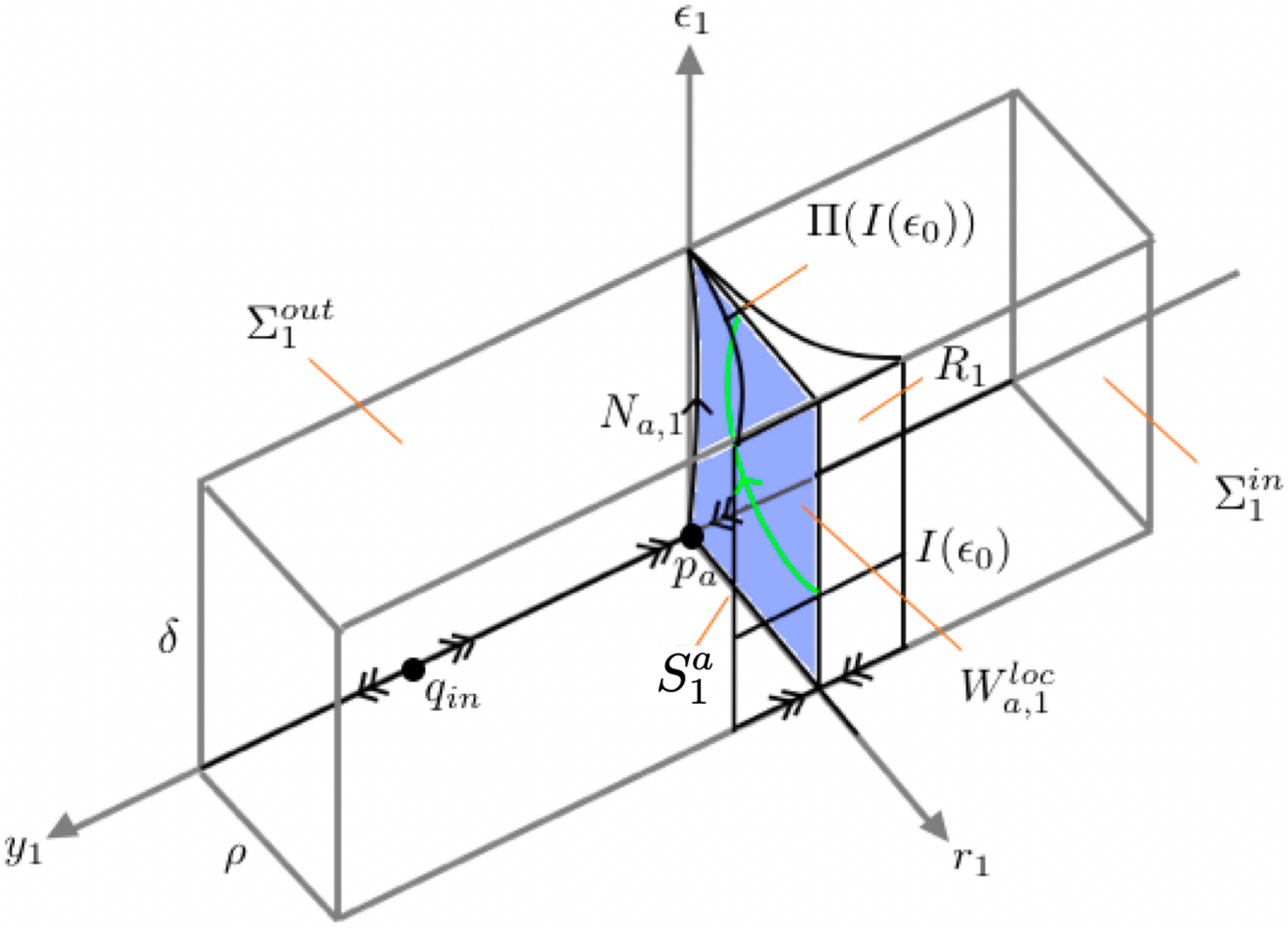}}
\caption{Setup for the local analysis near $p_a$.}\label{K1RCPFig}
\end{figure}

\begin{prop}\label{K1Prop1} (cf. Proposition 2.6 in \cite{Krupa2001a}).
Given $\delta, \rho$ sufficiently small, we have the following:
\begin{itemize}
\item[(i)] There exists a locally invariant, two-dimensional attracting centre manifold $W_{a,1}^{loc}$, which is tangent to
\[
E^c=span\{(0,0,1)^T,(d_0,b_1,0)^T\}
\]
at $p_a$ and given by a graph $u_1=g(\epsilon_1,r_1)$.
\item[(ii)] $W_{a,1}^{loc}$ contains two locally invariant, one-dimensional centre manifolds as restrictions
\[
N^a_{1}=W_{a,1}^{loc}\big|_{r_1=0} \qquad \text{and} \qquad S_1^a=W_{a,1}^{loc}\big|_{\epsilon_1=0}.
\]
The manifold $N^a_{1}$ is unique in $D_1$ and \textcolor{black}{coincides} with the image of the special trajectory $\gamma_c$ in chart $K_1$, denoted $\gamma_{c,1}$. The manifold $S^a_{1}$ \textcolor{black}{coincides} with image of the attracting (extended) critical manifold $\bar{S^a}$ in chart $K_1$.
\item[(iii)] There exists a stable invariant foliatation $\mathcal{F}^s$ with base $W_{a,1}^{loc}$ and one-dimensional fibers. Moreover, we have that $\forall c>-b_1$, $\exists \delta,\rho>0$ such that the contraction rate along $\mathcal{F}^s$ during the time interval $[0,T]$ is greater than $e^{cT}$.
\end{itemize}
\end{prop}

\begin{proof}
We prove statements (i) and (ii). Statement (iii) follows from centre manifold theory (see, e.g, \textcolor{black}{\cite{Guckenheimer1983}}). The Jacobian at $p_a$ is given by
\[
J_{p_a}=
\begin{pmatrix}
{-b_1} & {d_0} & {0} \\
{0} & {0} & {0} \\
{0} & {0} & {0}
\end{pmatrix},
\]
which has eigenvalues $-b_1<0$, $0$ and $0$. Centre manifold theory implies the existence of a locally invariant two-dimensional attracting centre manifold $W^{loc}_{a,1}$ given as a graph $u_1=g(\epsilon_1,r_1)$, which is tangent to the centre eigenspace $E^c=span\{(d_0,b_1,0)^T,(0,0,1)^T\}$ at $p_a$. Making a power series ansatz for $g(\epsilon_1,r_1)$ and matching terms gives a local expression for the manifold:
\[
W^{loc}_{a,1}=\bigg{\{}(u_1,\epsilon_1,r_1)\bigg|u_1=g(\epsilon_1,r_1)=\frac{d_0}{b_1}\epsilon_1+\bigg(\frac{b_3d_0-b_1d_1}{b_1^2}\bigg)r_1\epsilon_1-\frac{a_0d_0^2}{b_1^3}\epsilon_1^2+\mathcal{O}(3)\bigg{\}},
\]
where $\mathcal{O}(3):=\mathcal{O}(\epsilon_1^3,\epsilon_1^2r_1,\epsilon_1r_1^2,r_1^3)$. Dynamics on $W_{a,1}^{loc}$ are determined by restricting \eqref{K1Pert}:
\begin{align}\label{K1W}
\begin{array}{lcl}
u_1'=\frac{3a_0d_0^2}{b_1^2}\epsilon_1^2+\mathcal{O}(3), \\
\epsilon_1'=\frac{3a_0d_0}{b_1}\epsilon_1^2+\mathcal{O}(3), \\
r_1'=-\frac{a_0d_0}{b_1}r_1\epsilon_1+\mathcal{O}(3). \\
\end{array}
\end{align}
Restricting to the invariant plane $\{r_1=0\}$ gives
\[
N^a_{1}=\bigg{\{}(u_1,\epsilon_1,0)\bigg|u_1=g(\epsilon_1,0)=\frac{d_0}{b_1}\epsilon_1-\frac{a_0d_0^2}{b_1^3}\epsilon_1^2+\mathcal{O}(\epsilon_1^3)\bigg{\}},
\]
and restricting to the invariant plane $\{\epsilon_1=0\}$ gives $S^a_1=\{(u_1,0,r_1)|u_1=\mathcal{O}(r_1^3)\}$. One can verify that $N^a_{1}$ and $\gamma_{c,1}$ agree near $p_a$ by using equation \eqref{Ricatti2} and the form of the blow-up transformation \eqref{BUMap} to obtain
\[
\gamma_{c,1}=\big(\epsilon_1^{2/3}\zeta(-\epsilon_1^{-1/3}),\epsilon_1,0\big)=\bigg(\frac{d_0}{b_1}\epsilon_1+\mathcal{O}(\epsilon_1^2),\epsilon_1,0\bigg),
\]
in the limit $\epsilon_1\to0$. Uniqueness of $N^a_{1}$ follows from the equation $\epsilon_1'|_{N_{a,1}}>0$.
\end{proof}

Entry and exit segments $\Sigma_1^{in/out}$ shown in Figure \ref{K1RCPFig} are defined by
\[
\Sigma_1^{in}=\{(u_1,\epsilon_1,r_1)|r_1=\rho\}, \qquad \Sigma_1^{out}=\{(u_1,\epsilon_1,r_1)|\epsilon_1=\delta\}.
\]
We also define the rectangle $R_1=\{|u_1|\leq\omega\}\cap\Sigma_1^{in}$ for some sufficiently small $\omega>0$, and let $I(\epsilon_0)=R_1|_{\epsilon_1=\epsilon_0}$ for each fixed $\epsilon_0\in[0,\delta]$ (see Figure \ref{K1RCPFig}). The following result summarises the dynamics near $p_a$.

\begin{prop}\label{K1Prop2} (cf. Proposition 2.8 in \cite{Krupa2001a}).
Given $\delta,\rho,\omega$ sufficiently small, the transition map $\Pi_{1}:\Sigma_1^{in}\to\Sigma_1^{out}$ has the following properties:
\begin{itemize}
\item[(i)] $\Pi_{1}(R_1)$ is wedge-shaped in $\Sigma_1^{out}$.
\item[(ii)] $\forall\epsilon_1\in(0,\delta]$ and $c<b_1$ fixed, $\exists K>0$ such that $\Pi_{1}|_{I(\epsilon_1)}$ is a contraction with contraction rate bounded below by
\[
K\exp \bigg[-\frac{cb_1}{3a_0d_0}\bigg(\frac{1}{\epsilon_1}-\frac{1}{\delta}\bigg)\bigg].
\]
\end{itemize}
\end{prop}

\begin{proof}
We need an expression for the transition time $T$ taken for solutions to pass from $\Sigma_1^{in}$ to $\Sigma_1^{out}$ (cf. Lemma 2.7 in \cite{Krupa2001a}). Direct integration of the equation for $\epsilon_1'$ in \eqref{K1W} gives the leading order estimate
\begin{equation}\label{TransTime}
T\sim\frac{b_1}{3a_0d_0}\bigg(\frac{1}{\epsilon_1}-\frac{1}{\delta}\bigg).
\end{equation}
Statements (i) and (ii) follow from the expression \eqref{TransTime} and Proposition \ref{K1Prop1}.
\end{proof}

\

\textbf{Chart $K_3$.} After an additional desingularisation $dt=r_3^{-1}d\bar{t}_3$, we obtain the following in exit chart $K_3$:
\begin{align}\label{K3Pert}
\begin{array}{lcl}
u_3'=b_1u_3+d_0\epsilon_3-2u_3\theta_3(u_3,\epsilon_3,r_3)+r_3((b_3+b_2u_3)u_3+d_1\epsilon_3)+\mathcal{O}(r_3^2), \\
\epsilon_3'=-3\epsilon_3\theta_3(u_3,\epsilon_3,r_3), \\
r_3'=r_3\theta_3(u_3,\epsilon_3,r_3),
\end{array}
\end{align}
where the subscript notation indicates the use of chart-specific coordinates
\[
(x,u,\epsilon)=(r_3,r_3^2u_3,r_3^3\epsilon_3),
\]
and
\[
\theta_3(u_3,\epsilon_3,r_3)=a_0u_3+r_3(a_1u_3+c_0\epsilon_3)+\mathcal{O}(r_3^2).
\]
The portion of the equator $S^1$ visible in chart $K_3$ can be identified with the invariant line $\{(u_3,0,0)|u_3\in\mathbb{R}\}$, along which we identify the partially hyperbolic singularity $p_r=(0,0,0)$, and hyperbolic saddle $q_{out}=(b_1/2a_0,0,0)$ shown in Figure \ref{RegBU1Fig}. We are interested in the manner by which solutions leave the neighbourhood of $S^2_+$ near $q_{out}$. Accordingly, we restrict to the set
\[
D_3:=\big{\{}(u_3,\epsilon_3,r_3)\big|(\epsilon_3,r_3)\in[0,\delta]\times[0,\rho] \big{\}},
\]
and define the sections
\[
\Sigma_3^{in}=\big{\{}(u_3,r_3,\epsilon_3)\big|r_3\in[0,\rho], \epsilon_3=\delta\big{\}}, \qquad \Sigma_3^{out}=\big{\{}(u_3,r_3,\epsilon_3)\big|r_3=\rho, \epsilon_3\in[0,\delta]\big{\}}.
\]

\begin{prop}\label{K3Prop} (cf. Propostion 2.11 in \cite{Krupa2001a}). The transition map $\Pi_3:\Sigma_3^{in}\to\Sigma_3^{out}$ has form
\[
\Pi_3(u_3,\delta,r_3)=\bigg(\Pi_{31}(u_3,\delta,r_3),\delta\bigg(\frac{r_3}{\rho}\bigg)^3,\rho\bigg)^T,
\]
where
\begin{equation}\label{K3TransU}
\Pi_{31}(u_3,\delta,r_3)\sim\frac{b_1}{2a_0}+\bigg(u_3-\frac{b_1}{2a_0}\bigg)\bigg(\frac{r_3}{\rho}\bigg)^2.
\end{equation}
In particular, if $q=\gamma_{c,3}\cap\Sigma_3^{in}$, then
\begin{equation}\label{K3TransG}
\Pi_{31}(q)\sim\frac{b_1}{2a_0}+\Omega_0\bigg(\frac{2d_0^2}{a_0b_1}\bigg)^{1/3}\bigg(\frac{r_3}{\rho}\bigg)^2\delta^{2/3}.
\end{equation}
\end{prop}

\begin{proof}
Although $q_{out}$ is hyperbolic, it has a resonance preventing a local transformation into the leading order linear system. An explicit solution for the leading order dynamics near $q_{out}$ can be obtained directly from \eqref{K3Pert}, though. Since $\theta_3(u_3,\epsilon_3,r_3)\sim b_1/2>0$ near $q_{out}$ we can rescale $d\bar{t}_3=(\theta_3(u_3,\epsilon_3,r_3))^{-1}d\tilde{t}$, obtaining
\begin{align}\label{K3Pert2}
\begin{array}{lcl}
u_3'=\frac{b_1u_3+d_0\epsilon_3}{\theta_3(u_3,\epsilon_3,r_3)}-2u_3, \\
\epsilon_3'=-3\epsilon_3, \\
r_3'=r_3,
\end{array}
\end{align}
where the dash notation now denotes differentiation with respect to $\tilde t_3$. The time $T$ taken for solutions of \eqref{K3Pert2} to travel from $\Sigma_{3}^{in}$ to $\Sigma_3^{out}$ can be determined from the equation for $\epsilon_3'$. Direct integration gives
\[
T=\ln \bigg(\frac{\rho}{r_3}\bigg),
\]
and the equation for $\Pi_{31}(u_3,\delta,r_3)$ is determined by expanding the right hand side of the expression for $u_3'$ near $q_{out}$, which gives
\[
u_3'\sim\frac{b_1}{a_0}-2u_3.
\]
Integrating and evaluating at $T$ yields equation \eqref{K3TransU}. Finally, expression \eqref{K3TransG} follows from the form for the image of $\gamma_c$ in chart $K_3$, given by
\[
\gamma_{c,3}=\big(\epsilon_3^{2/3}\zeta(\epsilon_3^{-1/3}),\epsilon_3,0\big)=\bigg(\frac{b_1}{2a_0}+\Omega_0\bigg(\frac{2d_0^2}{a_0b_1}\bigg)^{1/3}\epsilon_3^{2/3}+\mathcal{O}(\epsilon_3),\epsilon_3,0\bigg).
\]
\end{proof}

\begin{rem}
One can apply center manifold theory to prove the existence of a locally invariant, two-dimensional repelling center manifold $W_{r,3}^{loc}$ tangent to the center subspace at $p_r$, and appeal to similar arguments to those provided in the preceding proof to understand the leading order dynamics near $q_{in}$. We omit the details, since the dynamics near $p_r$ and $q_{in}$ are not of primary importance here.
\end{rem}

\

{\bf Proof of Theorem \ref{RCPThm}}. The proof follows conceptually the same arguments as those given in Section 2.8 of \cite{Krupa2001a}. We denote the transition maps between charts $K_i$ and $K_j$ by $\kappa_{ij}$, and note that smoothness of the maps $\kappa_{ij}$ follows from the fact that $S^2$ is a manifold.

\begin{proof}
We need to track $W_{a,1}^{loc}$ and $R_1$ under the flow. Define a map $\Pi:\Sigma_1^{in}\to\Sigma_3^{out}$ by the composition
\[
\Pi=\Pi_3\circ\kappa_{23}\circ\Pi_2\circ\kappa_{12}\circ\Pi_1.
\]
The map $\pi:\Sigma_{in}\to\Sigma_{out}$ from the statement of Theorem \ref{RCPThm} is given by blow-down $\pi=\Phi\circ\Pi\circ\Phi^{-1}$, for $\epsilon>0$.

By Proposition \ref{K1Prop2} and smoothness of the transition maps $\kappa_{ij}$, the image $\kappa_{12}\circ\Pi_1(R_1\cap W_{a,1}^{loc})$ is a smooth curve in $\Sigma_2^{in}$, transverse to $\{r_2=0\}$. Since $\kappa_{12}\circ\Pi_1(W_{a,1}^{loc}\cap\{r_1=0\})=\gamma_{c,2}\cap \Sigma_2^{in}$, it follows from Proposition \ref{K2Prop} that $\Pi_2\circ\kappa_{12}\circ\Pi_1(W_{a,1}^{loc}\cap R_1)$ is a smooth curve in $\Sigma_2^{out}$ of form
\[
\{(\delta^{-1/3},s_2^{out}(r_2),r_2)|r_2\in[0,\delta^{1/3}\rho]\},
\]
where $s_2^{out}:[0,\delta^{1/3}\rho]\to\mathbb{R}$ is a smooth function such that $\{s_2^{out}(0)\}=\gamma_{c,2}\cap\Sigma_2^{out}$. In particular, this curve is transverse to $\{r_2=0\}$, and so the curve $\kappa_{23}\circ\Pi_2\circ\kappa_{12}\circ\Pi_1(W_{a,1}^{loc}\cap R_1)$ is transverse to $\{r_3=0\}$ in $\Sigma_3^{in}$. Proposition \ref{K3Prop} implies that the image of the curve under $\Pi_3$ is of form
\[
\{(s_3^{out}(\epsilon_3),\epsilon_3,\rho)|\epsilon_3\in[0,\delta]\},
\]
with
\[
s_3^{out}(\epsilon_3)=\frac{b_1}{2a_0}+\mathcal{O}(\epsilon_3^{2/3}).
\]
Hence after applying the blow-down transformation we obtain
\[
\vert u_{s}-u_{l}\vert=\mathcal{O}(\epsilon^{2/3}) \qquad \implies \qquad  \vert M(\rho,u_s)-M(\rho,u_l)\vert=\vert y_s-y_l\vert=\mathcal{O}(\epsilon^{2/3}),
\]
where $u_s$, $u_l$ denote the $u$-coordinate of the intersections $S^a_\epsilon\cap\Sigma_{out}$, $\mathcal{F}\cap\Sigma_{out}$ respectively, and we have used the fact that $M(\rho,u)=\mathcal{O}(u)$. This proves assertion (1) in Theorem \ref{RCPThm}. Assertion (2) follows exactly as in \cite{Krupa2001a}, only with Propositions \ref{K1Prop2} and \ref{K3Prop} in place of Propositions 2.8 and 2.11 respectively in the cited work.
\end{proof}

\begin{rem}
Strictly speaking, one must also verify that the flow is regular across $\{(\bar{x},\bar{u},\bar{\epsilon})|\bar{x}=0,\bar{\epsilon}\geq0\}\times[0,\rho]$. This can be done in additional charts $\bar{u}=\pm1$.
\end{rem}

\begin{figure}[t]
\captionsetup{format=plain}
\hspace{0em}\centerline{\includegraphics[trim={0 2cm 0 2cm},scale=.4]{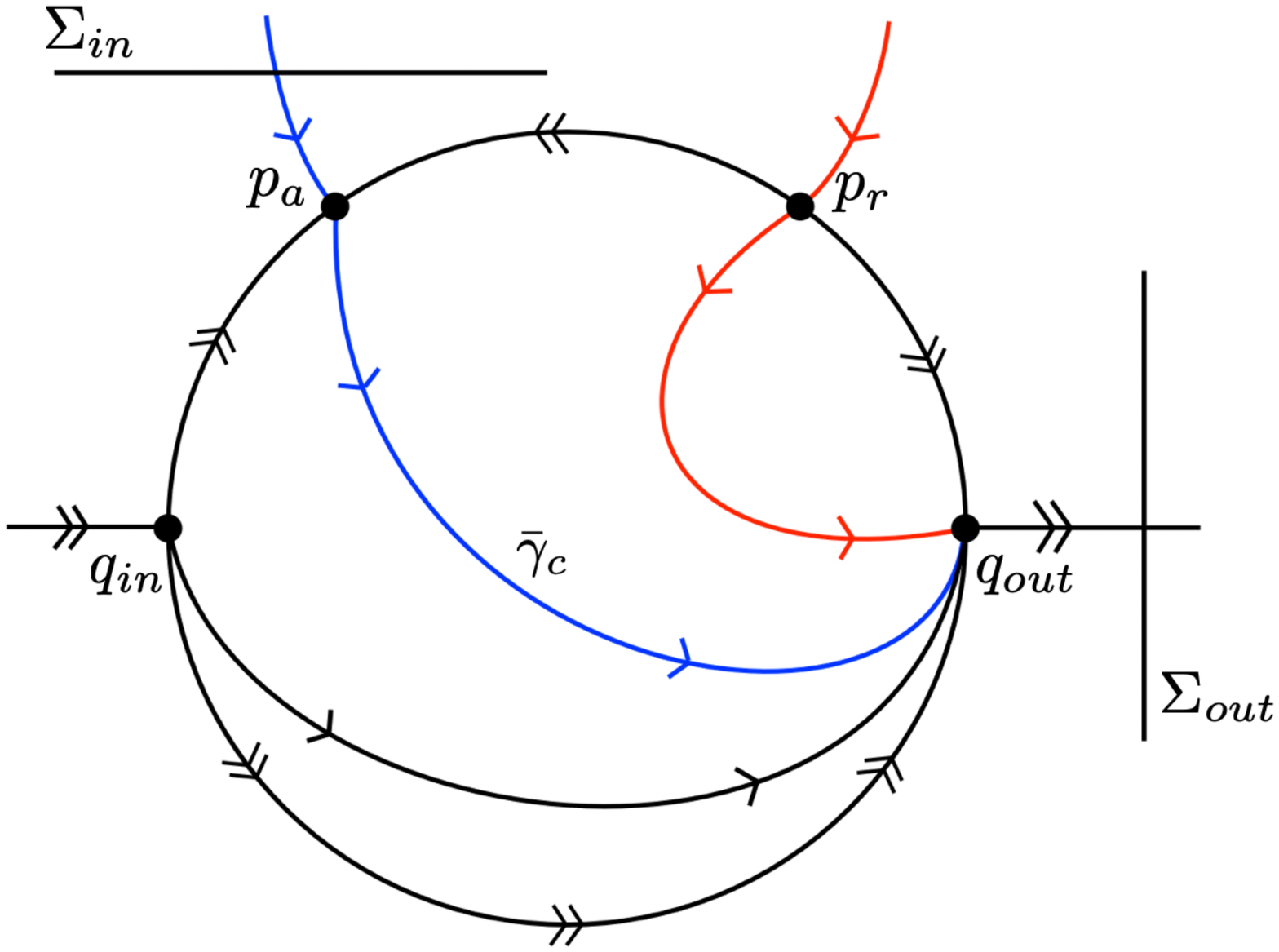}}
\caption{Regular fold point dynamics on $S^2_+$, as studied in \cite{Krupa2001a}. Compare with Figure \ref{RegBU1Fig}.}\label{JumpCompFig}
\end{figure}

\begin{rem}
In \cite{Krupa2001a}, the authors identify a Ricatti equation in chart $K_2$ of the form
\begin{align}\label{KSRicatti}
\begin{array}{lcl}
\tilde{x}_2'=\tilde{x}_2^2-\tilde{y}_2, \\
\tilde{y}_2'=-1,
\end{array}
\end{align}
in the limit $\tilde{r}_2\to0$, where we have used the tilde notation to distinguish their coordinates from ours. This system is related to the Ricatti equation in \eqref{K2L} by the coordinate transformation
\[
(\tilde{x}_2,\tilde{u}_2)=(\tilde{x}_2,\tilde{x}_2^2-\tilde{y}_2)
\]
followed by
\[
\big(x_2,u_2,t_2\big)=(\alpha^{-1}\tilde{x}_2,\beta^{-1}\tilde{u}_2,\gamma^{-1}\tilde{t}_2),
\]
where
\[
\alpha=\bigg(\frac{b_1^2}{4a_0d_0}\bigg)^{1/3}, \qquad
\beta=\bigg(\frac{a_0b_1}{2d_0^2}\bigg)^{1/3}, \qquad
\gamma=\bigg(\frac{a_0b_1d_0}{2}\bigg)^{1/3}.
\]
The dynamics in the blow-up observed in \cite{Krupa2001a} are sketched in \ref{JumpCompFig}, which should be compared with Figure \ref{RegBU1Fig}. Note that the segments chosen for the statement and proof of Theorem \ref{RCPThm} are not the same as those in \cite{Krupa2001a}.
\end{rem}

\end{appendix}

\end{document}